\documentclass[11pt,a4paper]{article}
\usepackage{graphicx}
\usepackage{mathrsfs}
\usepackage{color}
\usepackage[latin1]{inputenc}
\usepackage[english]{babel}

\usepackage{libertine}
\usepackage[T1]{fontenc}
\usepackage{lipsum}
\usepackage{titlesec}
\usepackage[colorlinks=true,linkcolor=magenta,citecolor=blue]{hyperref}

\usepackage[cmex10]{amsmath}
\usepackage{stmaryrd}
\usepackage{color}
\usepackage{pstricks}
\usepackage{cmll}
\usepackage{amsmath,amssymb}
\usepackage[sloped]{fourier}
\usepackage{amsthm}
\usepackage{esint}
\usepackage{psfrag}
\author{S\'ebastien Alvarez}
\date{}
\title{Gibbs $u$-states for the foliated geodesic flow and transverse invariant measures}

\begin{document}

\newtheorem{maintheorem}{Theorem}
\newtheorem{maincoro}[maintheorem]{Corollary}
\renewcommand{\themaintheorem}{\Alph{maintheorem}}

\theoremstyle{plain}
\newtheorem{theorem}{Theorem}[section]
\newtheorem{defi}[theorem]{Definition}
\newtheorem{lemma}[theorem]{Lemma}
\newtheorem{proposition}[theorem]{Proposition}
\newtheorem{coro}[theorem]{Corollary}

\theoremstyle{definition}

\newtheorem{conj}{Conjecture}[section]
\newtheorem{exmp}{Example}[section]
\newtheorem{rem}{Remark}[section]

\theoremstyle{remark}

\newtheorem{note}{Note}[section]

\def\ep{\noindent{\hfill $\fbox{\,}$}\medskip\newline}
\renewcommand{\theequation}{\arabic{section}.\arabic{equation}}
\renewcommand{\thetheorem}{\arabic{section}.\arabic{theorem}}
\newcommand{\eps}{\varepsilon}
\newcommand{\disp}[1]{\displaystyle{\mathstrut#1}}
\newcommand{\fra}[2]{\displaystyle\frac{\mathstrut#1}{\mathstrut#2}}
\newcommand{\dif}{{\rm Diff}}
\newcommand{\homeo}{{\rm Homeo}}
\newcommand{\Per}{{\rm Per}}
\newcommand{\Fix}{{\rm Fix}}
\newcommand{\A}{\mathcal A}
\newcommand{\Z}{\mathbb Z}
\newcommand{\Q}{\mathbb Q}
\newcommand{\R}{\mathbb R}
\newcommand{\C}{\mathbb C}
\newcommand{\N}{\mathbb N}
\newcommand{\T}{\mathbb T}
\newcommand{\U}{\mathbb U}
\newcommand{\D}{\mathbb D}
\newcommand{\PP}{\mathbb P}
\newcommand{\Sp}{\mathbb S}
\newcommand{\K}{\mathbb K}
\newcommand{\car}{\mathbf 1}
\newcommand{\g}{\mathfrak g}
\newcommand{\gs}{\mathfrak s}
\newcommand{\hh}{\mathfrak h}
\newcommand{\rr}{\mathfrak r}
\newcommand{\s}{\sigma}
\newcommand{\fhi}{\varphi}
\newcommand{\ffhi}{\tilde{\varphi}}
\newcommand{\moins}{\setminus}
\newcommand{\ds}{\subset}
\newcommand{\W}{\mathcal W}
\newcommand{\WW}{\widetilde{W}}
\newcommand{\F}{\mathcal F}
\newcommand{\UF}{\mathcal U \mathcal F}
\newcommand{\G}{\mathcal G}
\newcommand{\CC}{\mathcal C}
\newcommand{\RR}{\mathcal R}
\newcommand{\DD}{\mathcal D}
\newcommand{\M}{\mathcal M}
\newcommand{\B}{\mathcal B}
\newcommand{\cS}{\mathcal S}
\newcommand{\HH}{\mathcal H}
\newcommand{\Hyp}{\mathbb H}
\newcommand{\UU}{\mathcal U}
\newcommand{\Pp}{\mathcal P}
\newcommand{\QQ}{\mathcal Q}
\newcommand{\E}{\mathcal E}
\newcommand{\GG}{\Gamma}
\newcommand{\LL}{\mathcal L}
\newcommand{\cN}{\mathcal N}
\newcommand{\KK}{\mathcal K}
\newcommand{\TT}{\mathcal T}
\newcommand{\X}{\mathcal X}
\newcommand{\Y}{\mathcal Y}
\newcommand{\ZZ}{\mathcal Z}
\newcommand{\bE}{\overline{E}}
\newcommand{\bF}{\overline{F}}
\newcommand{\wF}{\widetilde{F}}
\newcommand{\hcF}{\widehat{\mathcal F}}
\newcommand{\hM}{\widehat{M}}
\newcommand{\hL}{\widehat{L}}
\newcommand{\hg}{\widehat{g}}
\newcommand{\bW}{\overline{W}}
\newcommand{\bcW}{\overline{\mathcal W}}
\newcommand{\tL}{\widetilde{L}}
\newcommand{\diam}{{\rm diam}}
\newcommand{\diag}{{\rm diag}}
\newcommand{\Jac}{\mathcal J}
\newcommand{\Plong}{{\rm Plong}}
\newcommand{\Tr}{{\rm Tr}}
\newcommand{\Conv}{{\rm Conv}}
\newcommand{\Ext}{{\rm Ext}}
\newcommand{\Spec}{{\rm Sp}}
\newcommand{\Isom}{{\rm Isom}\,}
\newcommand{\Supp}{{\rm Supp}\,}
\newcommand{\Grass}{{\rm Grass}}
\newcommand{\Hold}{{\rm H\ddot{o}ld}}
\newcommand{\Ad}{{\rm Ad}}
\newcommand{\ad}{{\rm ad}}
\newcommand{\Aut}{{\rm Aut}}
\newcommand{\End}{{\rm End}}
\newcommand{\Leb}{{\rm Leb}}
\newcommand{\Liouv}{{\rm Liouv}}
\newcommand{\Int}{{\rm Int}}
\newcommand{\cc}{{\rm cc}}
\newcommand{\grad}{{\rm grad}}
\newcommand{\vol}{{\rm vol}}
\newcommand{\proj}{{\rm proj}}
\newcommand{\mass}{{\rm mass}}
\newcommand{\dive}{{\rm div}}
\newcommand{\dom}{{\rm dom}}
\newcommand{\dist}{{\rm dist}}
\newcommand{\im}{{\rm Im}}
\newcommand{\re}{{\rm Re}}
\newcommand{\codim}{{\rm codim}}
\newcommand{\Map}{\longmapsto}
\newcommand{\vide}{\emptyset}
\newcommand{\tr}{\pitchfork}
\newcommand{\ssl}{\mathfrak{sl}}

\newenvironment{demo}{\noindent{\textbf{Proof.}}}{\quad \hfill $\square$}
\newenvironment{pdemo}{\noindent{\textbf{Proof of the proposition.}}}{\quad \hfill $\square$}
\newenvironment{IDdemo}{\noindent{\textbf{Idea of proof.}}}{\quad \hfill $\square$}

\newenvironment{remark}[1][Remark]{\begin{trivlist}
\item[\hskip \labelsep {\bfseries #1}]}{\end{trivlist}}

\def\to{\mathop{\rightarrow}}
\def\act{\mathop{\curvearrowright}}
\def\To{\mathop{\longrightarrow}}
\def\Sup{\mathop{\rm Sup}}
\def\Max{\mathop{\rm Max}}
\def\Inf{\mathop{\rm Inf}}
\def\Min{\mathop{\rm Min}}
\def\lims{\mathop{\overline{\rm lim}}}
\def\limi{\mathop{\underline{\rm lim}}}
\def\egal{\mathop{=}}
\def\dans{\mathop{\subset}}
\def\surj{\mathop{\twoheadrightarrow}}

\def\h#1#2#3{ {\mathfrak{h}}^{#1}_{#2\rightarrow#3}}

\maketitle

\begin{abstract}
This paper is devoted to the study of Gibbs $u$-states for the geodesic flow tangent  to a foliation $\F$ of a manifold $M$ having negatively curved leaves. By definition, they are  the probability measures on the unit tangent bundle to the foliation that are invariant under the foliated geodesic flow and have Lebesgue disintegration in the unstable  manifolds of this flow.

On the one hand we give sufficient conditions for the existence of transverse invariant measures. In particular we prove that when  the foliated geodesic flow  has a Gibbs $su$-state, i.e.  an invariant measure with Lebesgue disintegration both in the stable and unstable manifolds, then this measure has to be obtained by combining a transverse invariant measure and the Liouville measure on the leaves.

On the other hand we exhibit a  bijective correspondence between the set of Gibbs $u$-states and a set of probability measure on $M$ that we call $\phi^u$-harmonic. Such measures have Lebesgue disintegration in the leaves and their local densities have a very specific form: they possess an integral representation analogue to the Poisson representation of harmonic functions.
\end{abstract}

\section{Introduction}

This is the second of a series of three papers in which we study a notion of Gibbs measure for the geodesic flow tangent to the leaves of a closed foliated manifold with negatively curved leaves \cite{Al2,Al1}.

Let $\F$ be a foliation of a closed manifold $M$ (we say that $(M,\F)$ is a \emph{closed foliated manifold}) whose leaves $L$ are endowed with smooth (i.e. of class $C^{\infty}$)
 Riemannian metrics $g_L$ which vary continuously transversally in the smooth topology (we refer to the assignment $L\mapsto g_L$ as a \emph{leafwise metric}). The \emph{unit
  tangent bundle} of $\F$ is the set $\hM$ of unit vectors tangent to $\F$. An element of $\hM$ shall be denoted by $v$ or, when we want to specify its basepoint, by $(x,v)$ with
   $x\in M$ and $v\in T^1_x\F$, the unit sphere of the tangent space to $\F$ at $x$. The set $\hM$ is naturally a manifold endowed with a foliation $\hcF$ whose leaves are the unit tangent bundles of the leaves of $\F$ (see \S \ref{definitionunittgtbdle}). It also carries a flow $G_t$, the \emph{foliated geodesic flow}, which preserves the leaves of $\hcF$ and whose restriction to a leaf $T^1L$ is precisely its geodesic flow. We will be mostly interested in \emph{negatively curved leafwise metrics} i.e. in the case where all metrics $g_L$ have negative sectional curvature. In that case the foliated geodesic flow possesses a weak form of hyperbolicity that resembles the classical notion of \emph{partial hyperbolicity} and that we shall analyze in detail in \S \ref{foliatedhyperbolicity}. This new notion has been introduced by Bonatti, G{\'o}mez-Mont and Mart{\'i}nez in a recent preprint \cite{BGM}. They called it the \emph{foliated hyperbolicity}. This means that $\hcF$ admits two \emph{continuous subfoliations}, called \emph{stable and unstable foliations} and denoted by $\W^s$ and $\W^u$, which are $G_t$-invariant and respectively uniformly contracted and dilated by $G_t$. One would like to know the extent to which the classical results concerning partial hyperbolicity can be applied to that context.

\begin{figure}[hbtp]
\centering
\includegraphics[scale=0.6]{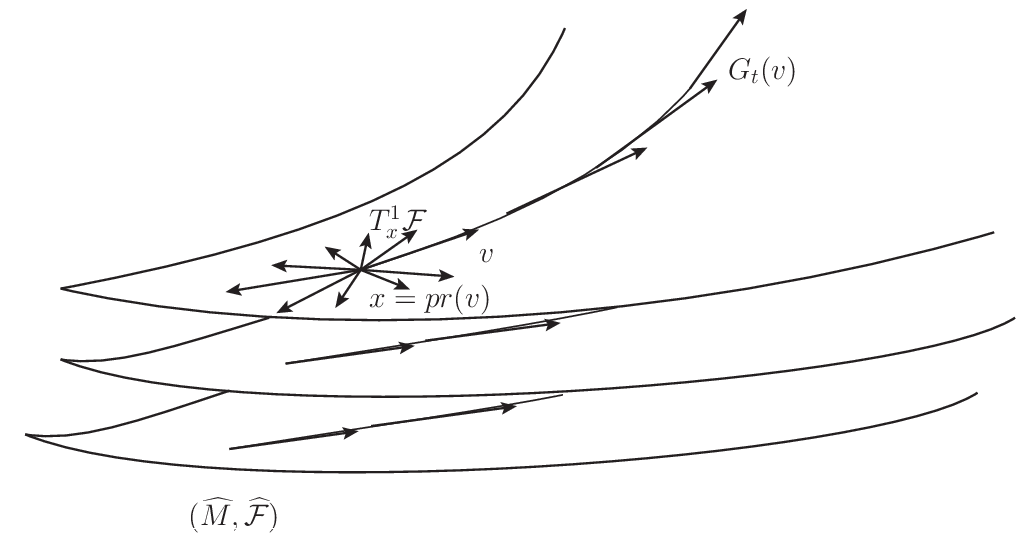}
\caption{Construction  of the unit tangent bundle and the foliated geodesic flow}
\label{obiwankenobi}
\end{figure}

The motivating problem of this paper is the research of \emph{SRB measures} (or physical measures) for $G_t$, i.e. of $G_t$-invariant measures $\mu$ whose basin (the set of  elements $v\in\hM$ such that the averages of Dirac masses along the orbit of $v$ tend to $\mu$ in the weak$^{\ast}$ sense) has positive Lebesgue measure. As for partially hyperbolic dynamics (see \cite{BV,BDV}) we expect that SRB measures will be given by \emph{Gibbs $u$-states} for $G_t$ i.e. invariant measures whose conditional measures along the leaves of  $\W^u$ are equivalent to the Lebesgue measure. Such invariant measures were first considered in the partially hyperbolic context by Pesin and Sinai in \cite{PS}.

Motivated by the work of Deroin-Klpetsyn \cite{DK} we wish to prove the following dichotomy, at least for codimension $1$ foliations (they deal with transversally conformal foliations) with negatively curved leaves.
\begin{itemize}
\item Either there exists a transverse invariant measure for $\F$;
\item or there exists a finite number of SRB measures for $G_t$ (which are precisely the only ergodic Gibbs $u$-states of $G_t$) whose basins cover a Borel subset of $\hM$ full for the Lebesgue measure. 
\end{itemize}

Establishing this dichotomy, \emph{when the leaves are hyperbolic}, is the purpose of Bonatti, G\'omez-Mont and Mart\'inez in their preprint \cite{BGM}, which strongly relies on the study of the statistical properties of the tangential Brownian motion performed in \cite{DK}. The dichotomy has also been established in the case where the foliation is transverse to a projective $\C\PP^1$-bundle over a closed manifold with negative sectional curvature: see \cite{Al2,BG,BGV} where the authors prove the uniqueness of the SRB measure in the absence of transverse invariant measure. Recently, in joint work with Yang (see \cite{AlY}) and using the main result of the present paper, we were able to prove the dichotomy for transversally conformal foliations with negatively curved leaves.

The purpose of this paper is twofold. Firstly, we wish to give new sufficient conditions for a foliation with negatively curved leaves to admit a transverse invariant measure that are decisive for establishing the aforementioned dichotomy (see \cite{AlY}). Secondly, this paper serves as a companion to \cite{Al1}. In the latter, we were led to associate a notion of Gibbs measure for the foliated geodesic flow to the now classical notion of Garnett's \emph{harmonic measure} (see \cite{Gar}). Here, we associate a new family of measures on a manifold $M$ foliated by negatively curved manifolds, called $\phi^u$\emph{-harmonic measure}, to the notion of Gibbs $u$-states already mentioned (see also Definition \ref{Gibbsustates}) and study some of their properties.

\paragraph{Existence of invariant measures.}

Plante was the first to give in \cite{Pl} a sufficient condition for the existence of a transverse invariant measure for a foliation $\F$ endowed with a leafwise metric. He proved that if a leaf has a subexponential growth, then its closure supports a transverse invariant measure. Later in \cite{GLW}, Ghys, Langevin and Walczak developed a notion of geometric entropy for foliations and proved that its vanishing implies the existence of a transverse invariant measure.

When the leaves of $\F$ are hyperbolic Riemann surfaces there is another condition ensuring the existence of a transverse invariant measure. In that case $\hM$ comes with three flows: the foliated geodesic flow $G_t$ and the foliated stable and unstable horocyclic flows $H_t^{\pm}$. The following folklore result may be found for example in \cite{BG} and \cite{Ma}.

\emph{Any probability measure on} $\hM$ \emph{which is invariant by the foliated geodesic flow and by both the foliated horocyclic flows is totally invariant, meaning that it is locally the product of the Liouville measure by a transverse holonomy invariant measure}.

The proof follows from an elegant algebraic argument. Since all the leaves are hyperbolic surfaces they are uniformized by the upper half plane $\Hyp$ endowed with the Poincar{\'e} metric. The unit tangent bundle $T^1\Hyp$ can be identified with the Lie group $PSL_2(\R)$. The geodesic and the two horocyclic flows are then identified with actions of $1$-parameter subgroups of $PSL_2(\R)$ by right translations. Moreover they generate all $PSL_2(\R)$. Hence a measure invariant by these three subgroups has to be invariant by the whole action of $PSL_2(\R)$ by right translations. Thus it is a multiple of the bi-invariant \emph{Haar measure} which is identified with the \emph{Liouville measure} of $T^1\Hyp$.

Consequently the conditional measures in the plaques of $\hcF$ of any measure on $\hM$ invariant by the joint action of the three foliated flows are proportional to the Liouville measure. Such a measure reads locally as the product of the Liouville measure by a transverse invariant measure.

We are interested in the case where the leaves of $\F$ are of arbitrary dimension and with negative sectional curvature. In this case there are no horocyclic flows on $\hM$ anymore, but two stable and unstable foliations which are continuous, subfoliate $\hcF$ and, as we recall, are denoted by $\W^s$ and $\W^u$. We can define a \emph{Gibbs su-state} as an invariant measure for the foliated geodesic flow $G_t$ which is both a Gibbs $s$-state and a Gibbs $u$-state. If one prefers it is a $G_t$-invariant probability measure whose conditional measures along the leaves of both $\W^s$ and $\W^u$ are equivalent to the Lebesgue measure. The main result is

\begin{maintheorem}
\label{mtheoremsugibbsgeodesique}
Let $(M,\F)$ be a closed foliated manifold endowed with a negatively curved leafwise metric. Suppose the foliated geodesic flow admits a Gibbs $su$-state $\mu$. Then $\mu$ is totally invariant i.e. is locally the product of the Liouville measures of leaves of $\hcF$ by a transverse invariant measure. Therefore $\F$ possesses a transverse invariant measure.
\end{maintheorem}

This theorem contains an implicit statement, namely that the existence of a transverse invariant measure for $\hcF$ implies the existence of a transverse invariant measure for $\F$. In fact these two assertions are equivalent as explained in Proposition \ref{equholinv}.

The Lie group action argument we gave above will be replaced by an argument of \emph{absolute continuity} of stable and unstable subfoliations, which will allow us to prove that all Gibbs $su$-states are totally invariant. The principle behind this result, which is the cornerstone of \cite{Al2}, is that prescribing both the measure classes of the conditional measures of a $G_t$-invariant measure in the stable and unstable manifolds is  very restrictive and implies the existence of a transverse invariant measure, which is a rare phenomenon.

Note that this theorem implies that foliated geodesic flows of most smooth foliations with negatively curved leaves don't preserve any smooth measure. Indeed, assume that $M$ is equipped with a smooth Riemannian metric whose restriction to the leaves has negative sectional curvature everywhere (see \cite{AlY} for a discussion on the existence of such metrics). Suppose the foliated geodesic flow $G_t$ preserves a smooth measure.
 Using the absolute continuity of the stable and unstable foliations (see Theorem \ref{absolutecontinuity}), the conditional measures in local stable and unstable manifolds of this measure are both equivalent to the Lebesgue measure. We deduce that such a measure is a Gibbs $su$-state and Theorem \ref{mtheoremsugibbsgeodesique} implies that there must exist a transverse invariant measure.

There is a relation with Walczak's work on dynamics of the foliated geodesic flow \cite{W}. He proved that when $(M,\F)$ is endowed with a Riemannian metric (the leaves are not supposed to be negatively curved), the foliated geodesic flow of $\hM$ preserves the Riemannian volume of $\hM$ if and only if $\F$ is transversally minimal (see \cite{W} for further details).

\begin{maincoro}
Let $(M,\F)$ be a closed foliated manifold endowed with a smooth Riemannian metric. Assume that for the induced leafwise metric, all leaves are negatively curved. Assume that the foliated geodesic flow of $T^1\F$ preserves a smooth measure. Then this measure is totally invariant.
\end{maincoro}

\paragraph{$\phi^u$-harmonic measures.} When the leaves of $\F$ are hyperbolic, the basepoint projection induces a bijective correspondence between Gibbs $u$-states (which are measures on $\hM$) and Garnett's harmonic measures (which are measures on $M$: see \cite{Gar} and Definition \ref{harmonicmeasure} for the definition). This has been proven in \cite{BMar,Ma} (for leaves of dimension $2$) and in \cite{Al1} (for leaves of higher dimension). As shown in \cite{Al1} the situation is more subtle when the curvature varies.

Our goal is to show the existence of a canonical bijective correspondence between Gibbs $u$-states 
(whose existence is stated in Theorem \ref{reconstruirelesetatsdeugibbs}) and
 a certain type of measure on $M$ with a special local form. Namely, this
 class consists of measures whose conditional measures in the plaques of $\F$
 have densities with respect to Lebesgue, these densities having an integral
  form reminescent of the \emph{Poisson representation} of harmonic functions of negatvely curved manifolds (see \cite{AS}). Let us be more precise.

Let $L$ be a leaf of $\F$, and $\widetilde{L}$ be its universal cover. It is a complete connected and simply connected Riemannian manifold whose sectional curvature is pinched between two negative constants. Therefore it can be compactified by adding a topological sphere $\widetilde{L}(\infty)$. This sphere is defined as the set of equivalence classes of geodesic rays for the relation ``stay at bounded distance''. Say a vector $v\in T^1\tL$ \emph{points to} $\xi\in\tL(\infty)$ if $\xi$ is the equivalence class of the geodesic ray it directs.

\begin{figure}[!h]
\centering
\includegraphics[scale=0.7]{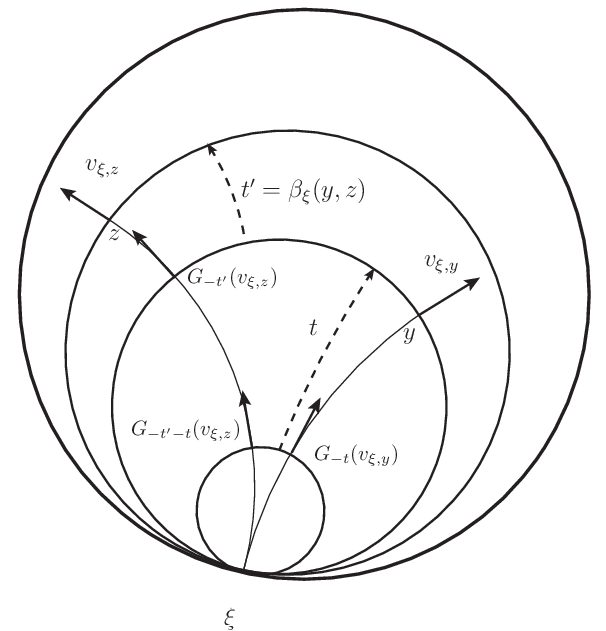}
\caption{Construction of the Gibbs kernel}\label{bordel}
\end{figure}

Let $\Jac^u_t=\det(DG_t)_{|E^u}$ denote the Jacobian of the time $t$ of the flow in the unstable direction and let $\beta_{\xi}$ denotes the Busemann function at $\xi\in\tL(\infty)$ (see \S \ref{sphereinfty} for the definition). One defines the \emph{Gibbs kernel} $k^u$ on $\widetilde{L}\times\widetilde{L}\times\widetilde{L}(\infty)$ by the formula
$$k^u(y,z;\xi)=\lim_{t\to\infty}\frac{\Jac^u_{-t-\beta_{\xi}(y,z)}(v_{\xi,z})}{\Jac^u_{-t}(v_{\xi,y})},$$
where $v_{\xi,z}$ is the unit vector based at $z$ such that $-v_{\xi,z}$ points to $\xi$. We justify the existence of this limit in \S \ref{gibbskernelsection}.

We define a $\phi^u$-\emph{harmonic function} on $\widetilde{L}$ as a function which has an integral representation
$$h(z)=\int_{\widetilde{L}(\infty)}k^u(o,z;\xi)d\eta_o(\xi),$$
where $o\in\widetilde{L}$ is a base point and $\eta_o$ is a finite Borel measure on $\widetilde{L}(\infty)$. When $L$ is hyperbolic, the Gibbs kernel coincides with the usual Poisson kernel (see Remark \ref{kernelhyperbolic}) and $\phi^u$-harmonic functions are in fact \emph{harmonic} (recall that a function is harmonic if its laplacian vanishes everywhere).

A $\phi^u$-\emph{harmonic measure} for $\F$ is a probability measure on $M$ which has Lebesgue disintegration for $\F$, and whose local densities in the plaques with respect to the Lebesgue measure are $\phi^u$-harmonic functions.

If one copies verbatim the definition above replacing ``$\phi^u$-harmonic'' by ``harmonic'' one obtains Garnett's definition of harmonic measures (see \cite{Gar} and Definition \ref{harmonicmeasure}). In particular, when the leaves of $\F$ are hyperbolic all $\phi^u$-harmonic measures for $\F$ are in fact harmonic.

The next theorem provides a canonical one-to-one correspondence between Gibbs $u$-states  and $\phi^u$-harmonic measures. In particular, it provides the existence of $\phi^u$-harmonic measures for foliations with negatively curved leaves.

\begin{maintheorem}
\label{mthbijcorresp}
Let $(M,\F)$ be a closed foliated manifold endowed with a negatively curved leafwise metric. Let $\TT$ be a complete transversal to $\F$. 
\begin{itemize}
\item For every Gibbs $u$-state for  the foliated geodesic flow $G_t$, there exists a unique $\phi^u$-harmonic measure for $\F$ inducing the same measure on $\TT$.
\item Reciprocally, for every $\phi^u$-harmonic measure for $\F$, there exists a unique Gibbs $u$-state for $G_t$ inducing the same measure on $\TT$.
\end{itemize}
\end{maintheorem}

We will postpone the definitions of complete transversals and induced measures until \S \ref{reparouille}.

The name $\phi^u$-harmonic has been chosen because Gibbs $u$-states come from a potential that is usually denoted by $\phi^u$ (see \cite{BR} for example and Formula \eqref{Eq:potential} for the definition of the potential). This choice of terminology is coherent with the notion of $F$-harmonic measures introduced in \cite{Al2} for foliated bundles and general potentials.

\paragraph{Ergodic decomposition.}

Using Theorem \ref{mthbijcorresp} we are able to study the structure of the space of $\phi^u$-harmonic measures. Say a $\phi^u$-harmonic $m$ is \emph{ergodic} if we have $m(\X)=0$ or $1$ for every Borel set $\X\dans M$ which is saturated by $\F$.

\begin{maintheorem}[Ergodic decomposition]
\label{decompositionergodiquephiu}
Let $(M,\F)$ be a closed foliated manifold endowed with a negatively curved leafwise metric. The space $\HH ar^{\phi^u}(\F)$ of $\phi^u$-harmonic measures for $\F$ is a non empty convex set whose extremal points are given by the ergodic measures.

Moreover, there exists a Borel set $\X$ which is full for all $\phi^u$-harmonic measures, as well as a unique family $(m_x)_{x\in\X}$ of probability measures on $M$ such that:
\begin{enumerate}
\item for all $x\in\X$, $m_x$ is an ergodic $\phi^u$-harmonic measure;
\item if $x,y\in\X$ belong to the same leaf then $m_x=m_y$;
\item for every $\phi^u$-harmonic measure $m$, we have:
$$m=\int_{\X} m_x\,dm(x).$$
\end{enumerate}
\end{maintheorem}

The proof of this theorem uses the fact that ergodic components of Gibbs $u$-states are still Gibbs $u$-states (see Theorem \ref{reconstruirelesetatsdeugibbs}). The measures $m_x$ of the theorem above are the images of ergodic  components of Gibbs-$u$-states by the correspondence given in Theorem \ref{mthbijcorresp}. More details will be found in \S \ref{ergounette}.

\paragraph{Generalization of a theorem of Matsumoto.} In \cite{Mat} Matsumoto considered closed foliated manifolds $(M,\F)$ with hyperbolic leaves and their harmonic measures (in the sense of Garnett \cite{Gar}).  Following Matsumoto's terminology we say that a property holds for an \emph{$m$-typical leaf} (or just \emph{typical} when there is no ambiguity) if it holds in a saturated Borel set full for $m$. Matsumoto showed that the extension by holonomy of a local harmonic density on a typical leaf $L$ (see Lemma \ref{extenddensity} for more details about extension of local densities by holonomy) defines, up to multiplication by a constant, a harmonic function on its universal cover $\widetilde{L}$ called the \emph{characteristic function} of $L$ denoted by $h_L$. Using the Poisson representation of harmonic functions, we see that the function $h_L$ is associated to a measure on $\widetilde{L}(\infty)$ denoted by $\eta_L$. Its measure class $[\eta_L]$ only depends on the leaf $L$, and we call it the \emph{characteristic measure class} of $L$.

By analyzing the properties of Brownian motion tangent to the leaves, he proved the following
\begin{theorem}[Matsumoto]
\label{matsumototheorem}
Let $(M,\F)$ be a closed foliatied manifold whose leaves are hyperbolic manifolds. Let $m$ be a harmonic measure which is not totally invariant, i.e. which is not locally the product of the Lebesgue measure of the plaques by a transverse invariant measure. Then for $m$-almost every leaf $L$, the characteristic measure class $[\eta_L]$ on $\widetilde{L}(\infty)$ is singular with respect to the Lebesgue measure. Moreover, the characteristic function of $m$-almost every leaf is unbounded.
\end{theorem}

When the curvatures of the leaves of $\F$ are variable and $m$ is a $\phi^u$-harmonic measure, we can define the characteristic $\phi^u$-harmonic function as well as the characteristic measure class of $m$-almost every leaf $L$ in the same way (see \S \ref{generthmatsumoto} for the details). Moreover there is a way to associate canonically to every transverse measure a $\phi^u$-invariant measure, which we call \emph{totally invariant} (see Definition \ref{totallyinvphiu}).

If $L$ is a leaf of $\F$ and $z\in\tL$, $T_z^1\tL$ can be identified to $\tL(\infty)$ by sending a vector $v$ on the equivalence class of the geodesic ray it directs. Pushing the Lebesgue measure of $T_z^1\tL$ by this identification provides a measure on $\tL(\infty)$ which depends on $z$. However it can be shown (see Lemma \ref{visibility}) that its measure class does not. We call it the \emph{visibility class of $\tL(\infty)$}.

The following theorem gives a sufficient condition on the densities of $\phi^u$-harmonic measures for the existence of a transverse invariant measure. It implies Matsumoto's result (we explain how in \S \ref{generthmatsumoto}). The proof of this theorem is dynamical and only relies on the absolute continuity of horospheric subfoliations.

\begin{maintheorem}
\label{matsumotocbrvar}
Let $(M,\F)$ be a closed foliated manifold endowed with a negatively curved leafwise metric. Let $m$ be a non totally invariant $\phi^u$-harmonic measure. Then for $m$-almost every leaf $L$, the characteristic measure class $[\eta_L]$ on $\widetilde{L}(\infty)$ is singular with respect to the visibility class.
\end{maintheorem}

\paragraph{Organization of the paper.} In Section \ref{transversemeasures} we will give the basic properties of transverse measures for foliations as well as of the associated cocycles. In Section \ref{folgeoflownegcurvleaves}, we introduce the foliated geodesic flow and study its hyperbolic properties when the leaves are negatively curved. We state theorems of existence and absolute continuity of stable and unstable foliations, and discuss the local product structure of the Liouville measure inside the leaves. In Section \ref{sugibbssection}, we prove Theorem \ref{mtheoremsugibbsgeodesique}. In Section \ref{phiuharmonicsection} we define $\phi^u$-harmonic measures, prove Theorem \ref{mthbijcorresp} and obtain some basic results about the ergodic theory of these measures. Finally in Section \ref{matsumotosection}, we prove our generalization of Matsumoto's theorem.

\section{Transverse measures for foliations}
\label{transversemeasures}

\subsection{Foliations and holonomy}

\paragraph{Foliations.} A closed manifold ($C^{\infty}$, compact and boundaryless) $M$ of dimension $n$ possesses a \emph{foliation} of dimension $d$ and of class $C^{\infty}$ (we will say that $(M,\F)$ is a closed foliated manifold) if it is endowed with a finite \emph{foliated atlas} $\A=(U_i,\phi_i)_{i\in I}$. This means that  $(U_i)_{i\in I}$ is a finite open cover of $M$, that $\phi_i:U_i\to P_i\times T_i$ is a $C^{\infty}$-diffeomorphism, where $P_i$ and $T_i$ are cubes of respective dimensions $d$ and $n-d$ and that when $U_i\cap U_j\neq\vide$, the corresponding changes of charts $\phi_j\circ\phi_i^{-1}:\phi_i(U_i\cap U_j)\to\phi_j(U_i\cap U_j)$ are $C^{\infty}$-diffeomorphisms of the form
\begin{equation}
\label{Eq:changecharts}
\phi_j\circ\phi_i^{-1}(z,x)=(\zeta_{ij}(z,x),\tau_{ij}(x)),
\end{equation}
where $(z,x)\in\phi_i(U_i\cap U_j)$. Here $\tau_{ij}$ is a diffeomorphism between $S_i=p_i\circ\phi_i(U_i\cap U_j)$ and $S_j=p_j\circ\phi_j(U_i\cap U_j)$ where $p_i$ denotes the natural projection $P_i\times T_i\to T_i$ ($p_j$ is defined similarly). The map $\zeta_{ij}$ is a smooth map such that $\zeta_{ij}(.,x)$ maps diffeomorphically $\phi_i(U_i\cap U_j)\cap (P_i\times\{x\})$ onto $\phi_j(U_i\cap U_j)\cap (P_j\times\{\tau_{ij}(x)\})$ for every $x\in S_i$ .

\begin{figure}[hbtp]
\centering
\includegraphics[scale=0.7]{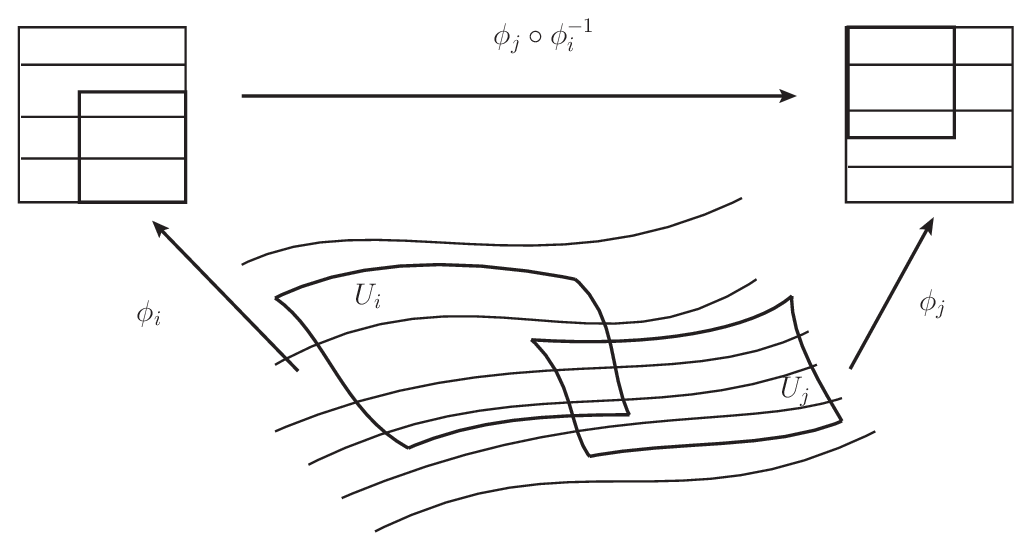}
\caption{Changes of charts}
\label{foliouille}
\end{figure}

The sets $\phi_i^{-1}(P_i\times\{x\})$ are called \emph{plaques}. Since the changes of charts have this very particular form, we can glue the plaques together, so as to obtain a partition of $M$ by immersed submanifolds, called the \emph{leaves} of $\F$. We will denote by $(P_i(x))_{x\in T_i,i\in I}$, the plaques of $\F$ and, abusively, we will always identify the transversals $T_i$ with their preimages $\phi_i^{-1}(\{z_i\}\times T_i)$ for some choice $z_i$. Such a collection of embedded submanifolds transverse to the leaves and whose union $\TT$ meets each leaf is called a \emph{complete system of transversals}. Finally we identify the maps $\tau_{ij}$ with diffeomorphisms between open sets of $\TT$.

\begin{rem}
We chose for commodity to work with smooth foliations but the work presented here would be also valid in the more general case of \emph{laminations} of compact metric spaces by manifolds of class $C^3$ as defined in \cite{Gh} for example (we explain in Remark \ref{regularity} why we need $C^3$).
\end{rem}

\paragraph{Leafwise metrics.} Say a closed foliated manifold $(M,\F)$ is endowed with a \emph{leafwise metric} if
\begin{itemize}
\item each leaf $L$ possesses a $C^{\infty}$ Riemannian metric denoted by $g_L$;
\item the metric $g_L$ varies continuously transversally in local charts in the $C^{\infty}$-topology.
\end{itemize}

\begin{rem}
\label{transitionisom}
Let $\A=(U_i,\phi_i)_{i\in I}$ be a foliated atlas for $\F$, which we assume to be endowed with a leafwise metric. By definition, if we push forward by $\phi_i$ the metric of the plaque $\phi_i^{-1}(P_i\times\{x\})$, $x\in T_i$, we obtain a metric on $P_i$ denoted by $g_i(x)$ which varies continuously with $x$ in the smooth topology. Assume that $U_i\cap U_j\neq\vide$. Then the function $\zeta_{ij}(.,x)$, where it is defined, is an isometry between $g_i(x)$ and $g_j(\tau_{ij}(x))$.
\end{rem}

\begin{rem}
\label{regularity}
We chose to stay in the $C^{\infty}$-regularity by commodity, and in the sequel, ``smooth'' will always refer to this regularity. But we could also have stayed in the $C^3$-regularity where the results would hold true as well. We can't ask for lower regularity since our results strongly depend on the absolute continuity of some subfoliations invariant by the foliated geodesic flow, which is automatic when the flow is $C^2$ in the leaves (see the discussion in the next Section).
\end{rem}

\begin{rem}
\label{bddgeo}
Since $M$ is compact, a leafwise metric gives \emph{uniformly bounded geometry} to the leaves of $\F$, i.e. their injectivity radii are bounded uniformly from below, and their sectional curvatures are uniformly pinched.
\end{rem}

\begin{rem} Our leafwise metrics don't come a priori from ambient smooth metrics. Indeed it would be too restrictive as shows the following example. Consider a foliation by hyperbolic Riemann surfaces (the universal cover of each leaf is conformally equivalent to a disc). Candel proves in \cite{C} that all the leaves can be simultaneously uniformized in the sense that there exists a leafwise hyperbolic metric (in the sense defined above), i.e. the curvature of each leaf is everywhere $-1$. However, this leafwise metric does not need to be induced by an ambient Riemannian metric. It does not even need to vary smoothly transversally. We refer to \cite{AlLe} for the construction of various leafwise hyperbolic metrics for a Riemann surface foliation known as the Hirsch foliation which don't vary smoothly transversally.
\end{rem}

\begin{defi}When all leaves have negative sectional curvatures, we say that the leafwise metric is negatively curved.
\end{defi}

\paragraph{Holonomy.} Let $(M,\F)$ be a closed foliated manifold. Fix a foliated atlas $\A=(U_i,\phi_i)_{i\in I}$ of $M$ and suppose it is a \emph{good foliated atlas} in the sense that
\begin{enumerate}
\item if two plaques intersect each other, their intersection is connected;
\item if two charts $U_i$ and $U_j$ intersect each other, then $\overline{U_i\cup U_j}$ is included in a foliated chart. Thus a plaque of $U_i$ intersects at most one plaque of $U_j$.
\end{enumerate}

We can define a \emph{complete transversal} $\TT$ as the union of all the $T_i$. Recall that we chose to identify $T_i$ and a local transversal $\phi_i^{-1}(\{z_i\}\times T_i)$, and thus to consider points of $\TT$ as points of $M$, and to identify the maps $\tau_{ij}$ with local diffeomorphisms of $\TT$. The maps $\tau_{ij}$ generate a \emph{pseudogroup} $\Pp$ of local diffeomorphisms of $\TT$, called the \emph{holonomy pseudogroup} of $\F$. 

Also, we can define the \emph{holonomy map along a path}. If $c:[0,1]\to M$ is a path tangent to a leaf, and $T_i$ and $T_j$ are submanifolds transverse to $\F$ containing $c(0)$ and $c(1)$, there exist two neighbourhoods $S_i\dans T_i$ and $S_j\dans T_j$ of $c(0)$ and $c(1)$ respectively such that we can send every $x\in S_i$ onto $\tau_c(x)\in U_1$ by sliding along the leaves of $\F$. More precisely, if we consider any chain of charts that cover $c$, say $U_{i_0},...,U_{i_n}$, then $\tau_c$ is defined as the composition $\tau_{i_{n-1}i_n}\circ...\circ\tau_{i_1 i_0}$. The germ at $c(0)$ of $\tau_c:S_i\to S_j$ does not depend on the choice of $c$ nor does it on the choice of the chain of charts, but only on the homotopy class of $c$.

\subsection{Transverse measures and cocycles}
\label{transversecocycles}

\paragraph{Invariant measures.} Let $(M,\F)$ be a closed foliated manifold endowed with a good foliated atlas. We have a dynamical system given by the action of the holonomy pseudogroup $\Pp$ on the complete transversal $\TT$. We say that a finite Borel measure $\nu$ on $\TT$ is \emph{invariant} by an element $\tau\in\Pp$ if $\nu(\tau A)=\nu(A)$ for every Borel set $A$ included in the domain of $\tau$. A \emph{transverse invariant measure} is a finite measure on $\TT$ that is invariant by the action of each element of $\Pp$. Note that the existence of a transverse invariant measure is extremely rare.

In \cite{BM}, in order to prove the unique ergodicity of horocycle foliations, Bowen and Marcus introduce another notion of transverse invariant measure that is equivalent to the one we just defined and is maybe more adapted to the theory of dynamical systems. In the sequel we will make use of both notions.

Let $(T_i)_{i\in I}$ be a complete system of transversals to the foliation. A transverse invariant measure is a family of finite nonnegative measures $(\nu_i)_{i\in I}$ satisfying

\begin{enumerate}
\item $\nu_i(T_i)>0$ for some $i\in I$;
\item if for $i,j\in I$ there is a holonomy map $\tau:S_i\to S_j$ between two open sets $S_i\dans T_i$ and $S_j\dans T_j$, then for any Borel set $A_i\dans S_i$ we have $\nu_i(A_i)=\nu_j(\tau(A_i))$.
\end{enumerate}

\paragraph{Totally invariant measures.} Assume that the foliation $\F$ is endowed with a leafwise metric and possesses a transverse invariant measure $(\nu_i)_{i\in I}$. Then, if $\nu_i(T_i)>0$, it is possible inside $U_i$ to integrate the volume of the plaques against $\nu_i$. We obtain this way a measure $m_i$ in the chart $U_i$.

Since the family of measures $(\nu_i)_{i\in I}$ is holonomy-invariant, these local measures glue together and provide a finite measure $m$ on $M$. Such a measure will be from now one called \emph{totally invariant}.

\paragraph{Quasi-invariant measures.} Most foliations don't possess transverse invariant measures. Thus, we are led to consider invariant \emph{measure classes} on a complete transversal.

We say that a finite Borel measure $\nu$ on $\TT$ is \emph{quasi-invariant} by the holonomy pseudogroup $\Pp$ if for every $\tau\in\Pp$ and every Borel set $A$ lying in the domain of $\tau$ such that $\nu(A)=0$, we still have $\nu(\tau A)=0$.

We associate to each transverse quasi-invariant measure the so called \emph{Radon-Nikodym cocycle}, defined for every $\tau\in\Pp$ and $\nu$-almost every $x$ inside the range of $\tau$ by
$$k(\tau,x)=\frac{d[\tau_{\ast}\nu]}{d\nu}(x).$$
We refer to \cite{KL} for a very interesting discussion about Radon-Nikodym cocycles.

\paragraph{Full and saturated sets.} The following lemma states that without loss of generality, one can assume that a Borel subset of a complete transversal which is of full measure for a transverse quasi-invariant measure is holonomy-invariant.

\begin{lemma}
\label{fullsaturated}
Let $\nu$ be a finite Borel measure on $\TT$ which is quasi-invariant by the action of the holonomy pseudogroup $\Pp$. Let $X\dans\TT$ be Borel set of full measure for $\nu$. Then there exists a Borel set $Y\dans X$ such that
\begin{enumerate}
\item $Y$ is of full measure for $\nu$:
\item $Y$ is saturated for the action of $\Pp$ i.e. for every $y\in Y$ the orbit $\Pp y$ is included in $Y$.
\end{enumerate}
\end{lemma}

\begin{proof}
Let $\nu$ and $X$ be such as in the statement. We assume that 
$\TT$ is associated to a good foliated atlas $\A=(U_i,\phi_i)_{i\in I}$, i.e. we have $\TT=\bigcup_{i\in I}T_i$. When $U_i\cap U_j\neq\vide$, we set $X_{ij}=X\cap\dom(\tau_{ij})\dans T_i$ (where $\dom(\tau_{ij})$ denotes the domain of the holonomy map $\tau_{ij}$).

Set $X'_{ji}=X_{ji}\moins\tau_{ij}(X_{ij})\dans T_j$. Since the class of $\nu$  is preserved by the holonomy maps $\tau_{ij}$ and $X_{ij}$ is full in $\dom(\tau_{ij})$ we have $\nu(X'_{ji})=0$.

Define $Y=X\moins\bigcup_{i,j}X'_{ij}$. This is a Borel set which is full for $\nu$ (since the sets $X_{ij}'$ are null). Let $Y_{ij}=Y\cap\dom(\tau_{ij})$. One checks that for every $i,j$ the sets $\tau_{ij}(Y_{ij})$ and $Y_{ji}$ coincide. This implies that the set $Y$ is invariant by every map $\tau_{ij}$ and thus that $Y$ is saturated by $\Pp$.
\end{proof}

\paragraph{Ghys' lemma.} The following lemma, due to Ghys (see \cite[pp.413-414]{Gh}), will be useful all along this text. Ghys stated it for some special families of transverse measures associated to harmonic measures (see Definition \ref{harmonicmeasure}). In his proof, that we recall below, he only needs the quasi-invariance of the families of measures.

\begin{lemma}[Ghys]
\label{ghysquasiinvariant}
Let $(M,\F)$ be a closed foliated manifold and let $\TT$ be a complete transversal. Let $\nu$ be a measure on $\TT$ quasi-invariant by the action of the holonomy pseudogroup $\Pp$. Then there exists a Borel set $\X_0\dans\TT$ which is full for $\nu$, and saturated by the action of $\Pp$, such that for every $x\in\X_0$ and every $\tau\in\Pp$ that fixes $x$, we have
$$\frac{d[\tau_\ast\nu]}{d\nu}(x)=1.$$
\end{lemma}

\begin{proof}
The Radon-Nikodym cocycle $d[\tau_\ast\nu]/d\nu$, $\tau\in\Pp$ allows us to define for every $x$ lying in the support of $\nu$ a morphism $\pi_1(L_x,x)\to(0,\infty)$. We want to prove that, almost surely, it is trivial.

Let $\tau\in\Pp$. Consider the Borel set constituted of points $x\in\TT$ which are fixed by $\tau$ and satisfy $d[\tau_\ast\nu]/d\nu\,(x)<1$. By definition, this set is fixed by $\tau$, and its measure is contracted by $\tau$. Hence, it has measure zero. By considering $\tau^{-1}$, we can prove that the set of points $x$ which are fixed by $\tau$ and satisfy $d[\tau_\ast\nu]/d\nu(x)>1$ has also measure zero. It follows that the set of points $x$ fixed by $\tau$ and such that $d[\tau\ast\nu]/d\nu(x)\neq 1$ has zero measure for $\nu$.

The pseudogroup $\Pp$ is finitely generated. In particular, it is countable. Using the $\sigma$-additivity of $\nu$, we see that the set of points which are fixed by an element of the pseudogroup $\Pp$ with Jacobian $\neq 1$ is of measure zero for $\nu$. Denote by $\X_0$ the complement of this set. We have to show that it is saturated by the action of the pseudogroup.

The groups of germs of holonomy transformations which fix two points of the same leaf are conjugated. Thus if $x\in\X_0$, we see that for every $y\in L_x$, and $\tau\in\Pp$ that fixes $y$, we also have $d[\tau_\ast\nu]/d\nu(y)=1$. This shows that $\X_0$ is saturated by $\Pp$, completing the proof of the lemma.
\end{proof}

We have the following useful interpretation of Ghys' lemma. For $\nu$-almost every point $x\in\TT$ we have, if $\tau_1,\tau_2\in\Pp$ contain $x$ in their domains and satisfy $\tau_1(x)=\tau_2(x)$
\begin{equation}
\label{consequenceghys}
\frac{d\left[\tau_{1\ast}^{-1}\nu\right]}{d\nu}(x)=\frac{d\left[\tau_{2\ast}^{-1}\nu\right]}{d\nu}(x).
\end{equation}

\subsection{Disintegration and foliations}
\label{disintegrationfoliations}

\paragraph{Disintegration of measures.} Let us give a general discussion about Rokhlin's theory of disintegration of measures. For the details we refer to Rokhlin's seminal paper  \cite{Rok}

Let $\ZZ$ be a compact metric space, $\mu$ be a finite Borel measure on $\ZZ$ and $\Pi$ be a partition modulo $\mu$ of $\ZZ$ into measurable sets (i.e. the intersection of two distinct atoms of $\Pi$ is null for $\mu$ while the union of all atoms of $\Pi$ is full). Say the partition $\Pi$ is \emph{measurable} if there exist a set $F_0$ full for $\mu$ as well as a countable family of measurable sets $(B_n)_{n\in\N}$ such that for every pair of distinct atoms $(A_1,A_2)$ of $\Pi$ there exists $n\in\N$ such that:
$$A_1\cap F_0\dans B_n\,\,\,\,\,\,\,\,\textrm{and}\,\,\,\,\,\,\,\,\,A_2\cap F_0\dans\,\,^cB_n\,\,\,\,\,\,\,\,\,\,\,\,\,\,\,\,\,\,\,\,\,\,\,\,\,\,\,\,\,\,\,\,\,\,\,\,\,\,\,\,\,\,\,\,\,\,\,\,\textrm{or}\,\,\,\,\,\,\,\,\,\,\,\,\,\,\,\,\,\,\,\,\,\,\,\,\,\,\,\,\,\,\,\,\,\,\,\,\,\,\,\,\,\,\,\,\,\,\,\,A_1\cap F_0\dans \,\,^cB_n\,\,\,\,\,\,\,\,\textrm{and}\,\,\,\,\,\,\,\,A_2\cap F_0\dans B_n.$$
Let $p:\ZZ\to\Pi$ denotes the natural projection which associates to $z\in\ZZ$ the atom containing it, we define $\nu=p_\ast\mu$. The set $\Pi$ is endowed with the image by $p$ of the Borel $\sigma$-algebra.

\begin{defi}[Disintegration]
A \emph{disintegration} of $\mu$ with respect to $\nu$ is a family of measures $(\mu_P)_{P\in\Pi}$, called \emph{conditional measures}, such that:
\begin{enumerate}
\item $\mu_P(P)=1$ for $\nu$-almost every $P\in\Pi$;
\item for every continuous function $f:\ZZ\to\R$ the map $P\mapsto\int_P fd\mu_P$ is measurable and:
$$\int_{\ZZ}fd\mu=\int_{\Pi}\left(\int_P fd\mu_P\right) d\nu(P).$$

We will often denote the disintegration by:
$$d\mu=(d\mu_P)\,\,d\nu(P).$$
\end{enumerate}
\end{defi}

\begin{theorem}[Rokhlin]
\label{Rokhlin}
Let $\ZZ$ be a compact metric space, $\Pi$ be a measurable partition. Every finite Borel measure $\mu$ can be disintegrated on the atoms of $\Pi$ with respect to its projection. Moreover the disintegration is unique up to a zero measure set in $\Pi$.
\end{theorem}

\begin{rem} We can disintegrate $\mu$ with respect to \emph{any measure} $\nu'$ \emph{equivalent to the projection}. But in that case, the new conditional measures $\mu'_P$ \emph{are no longer probability measures}: by uniqueness in Rokhlin's theorem, they are obtained by multiplication of $\mu_P$ by the Radon-Nikodym derivative $d\nu/d\nu'(P)$.
\end{rem}

\begin{rem}
If $\ZZ=\X\times\Y$ where $\X$ and $\Y$ are compact metric spaces the trivial partition $\Pi=(\X\times\{y\})_{y\in\Y}$ is measurable.

More generally the fibers of a continuous fiber bundle $p:\ZZ\to\Y$ with fiber $\X$, with $\X,\Y$ compact metric spaces, form a measurable partition of $\ZZ$.

Finally in general the leaves of a foliation form a partition which is not measurable. However since this partition is locally trivial, it is possible  to disintegrate locally probability measures in the local foliated charts.
\end{rem}

\paragraph{Lebesgue disintegration and cocycles.}

Assume now that $(M,\F)$ is endowed with a leafwise metric. In particular each leaf of $\F$ is endowed with a volume form.

\begin{defi}[Lebesgue disintegration]
\label{Lebacsing}
We say that a probability measure $m$ on $M$ has \emph{Lebesgue disintegration} for $\F$ if for every foliated chart $U$ the conditional measures of the restriction $m_{|U}$ in the plaques of $U$ are equivalent to the volume of the plaques.

We say that a probability measure $m$ on $M$ has \emph{Lebesgue-singular disintegration} for $\F$ if for every foliated chart $U$ the conditional measures of the restriction $m_{|U}$ in the plaques of $U$ are singular with respect to the volume of the plaques.
\end{defi}

Let $m$ be a measure on $M$ which has Lebesgue disintegration for $\F$. We can associate to $m$ a family of transverse quasi-invariant measures as well as a Radon-Nikodym cocycle. Indeed, let $\A=(U_i,\phi_i)_{i\in I}$ be a good foliated atlas. Consider a chart of the form $U_i=\bigcup_{x\in T_i}P_i(x)$ satisfying $m(U_i)>0$. In restriction to $U_i$, $m$ disintegrates as follows
\begin{equation}
\label{disintegration}
dm_{|U_i}=\left(h_i(z,x)\,d\Leb_{P_i(x)}(z)\right)\,\,d\nu_i(x),
\end{equation}
where $\nu_i$ is a finite Borel measure on $T_i$, $\Leb_{P_i(x)}$ denotes the Lebesgue measure of the plaque $P_i(x)$ and $h_i$ is a measurable function of $U_i$ such that for $\nu_i$-almost every $x$, $h_i(.,x)$ is positive and Lebesgue-integrable on $P_i(x)$ (we use here an abusive identification between $U_i$ and $P_i\times T_i$).

Then the family of transverse measures $(\nu_i)_{i\in I}$ is quasi-invariant by the holonomy pseudogroup. Indeed, if $m(U_i\cap U_j)>0$ the evaluation of $m$ on $U_i\cap U_j$ gives 
\begin{eqnarray*}
dm_{|U_i\cap U_j}&=&\left(h_i(z,x)\,d\Leb_{P_i(x)\cap P_j(\tau_{ij}(x))}(z)\right)d\nu_i(x)\\
                 &=&\left(h_j(z,\tau_{ij}(x))\,d\Leb_{P_i(x)\cap P_j(\tau_{ij}(x))}(z)\right)d[\tau_{ij\ast}^{-1}\nu_j](x).
\end{eqnarray*}
In particular, we find that for, $\nu_i$-almost all $x\in T_i$
\begin{equation}
\label{quasiinvariantemesureabsolumtcont}
\frac{d[\tau_{ij\ast}^{-1}\nu_j]}{d\nu_i}(x)=\frac{h_i(z,x)}{h_j(z,\tau_{ij}(x))}.
\end{equation}
 Note that in particular the right hand side of \eqref{quasiinvariantemesureabsolumtcont} does not depend on $z$.

\paragraph{Harmonic measures.} When $(M,\F)$ is endowed with a leafwise metric, there exist special measures with Lebesgue disintegration called \emph{harmonic} and which have been widely studied. See for example Ghys' topological classification for ``generic'' leaves of a Riemann surface lamination given in \cite{Gh}. There, a leaf is generic if it belongs to a saturated Borel set full for every harmonic measure.

Every leaf $L$ possesses a \emph{Laplace-Beltrami operator} $\Delta_L$. Recall that a real function $h$ defined on a leaf $L$ is said to be \emph{harmonic} if it is of class $C^2$ and if $\Delta_L h=0$.

\begin{defi}[Harmonic measures]
\label{harmonicmeasure}
Let $(M,\F)$ be a closed foliated manifold endowed with a leafwise metric. A probability measure $m$ on $M$ is said to be harmonic if it has Lebesgue disintegration for $\F$, and if the local densities with respect to the Lebesgue measure of the plaques are harmonic functions.

In other words for every foliated chart $U_i$ with $m(U_i)>0$ the restriction $m_{|U_i}$ disintegrates as \eqref{disintegration} for some finite Borel measure $\nu_i$ on the transversal $T_i$ and some measurable function $h_i$ such that for $\nu_i$-almost every $x$, $h_i(.,x)$ is a positive function of $P_i(x)$ which is of class $C^2$ and harmonic. 
\end{defi}

Harmonic measures have been introduced, and their existence has been proven, by Garnett in \cite{Gar}. They describe the behaviour of \emph{Brownian paths} tangent to the leaves of $\F$ (we refer to \cite{Gar} for further details).

\begin{rem}
\label{totallyinvariant}
Note that totally invariant measures, when they exist, are harmonic measures since the local densities in the plaques are constant functions which are in particular harmonic. In that sense, harmonic measures are generalizations of transverse invariant measures.
\end{rem}

\paragraph{Extension of local densities.} The proof of the next lemma is a straightforward application of Ghys' lemma (see \cite[p.414]{Gh}) and more precisely of the interpretation given by Formula \ref{consequenceghys}. It states that when a measure $m$ has Lebesgue disintegration, the local density with respect to Lebesgue defined in a typical plaque can be extended to the whole leaf.

\begin{lemma}
\label{extenddensity}
Let $(M,\F)$ be a closed foliated manifold endowed with a leafwise metric. Let $m$ be a measure with Lebesgue disintegration for $\F$. Then there is a full and saturated Borel set $\X\dans M$ such that for every $y\in\X$, if $y\in P_{i_0}(x)\dans U_{i_0}$, for some $i_0\in I$, and $x\in T_{i_0}$, the following formula defines a positive and measurable function of $z\in L_y$
\begin{equation}
\label{Eq:harmonicextension1}
H_x(z)=\frac{d[\tau_{c\ast}^{-1}\nu_i]}{d\nu_{i_0}}(x) h_i(z,\tau_c(x)),
\end{equation}
where $z\in U_i$, $c$ is any path joining $y$ and $z$, and $\tau_c$ is a holonomy map along $c$.
\end{lemma}

\section{The foliated geodesic flow of foliations with negatively curved leaves}
\label{folgeoflownegcurvleaves}

\subsection{Definitions}
\label{definitionunittgtbdle} 

 Let $M$ be a closed manifold endowed with a smooth foliation $\F$ of dimension $d\geq 2$.

\paragraph{The unit tangent bundle.} In what remains of the article $\hM$  shall denote the \emph{unit tangent bundle} of $\F$ i.e. the set of unit vectors tangent to the leaves of $\F$. It is a closed manifold endowed with a foliation that we will denote by $\hcF$ whose leaves are the unit tangent bundles of leaves of $\F$. We shall also denote by $pr:\hM\to M$ the \emph{basepoint projection} which associates to each vector tangent to $\F$ its basepoint.

When $v\in\hM$ we will often denote by $\hL_v$ the leaf of $v$, that is $T^1L_x$ where $x$ denotes the basepoint of $v$ and $L_x$, the leaf of $x$. We will also often denote by $T^1_x\F$ the unit tangent fiber $T^1_xL_x$. Note that these unit tangent fibers (which are the fibers of $pr$) form a subfoliation of $\hcF$. In the sequel, we shall denote this foliation by $\UF$.

\paragraph{The foliations $\F$ and $\hcF$ have the same holonomy.} More precisely, consider a good atlas $\A$ of $\F$. By lifting $\A$ via $pr$, one obtains an atlas $\widehat{\A}$ of $\hM$ consisting of local diffeomorphisms $\widehat{\phi}_i:\widehat{U}_i\to T^1P_i \times T_i$. The changes of charts are of the form
$$\widehat{\phi}_j\circ\widehat{\phi}_i^{-1}:((z,v),x)\in\widehat{\phi}_i(U_i\cap U_j)\mapsto\left((\zeta_{ij}(z,v),D_z\zeta_{ij}(.,x)v),\tau_{ij}(x)\right)\in\widehat{\phi}_j(U_i\cap U_j),$$
where $D_z\zeta_{ij}(.,x)$ denotes the derivative at $z$ of $z\mapsto \zeta_{ij}(z,x)$. Note that the definition is coherent since by Remark \ref{transitionisom} $D_z\zeta_{ij}(.,x)$ sends $T_z^1P_i$ onto $T_{\zeta_{ij}(z,x)}^1P_j$.

Since each plaque of $\A$ is diffeomorphic to an open cube it is contractible. Hence $pr$ induces a trivial bundle over such a plaque. In particular the charts $\widehat{U}_i$ are trivially foliated by fibers of $pr$ (which are spheres) and can be covered by open sets which trivialize jointly the foliations induced by fibers of $pr$ and by leaves of $\F$. This provides a foliated atlas $\B$ for $\hcF$, such that the holonomy along a path tangent to a fiber of $pr$ is trivial. Finally the respective pseudogroups of $\F$ and $\hcF$ associated to $\A$ and $\B$, acting on the system of transversals $(T_i)_{i\in I}$, are the same. In particular we get the following

\begin{proposition}
\label{equholinv}
$\F$ admits a transverse invariant measure if and only if $\hcF$ does.
\end{proposition}

\paragraph{The foliated Sasaki metric.} We refer to \cite[\S 1.3.1]{P} for further details about this topic. Let $L$ be a leaf of $\F$ and $g_L$ the corresponding Riemannian metric. The bundle $TTL$ splits as follows. Let $Z\in TTL$ and $X$ be a smooth curve of $TL$ with $\dot{X}(0)=Z$ and $c=pr\circ X$. The correspondence $Z\mapsto \left(\dot{c}(0), \left.\frac{D}{dt}\right|_{t=0} X\right)$ provides a bundle isomorphism $\iota:TTL\to H\oplus V$, where both $H$ and $V$ are copies of the pull-back of $TL$ by $pr$ ($H$ and $V$ are respectively called the \emph{horizontal} and \emph{vertical} bundles). Here $D/dt$ denotes the covariant derivative for the \emph{Levi-Civita connection}. Recall that the local coefficients of this connection, which are called the \emph{Christoffel's symbols}, only depend on the $1$-jet of the metric $g_L$ (see \cite[Chapter 2, Section 3]{dC}).

By pulling-back by $pr$ on the bundles $H$ and $V$ the bundle metric of $TL$ (i.e. the field of quadratic forms in the fibers given by $g_L$), one defines a natural bundle metric on $H\oplus V$ which makes $H$ and $V$ orthogonal. By pulling-back this metric by $\iota$ one defines a bundle metric on $TTL$. The resulting Riemannian metric on $TL$ is called the \emph{Sasaki metric}. This metric only depends on the $1$-jet of $g_L$ (as does the identification $\iota$). In particular, it varies locally continuously with the metric $g_L$. We consider the restriction, which we denote by $\hg_{\hL}$, of this Riemannian metric to $T^1L$, which from now on we denote by $\hL$.

The discussion above proves that the assignment $\hL\mapsto \hg_{\hL}$ is a leafwise metric on $\hM$. We shall refer to it as the \emph{foliated Sasaki metric}.

\paragraph{The foliated geodesic flow.} A vector $v\in\hM$ directs a unique geodesic inside $L_v$. We define $G_t(v)$ by flowing $v$ along this geodesic at unit speed during a time $t$. This flow $G_t$ is  called the \emph{foliated geodesic flow}.

This flow is $C^{\infty}$ inside the leaves of $\hcF$ and varies continuously in the $C^{\infty}$-topology with the transverse parameter. Indeed, the geodesics are solutions of the \emph{geodesic equation} which is a second order ODE whose coefficients are locally given by the Christoffel's symbols (see \cite[Chapter 3, Section 2]{dC}). Therefore, this equation as well as its solutions are continuous with the metric in the $C^{\infty}$-topology.

\subsection{Foliated hyperbolicity.} 
\label{foliatedhyperbolicity}

Until the end of the paper we assume that $(M,\F)$ is endowed with a negatively curved leafwise metric. Note that by Remark \ref{bddgeo} the sectional curvature of every leaf is everywhere pinched between two uniform negative constants denoted by $-b^2\leq-a^2<0$.

If $v,w$ belong to the same leaf $\hL$ we denote by $\dist_{\hcF}(v,w)$ the distance between them for the Sasaki metric $\hg_{\hL}$ on $\hL$. When $r>0$ we denote by $B_{\hcF}(v,r)$ the ball inside $\hL_v$ centered at $v$ of radius $r$  for $\dist_{\hcF}$.

The leaves of $\F$ are not compact a priori. But since they are leaves of a foliation of a compact manifold, there has to be some sort of recurrence in their geometries (see Remark \ref{bddgeo}). In particular this will enable us to recover some of the fundamental tools of uniform hyperbolic dynamics. Most of the proofs can be copied without modification from the classical results, for which we will give precise references. In the sequel $T\hcF$ shall denote the \emph{tangent bundle} of $\hcF$, i.e. the subbundle of $T\hM$ consisting of vectors which are tangent to some leaf of $\hcF$. 

\paragraph{Invariant bundles.} Following \cite[Chapter IV]{Ba},  one proves that there are two continuous and $DG_t$-invariant subbundles of $T\hcF$ of the same dimension $d-1$ ($d$ being the dimension of $\F$) denoted by $E^s$ and $E^u$ such that

\begin{equation}
\label{uniformcontraction}
\begin{cases}
\displaystyle{\frac{a}{b}e^{-bt}||v_s||\leq ||DG_t(v_s)||\leq \frac{b}{a}e^{-at}||v_s||} & \text{if $v_s\in E^s$ and $t>0$} \\
                                                                                                                      \\
               \displaystyle{\frac{a}{b}e^{-bt}||v_u||\leq ||DG_{-t}(v_u)||\leq \frac{b}{a}e^{-at}||v_u||} & \text{if $v_u\in E^u$ and $t>0$}
   .\end{cases}
\end{equation}

We then have the following continuous and $DG_t$-invariant splitting of the tangent bundle of the foliation

\begin{equation}
\label{invsplitting}
T\hcF=E^s\oplus E^c\oplus E^u,
\end{equation}
where $E^c=\R X$, $X$ being the generator of the foliated geodesic flow. We will also set $E^{cs}=E^s\oplus E^c$ and $E^{cu}=E^u\oplus E^c$. These continuous and $DG_t$-invariant subbundles are respectively called \emph{center-stable} and \emph{center-unstable} bundles.

\emph{Let us recapitulate}. $G_t$ is a continuous flow, which is smooth inside the leaves and varies transversally continuously in the smooth topology. It preserves a continuous splitting of $T\hcF$ of the form \eqref{invsplitting} where the first and last factors are respectively uniformly exponentially contracted and expanded by $DG_t$. This is precisely the definition of \emph{foliated hyperbolicity} given in \cite{BGM}.

When $L$ is a leaf of $\F$ and $E$ is a subbundle of $T\hL$ we define the distance $\delta(E(x),E(y))$ as the Hausdorff distance in $T\hL$ of the unit spheres of $E(x)$ and $E(y)$. Brin's proof of \cite[Proposition 4.4.]{Ba} can be copied verbatim in order to get the

\begin{proposition}
\label{holderbundles}
The $DG_t$-invariant distributions $E^s,E^u,E^{cs},E^{cu}$ are uniformly H{\"o}lder continuous in the leaves of $\hcF$.
\end{proposition}

\paragraph{Stable and unstable manifolds.}  In the following theorem we use the notations defined in the preceding paragraphs.

\begin{theorem}[Stable manifold theorem]
\label{stable}
Let $(M,\F)$ be a closed foliated manifold endowed with a negatively curved leafwise metric. For any $v\in \hM$, there exists a pair of $C^{\infty}$ open discs embedded in the leaf of $v$ which contain $v$ and are denoted by $W_{loc}^s(v)$ and $W_{loc}^u(v)$ such that

\begin{enumerate}
\item $T_v W_{loc}^s(v)=E^s(v)$ and $T_v W_{loc}^u(v)=E^u(v)$;
\item for any $t\geq 0$, $G_t(W^s_{loc}(v))\dans W^s_{loc}(G_t(v))$ and $G_{-t}(W^u_{loc}(v))\dans W^u_{loc}(G_{-t}(v))$;
\item there exist immersed global manifolds that subfoliate $\hcF$ defined by
$$W^s(v)=\bigcup_{t\geq 0} G_{-t}(W^s_{loc}(G_t(v)))$$
$$W^u(v)=\bigcup_{t\geq 0} G_{t}(W^u_{loc}(G_{-t}(v))).$$
\end{enumerate}
The sets $W_{loc}^s(v)$ and $W_{loc}^u(v)$ are respectively called the local stable and unstable manifolds of $v$. These local manifolds vary continuously with the point $v$ in the $C^{\infty}$-topology (in the leaves, and also transversally). The global manifolds are characterized by the following dynamical properties
$$W^s(v)=\left\{w\in \hL_v; \lim_{t\to\infty}\dist_{\hcF}(G_t(v),G_t(w))=0\right\}$$
$$W^u(v)=\left\{w\in \hL_v; \lim_{t\to\infty}\dist_{\hcF}(G_{-t}(v),G_{-t}(w))=0\right\}.$$
\end{theorem}

\begin{proof}[Sketch of proof]
The usual \emph{Stable Manifold Theorem} given by \cite[Theorem 6.2]{Sh} adapts without modification to this context (see also \cite{BGM}).

Consider a continuous cone field $\CC^u$ around $E^u$ which is tangent to $\hcF$ and consider the set $\DD^u$ of families $(D^u_v)_{v\in\hM}$ of $C^1$-discs tangent to $\CC^u$ of a given small radius and continuous with $v$ for the $C^1$-topology. Since the flow preserves $\hcF$ and expands $E^u$, the time $T_0$-map of the flow $G_{T_0}$ sends $\CC^u$ strictly inside itself for some $T_0>0$. Moreover it acts on the space $\DD^u$ by rescaling the disc $G_{T_0}(D^u_{G_{-T_0}(v)})$ (see \cite[Chap. 6]{Sh} for more details).

Using the transverse contraction of the dynamics, we can prove that the preceeding action has a fixed point in the family $\DD^u$, which corresponds to the family of local unstable manifolds $W^u_{loc}(v)$, varying continuously with $v\in\hM$ in the $C^1$-topology (this cone field-method ensures the continuity both tangentially and transversally).

In order to prove the regularity of the family of discs we make an inductive argument. We let the derivative of the flow act on the tangent spaces of the discs of $\DD^u$ and prove that the natural action is still a contraction.
\end{proof}

\begin{figure}[hbtp]
\centering
\includegraphics[scale=0.4]{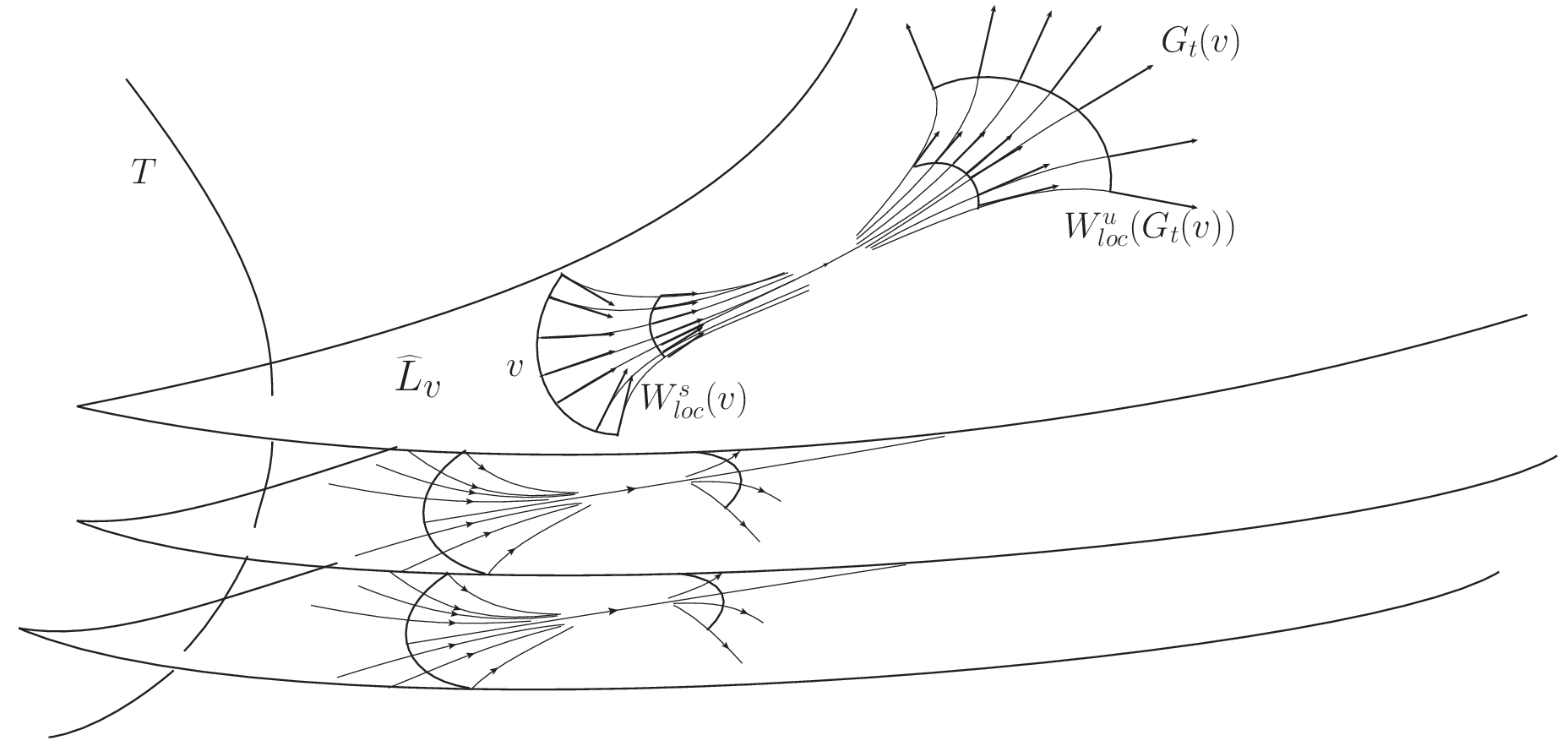}
\caption{Stable and unstable manifolds}
\label{jetehaislatex}
\end{figure}

\paragraph{Invariant foliations and local product structure.}

The collections $(W^s(v))_{v\in\hM}$ and $(W^u(v))_{v\in\hM}$ form two continuous subfoliations of $\hcF$ that we call \emph{stable} and \emph{unstable foliations}. We denote them respectively by $\W^s$ and $\W^u$.

These foliations are invariant by the foliated geodesic flow, i.e. for every $t\in\R$ and $v\in \hM$, we have $G_t(W^s(v)))=W^s(G_t(v))$ and $G_t(W^u(v)))=W^u(G_t(v))$.

The \emph{center-stable} and \emph{center-unstable manifolds} of a point $x$ will be denoted by $W^{cs}(v)$ and $W^{cu}(v)$. They are by definition the saturations of $W^s(v)$ and $W^u(v)$ in the direction of the flow. These manifolds form two continuous subfoliations of $\hcF$ denoted respectively by $\W^{cs}$ and $\W^{cu}$ that we call the \emph{center-stable} and \emph{center-unstable foliations}.

For every $t\in\R$ and $v\in\hM$, one has $G_t(W^{cs}(v))=W^{cs}(G_t(v))=W^{cs}(v)$ and $G_t(W^{cu}(v))=W^{cu}(G_t(v))=W^{cu}(v)$.

Let $v\in\hM$ and $\delta>0$ small enough. The stable manifold of radius $\delta$, denoted by $W_{\delta}^s(v)\dans\hL_v$, is defined as the connected component of $B_{\hcF}(v,\delta)\cap W^s(v)$ containing $v$. The notation $W_{loc}^s(v)$ will be used to denote a stable manifold ``with sufficiently small radius''. The same notations will be used replacing $s$ by $u,cs$ or $cu$.

The following proposition states that the leaves of $\hcF$ possess a structure of local product which is uniformly H\"older continuous. The proof of the local product structure may be found in \cite[Theorem 3.2]{PuSu}. The proof of the H\"older continuity is an application to our context of arguments of \cite{PSW}. The details can be found in an appendix of the author's thesis \cite{Al3}.

\begin{proposition}[Local product structure]
\label{lpshypfeuilletee}
For sufficiently small $\delta>0$ there exists $\gamma=\gamma(\delta)\in (0,\delta)$ such that for every $v,w$ lying on the same leaf $\hL$ and satisfying $\dist_{\hL}(v,w)\leq 2\gamma$ then $W_{\delta}^{cs}(v)\cap W_{\delta}^u(w)$ is a single point denoted by $[v,w]$. Furthermore, the local function $[.,.]$ depends continuously on $v$ and $w$, and H\"older continuously with uniform H\"older constants when $v,w$ vary on the same leaf.
\end{proposition}

\begin{figure}[hbtp]
\centering
\includegraphics[scale=0.4]{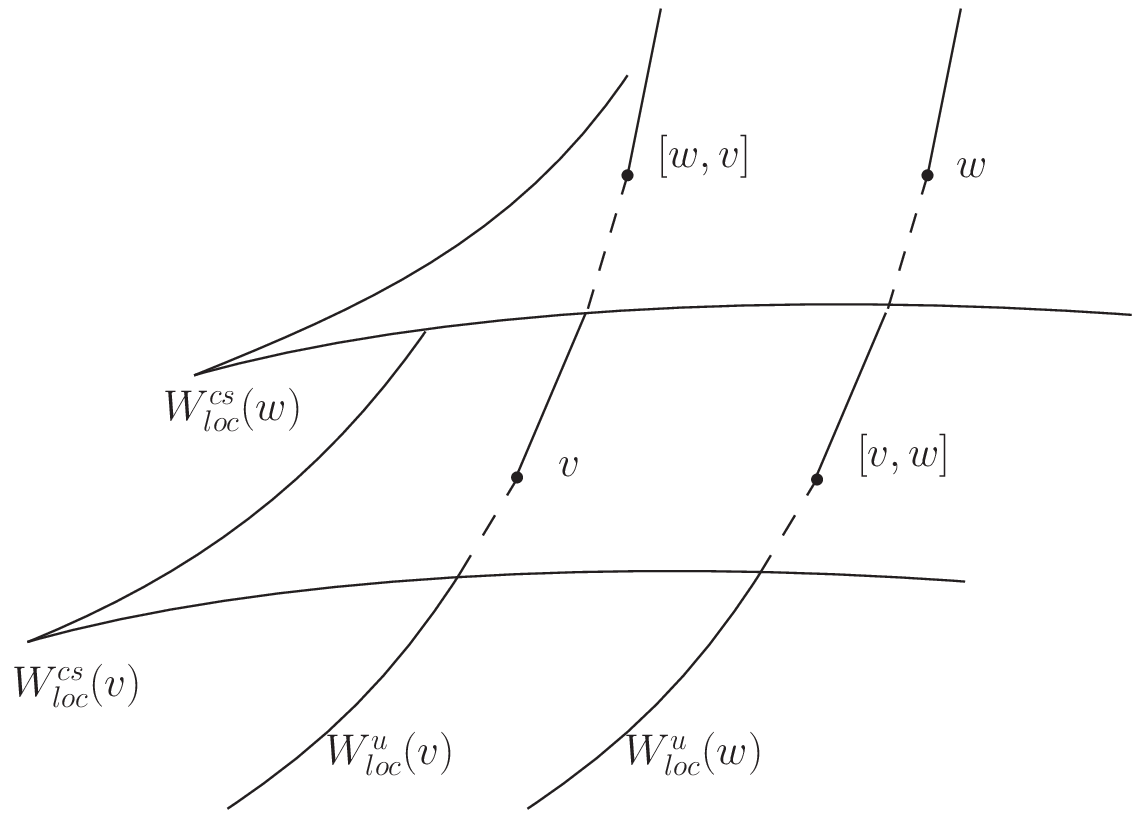}
\caption{Local product structure}
\label{baaaah}
\end{figure}

We will use the following convenient notation (see Figure \ref{rectouille}). If $v\in \hM$ and $A_u\dans W_{\gamma}^u(v)$ and $A_{cs}\dans W_{\gamma}^{cs}(v)$, set
$$[A_u,A_{cs}]=\{[v_u,v_{cs}];\,v_u\in A_u,\,v_{cs}\in A_{cs}\}.$$
The set $[W_{\gamma}^u(v),W_{\gamma}^{cs}(v)]$ will be denoted by $\RR_v$ and called a \emph{rectangle}. These rectangles are by definition foliated by local stable and unstable manifolds and are open subset of the leaves.

\begin{figure}[!h]
\centering
\includegraphics[scale=0.4]{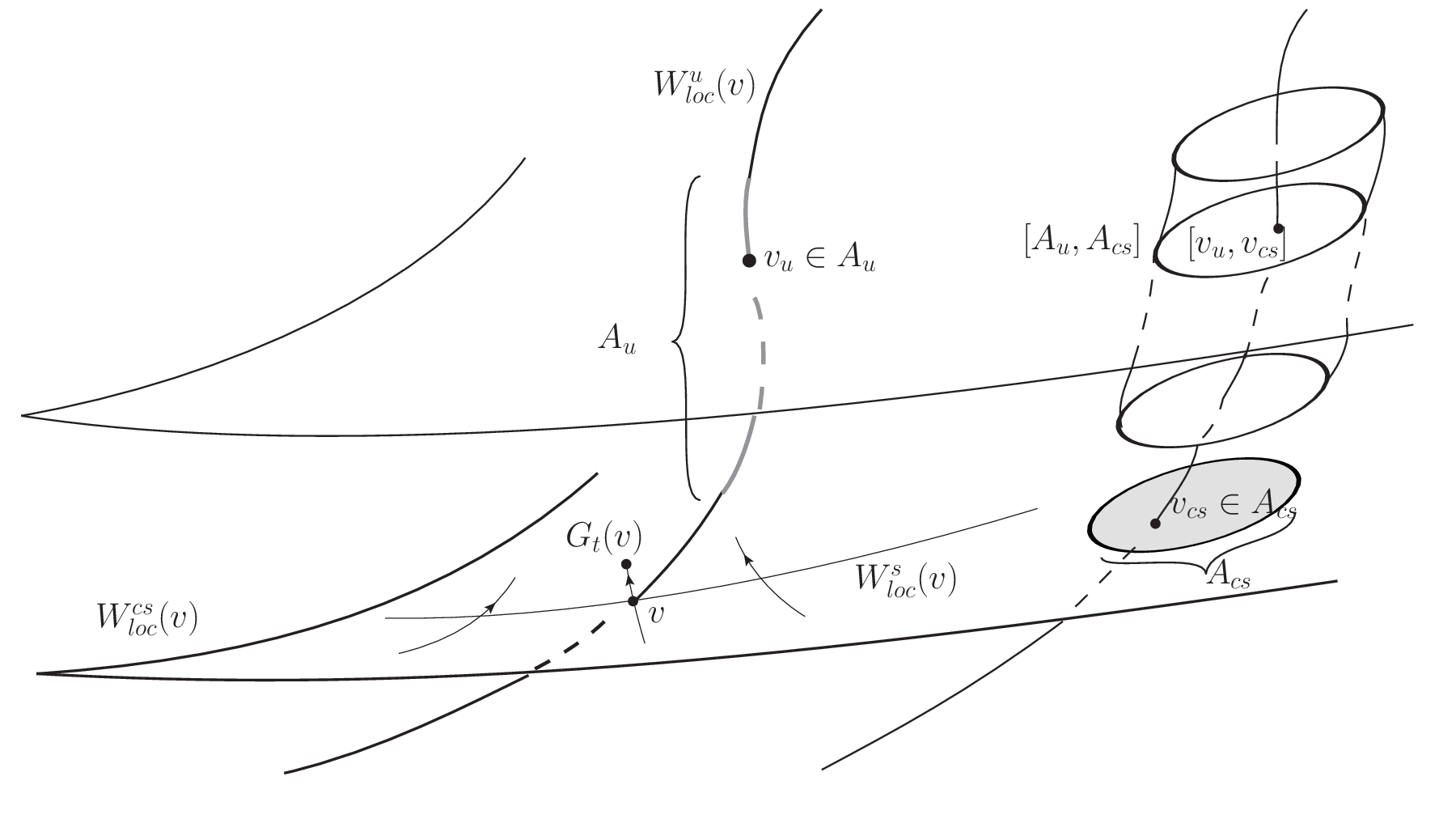}
\caption{Constructing the rectangles}
\label{rectouille}
\end{figure}

\begin{rem}
\label{transvcontrectangle}
Usually, we will call rectangle every set of the form $\RR_v=[W_{loc}^u(v),W^{cs}_{loc}(v)]$, without specifying the size of local unstable and center-stable manifolds.

By the stable manifold theorem the manifolds $W^u_{loc}(v)$ and $W^{cs}_{loc}(v)$ vary transversally continuously, and so does the local function $[.,.]$. Hence it is possible to consider transversally continuous families of rectangles $\RR_v$. This will allow us to consider foliated charts $U$ for $\hcF$ of the form
$$U=\bigcup_{v\in T}\RR_v,$$
where $T$ is a transverse section of $\hcF$.
\end{rem}

\paragraph{The unstable foliation as a limit.} Consider the foliation $\UF$ whose leaves are the unit tangent fibers (the fibers of $pr:\hM\to M$).

\begin{lemma}
\label{transvouille}
The foliations $\UF$ and $\W^{cs}$ are transverse.
\end{lemma}

\begin{proof}
Recall that we introduced a bundle isomorphism $TTL\to H\oplus V$. If $(x,v)\in \hL$, a vector tangent to $T^1_xL$ is of the form $(0,Y)$, $Y\in T_xL$. It is proven in \cite[p.72]{Ba} that a vector of $E^{cs}(v)$ has the form $(X,J'_X(0))$ where $X\in T_xL$ and $J_X$
 denotes the unique \emph{stable Jacobi field} along the geodesic directed by $v$ 
 satisfying $J_X(0)=X$ 
 (we refer to \cite[Chapter IV]{Ba} for the precise definitions). Hence these two subspaces of $T_{(x,v)}\hL$ are in direct sum.
\end{proof}

For $t\geq 0$ consider the image foliations $\UF_t=G_t(\UF)$. An adaptation of the classical inclination lemma (or $\lambda$-lemma, see \cite[Chapter 9]{Sh}) shows the following
\begin{theorem}[$\lambda$-lemma]
\label{lambdalemma}
The foliation $\UF_t$ converges to $\W^u$ in the $C^0$-topology of plane fields.
\end{theorem}

\subsection{Absolute continuity and local product structure of the volume in the leaves}

\paragraph{Some notations.} All along the text we will make use of the following notations.

First recall that the foliated geodesic flow preserves the unstable bundle $E^u$. We will denote
$$\Jac^u_t(v)=\det (D_vG_t)_{|E^u(v)}.$$
We define similarly $\Jac^s_t,\Jac^{cu}_t$ and $\Jac^{cs}_t$. We call these functions respectively \emph{unstable, stable, center-unstable and center-stable jacobians} of $G_t$.

We will also denote by $\dist_u$ the distance on unstable manifolds induced by the foliated Sasaki metric. The distances $\dist_s,\dist_{cu},\dist_{cs}$ are defined similarly.

Also, in order to prevent the confusion with holonomies of the foliations $\F$ and $\hcF$, we chose to adopt a special notation for holonomies of the invariant foliations. When $v,w$ lie in a same leaf of $\W^u$, we will use the following notation for the unstable holonomy between the respective local center-stable manifolds
$$\h uvw:W^{cs}_{loc}(v)\to W^{cs}_{loc}(w).$$
Analogous definitions are given for the holonomy maps $\h {s}vw$, $\h {cu}vw$, and $\h {cs}vw$.

The foliated Sasaki metric induces a Riemannian metric on the leaves of $\W^u$. We shall denote by $\Leb^u_v$ the Lebesgue measure induced on $W^u(v)$. The measures $\Leb^s_v$, $\Leb^{cu}_v$ and $\Leb^{cs}_v$ are defined similarly. In general, we will denote by $\Leb_W$ the Lebesgue (or Liouville in the context of unit tangent bundles) measure on a Riemannian manifold $W$.

\subsubsection{Absolute continuity}

\paragraph{Distortion controls.} The following lemma gives foliated versions of two classical distortion controls.

\begin{lemma}
\label{distortioncontrol}
There exist constants $C_0>0$ and $\alpha_0>0$ such that for every $v,w\in\hM$ lying on the same unstable leaf and for every $t>0$
\begin{equation}
\label{distortion1}
\frac{\Jac^u_{-t}(v)}{\Jac^u_{-t}(w)}\leq \exp\left(C_0\dist_{\hcF}(v,w)\right).
\end{equation}
and

\begin{equation}
\label{distortion2}
\frac{\Jac^{cs}_{-t}(v)}{\Jac^{cs}_{-t}(w)}\leq\exp\left( C_0\dist_{\hcF}(v,w)^{\alpha_0}\right).
\end{equation}

In particular, the left hand sides of \eqref{distortion1} and \eqref{distortion2} have limits as $t$ goes to $\infty$.
\end{lemma}

\begin{proof} We recall briefly the proofs of these distortion controls for the reader's commodity (the classical reference is \cite[Lemma 3.2, Chapter III]{M}). Let $n\in\N^{\ast}$ and $v,w\in\hM$ lying in the same unstable manifold.
$$
\left|\log\frac{\Jac^u_{-n}(v)}{\Jac^u_{-n}(w)}\right|\leq\sum_{i=0}^{n-1}|\log\Jac^u_{-1}(G_{-i}(v))-\log\Jac^u_{-1}(G_{-i}(w))|\leq C\sum_{i=0}^{n-1}\dist_{\F}(G_{-i}(v),G_{-i}(w)),$$
where $C$ denotes a Lipshitz constant of $\log\Jac_{-1}^u$ in unstable manifolds (we use here that this function is uniformly $C^1$ in unstable leaves: this comes from the fact that $E^u$ is smooth along unstable manifolds). Using that $v,w$ lie on the same unstable manifold, one sees that there is $C'>0$ and $\lambda>0$ such that the $\dist_{\F}(G_{-i}(v),G_{-i}(w))\leq C'\,e^{-i\lambda}\dist_{\F}(v,w)$ and the result follows since 
$\sum_{i\geq 0} e^{-i\lambda}<\infty$.

The second distortion control follows the same lines with a small change. We only know that $E^{cs}$ is H{\"o}lder continuous along unstable manifolds.
\end{proof}

\paragraph{Absolute continuity.} We say that an invertible map $\hh:W_1\to W_2$ between two smooth Riemannian manifolds is \emph{absolutely continuous} if both $\hh$ and $\hh^{-1}$ send sets of zero Lebesgue measure on sets of zero Lebesgue measure. In that case the \emph{Jacobian} of $\hh$ is well defined for $\Leb_{W_1}$-almost every $x$ as the Radon-Nikodym derivative
$$\Jac\hh(x)=\frac{d\left[\hh^{-1}_{\ast}\Leb_{W_2}\right]}{d\Leb_{W_1}}(x).$$

The proof of the next theorem will be sketched in \S \ref{absolouille}.

\begin{theorem}[Absolute continuity]
\label{absolutecontinuity}
Let $(M,\F)$ be a closed foliated manifold endowed with a negatively curved leafwise metric. Then the unstable holonomy maps are absolutely continuous, their Jacobians are defined everywhere and satisfy the following.
\begin{enumerate}
\item For all $v,w$ lying in the same unstable manifold, one has
\begin{equation}
\label{jacuhol}
\Jac\,\h uvw(v)=\lim_{t\to\infty}\frac{\Jac^{cs}_{-t}(v)}{\Jac^{cs}_{-t}(w)}.
\end{equation}
We refer to Figure \ref{figureabscont} located at the beginning of \S \ref{absolouille} for an illustration of the formula, and of the proof.
\item There exist uniform positive constants $C_1,\alpha_1>0$ such that for $v,v'$ lying on the same center-stable manifold and $w,w'$ lying respectively in $W^u(v),W^u(v')$ such that if
$$\dist_{cs}(v,v'),\dist_u(v,w),\dist_u(v',w')<1,$$
then we have
$$\left|\log\Jac\,\h uvw(v)-\log\Jac\,\h u{v'}{w'}(v')\right|\leq C_1\dist_{cs}(v,v')^{\alpha_1}.$$
\end{enumerate}
\end{theorem}

\begin{rem}
\label{othertransversals}
We stated the theorem for holonomy maps \emph{between local center-stable manifolds}. The same statement holds for all holonomy maps along unstable leaves $\hh:T_1\to T_2$, where $T_1,T_2$ are small smooth transversals of the unstable foliations \emph{included in the same leaf} (this is a small modification of the proof we present: one should follow Ma{\~n}{\'e}'s \cite[Lemma 3.7 Chap. III]{M}). Proceeding as in Appendix \ref{absolouille} one sees that the formula is

$$
\Jac\hh(v)=\lim_{t\to\infty}\frac{\Jac^{T_1}_{-t}(v)}{\Jac^{T_2}_{-t}(w)},
$$
where $v\in T_1$, $w=\hh(v)\in T_2$ and $\Jac^{T_i}_t$ denotes the restriction of $\det(DG_t)$ to the tangent space to $T_i$.

One should be careful with this notion of absolute continuity. We are looking at holonomy maps between transversals of $\W^u$ \emph{that are included in a same leaf of $\hcF$} and not between transversals of $\W^u$ \emph{as a foliation of} $\hM$. In fact this weak notion of absolute continuity would also hold in the \emph{laminated} context where the ambient space is not a manifold anymore (provided this space is compact).
\end{rem}

\begin{rem}
\label{stablecase}
By inversion of time, we get the analogue version for stable holonomies. Moreover since the flow is smooth in the leaves, holonomies along center-stable and center-unstable foliations are also absolutely continuous.
\end{rem}

\subsubsection{Local product structure of the Liouville measure}
\label{lpsliouv}

\paragraph{Natural measures in rectangles.} 
Let $v\in\hM$. Recall that we called rectangles the subsets of $\hL_v$ of the form $\RR_v=[W_{loc}^u(v),W_{loc}^{cs}(v)]$.

One can define a finite measure $m_v$ on the rectangle $\RR_v$ by requiring that for every Borel sets $A^u\dans W^u_{loc}(v)$ and $A^{cs}\dans W^{cs}_{loc}(v)$ the following relation holds
$$m_v[A^u,A^{cs}]=\Leb^u_v(A^u)\Leb^{cs}_v(A^{cs}).$$

This amounts to the same as defining $m_v$ on the rectangle $\RR_v$
\begin{itemize}
\item by integration against $\Leb_v^{cs}$ of the measures  defined on $W^u_{loc}(z)$ by $(\h {cs}v{z})_{\ast} \Leb^u_v$;
\item by integration against $\Leb_v^u$ of the measures defined on $W^{cs}_{loc}(w)$ by $(\h uv{w})_{\ast} \Leb^{cs}_v$.
\end{itemize}

By absolute continuity of the holonomy maps (see Theorem \ref{absolutecontinuity}), we find the following disintegrations of $m_v$ in local unstable and center-stable manifolds respectively:
\begin{eqnarray}
\label{dismx1}
dm_v&=&\left(\Jac\,\h {cs}zv\,d\Leb^u_z\right)\,d\Leb^{cs}_v(z)\\
\label{dismx2}
    &=&\left(\Jac\,\h uwv\,d\Leb^{cs}_w\right)\,d\Leb^{u}_v(w)
\end{eqnarray}

\paragraph{Angle function.}  Let $T\hcF=E\oplus F$ be a continuous splitting of the tangent space of $\hcF$.  Let $v\in\hM$. The foliated Sasaki metric induces a Euclidean structure on $T_v\hL_v$ which comes with a volume denoted by $\vol_v$. Let $P(v)$ be the parallelepiped of $T_v\hL_v$ spanned by the concatenation of an orthonormal basis of $E(v)$ and an orthonormal basis of $F(v)$ (these two spaces are supplementary in $T_v\hL_v$).

\begin{defi}[Angle function]
\label{anglefct}
The angle function of the splitting $T\hcF=E\oplus F$ is the function defined by the formula
$$\theta_{E,F}(v)=\vol_v\,P(v),\,\,\,\,\,\,\,\,\,\,\,\,\,\,\,v\in\hM.$$
\end{defi}

\begin{rem}
The number $\theta_{E,F}(v)$ is independent of the choices of the orthonormal bases of $E(v)$ and $F(v)$.  Moreover $\theta_{E,F}(v)\in(0,1]$ and is equal to $1$ if and only if the spaces $E(v)$ and $F(v)$ are orthogonal. If $E(v)$ and $F(v)$ were of dimension $1$, $\theta_{E,F}(v)$ would be the sine of the angle between $E(v)$ and $F(v)$.
\end{rem}

\paragraph{Decomposition of the Liouville measure.} The goal of this paragraph is to prove the following proposition, which is a direct consequence of the absolute continuity of the invariant foliations.

\begin{proposition}
\label{leblps}
The measure $m_v$ in $\RR_v$ has the following density with respect to $\Leb_{\hL_v}$
$$dm_v(w)=\left(\frac{1}{\theta_{E^u,E^{cs}}(w)}\,\Jac\,\h uwv(w)\,\Jac\,\h {cs}wv(w)\right)d\Leb_{\hL_v}(w).$$
\end{proposition}

The proof of this proposition follows from two lemmas.

\begin{lemma}
The following limit exists for every $v\in\hM$
$$\lim_{r\to 0}\,\,\,\frac{m_v(B_{\hcF}(v,r))}{\Leb_{\hL_v}(B_{\hcF}(v,r))}=\frac{1}{\theta_{E^u,E^{cs}}(v)}.$$
\end{lemma}

\begin{proof}
Let $v\in\hM$. We write, when $r$ is small enough
$$m_v(B_{\hcF}(v,r))=\int_{W^{cs}_r(v)}\left[\int_{B_{\hcF}(v,r)\cap W_{loc}^u(z)}\,\Jac\,\h {cs}zv(\zeta)\,d\Leb^u_z(\zeta)\right] d\Leb_v^{cs}(z).$$
Now, when $r$ tends to zero, we know that
\begin{itemize}
\item the Jacobian of the exponential map $\exp_v:\exp_v^{-1}( B_{\hcF}(v,r))\to B_{\hcF}(v,r)$ tends to $1$. Thus we can choose to work in the Euclidean space $T_v \hL_v$ with almost no volume distortion;
\item by continuity of the Jacobian of the center stable holonomy maps, $\Jac\, \h {cs}zv$ is uniformly close to $1$ when $z\in W^{cs}_r(v)$;
\item the Jacobian of $(\exp^{cs}_v)^{-1}: W_r^{cs}(v)\to E^{cs}(v)$ tends to $1$ thus we can assume that the integral above is taken on a Euclidean ball of radius $r$ inside $E^{cs}(v)$;
\item by continuity of unstable manifolds in the smooth topology, the Jacobians of all exponential maps $(\exp^u_z)^{-1}: W_r^u(z)\to E^u(z)$ are close to $1$, and the angle function between spaces $E^u(z)$ and $E^{cs}(v)$ (such as defined above) is close to $\theta_{E^u,E^{cs}}(v)$. Thus with little distortion, we may assume that the $E^u(z)$ are parallel with angle function $\theta_{E^u,E^{cs}}(v)$. 
\end{itemize}

Hence, consider in $T_v\hL_v$, a Euclidean structure obtained by requiring that the basis used to define $\Pp(v)$ is orthonormal (for that structure $E^u$ and $E^s$ are orthonormal). By definition the corresponding volume has density $1/\theta_{E^u,E^{cs}}(v)$ with respect to $\vol_v$.

By what precedes and by Fubini's theorem the previous integral is equivalent, when $r$ tends to zero, to the mass of a Euclidean ball of radius $r$ for the volume $1/\theta_{E^u,E^{cs}}(v)\vol_v$.

We then have the following limit
$$\lim_{r\to 0}\,\,\,\frac{m_v(B_{\hcF}(v,r))}{\Leb_{\hL_v}(B_{\hcF}(v,r))}=\frac{1}{\theta_{E^u,E^{cs}}(v)}.$$
\end{proof}

\begin{lemma}
Assume that $v,w$ belong to the same leaf, and that $\RR_v$ and $\RR_w$ intersect. Then $m_v$ and $m_w$ are equivalent on $\RR_v\cap\RR_w$. More precisely, if $\zeta\in\RR_v\cap\RR_w$, and if $\zeta_w^u=[w,\zeta]\in W^u_{loc}(w)$ and $\zeta_w^{cs}=[\zeta,w]\in W^{cs}_{loc}(w)$, we have
\begin{equation}
\label{floubidoubidoutroutrou}
\frac{dm_v}{dm_w}(\zeta)=\Jac\,\h uwv(\zeta^u_w)\Jac\,\h {cs}wv(\zeta^{cs}_w).
\end{equation}
\end{lemma}

\begin{figure}[hbtp]
\centering
\includegraphics[scale=0.8]{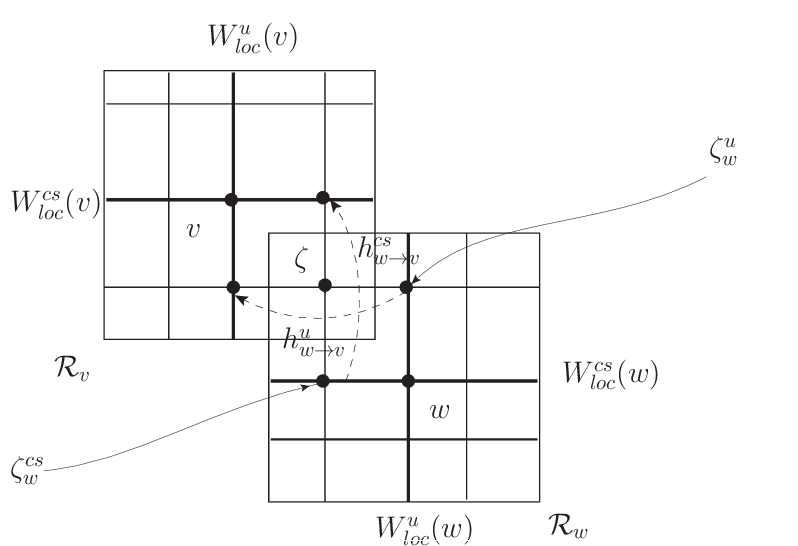}
\caption{Intersecting rectangles}
\label{f:Formouille}
\end{figure}

\begin{proof}
Suppose that $v,w$ satisfy the hypothesis of the lemma. Let $\zeta\in\RR_v\cap\RR_w$, $\zeta^u_w=[w,\zeta]$, $\zeta_w^{cs}=[\zeta,w]$ and consider two small discs $D^{u}\dans W^{u}_{\delta}(w)$ and $D^{cs}\dans W^{cs}_{\delta}(w)$ with same (small) radius and centered respectively at $\zeta^{u}_w$ and $\zeta^{cs}_w$.

The rectangle $D=[D^u,D^{cs}]$ is an open set containing $\zeta$. We have the following equality
$$\frac{m_v(D)}{m_w(D)}=\frac{\Leb^u_v(\h uwv(D^u))}{\Leb^u_w(D^u)}\frac{\Leb^{cs}_v(\h {cs}wv(D^{cs}))}{\Leb^{cs}_w(D^{cs})}.$$

When the common radius of the discs $D^u$ and $D^{cs}$ goes to zero

\begin{equation}
\label{Boreldensity}
\frac{\Leb^u_v(\h uwv(D^u))}{\Leb^u_w(D^u)}\frac{\Leb^{cs}_v(\h {cs}wv(D^{cs}))}{\Leb^{cs}_w(D^{cs})}\to \Jac\,\h uwv(\zeta_w^u)\Jac\,\h {cs}wv(\zeta_w^{cs}).
\end{equation}

The angle function between unstable and center-stable manifolds is uniformly bounded, and the unstable and center-stable holonomy maps are uniformly H\"older continuous. Thus as the common radius of the discs $D^u$ and $D^{cs}$ tends to zero, the open set $D$ shrinks nicely between two Riemannian balls such that the quotients of their radii are uniformly bounded. Hence, we can use the Borel density theorem (see \cite[Theorem 2.12]{Matti}) in order to get for $m_w$-almost every $\zeta\in\RR_v\cap\RR_w$

$$\frac{m_v(D)}{m_w(D)}\to\frac{dm_v}{dm_w}(\zeta).$$
By continuity of the right hand side of \eqref{Boreldensity} with respect to $\zeta$, the result follows.
\end{proof} 

\begin{proof}[Proof of Proposition \ref{leblps}]
Now, we can conclude the proof of Proposition \ref{leblps}. If $w\in\RR_v$, then
$$\frac{m_v(B_{\hcF}(w,r))}{\Leb_{\hL_v}(B_{\hcF}(w,r))}=\frac{m_w(B_{\hcF}(w,r))}{\Leb_{\hL_v}(B_{\hcF}(w,r))}\frac{m_v(B_{\hcF}(w,r))}{m_w(B_{\hcF}(w,r))}.$$
As $r$ tends to $0$, the first factor of the product above converges to $\theta_{E^u,E^{cs}}(w)^{-1}$ by the first lemma and, for Liouville-almost every $w$. The second factor converges to $\Jac\,\h uwv(w)\Jac\,\h {cs}wv(w)$ (since here $\zeta=w$, we have $\zeta^u_w=\zeta^{cs}_w=w$), and the proposition follows.
\end{proof}

\section{Gibbs $su$-states and transverse invariant measures}
\label{sugibbssection}
\subsection{Gibbs $u$-states for leafwise hyperbolic flows}

\begin{defi}[Gibbs $u$-states]
\label{Gibbsustates}
Let $(M,\F)$ be a closed foliated manifold endowed with a negatively curved leafwise metric. A Gibbs $u$-state for the foliated geodesic flow $G_t$ is a probability measure on $\hM$ which is invariant by the flow, and has Lebesgue disintegration for $\W^u$.

Similarly, a Gibbs $s$-state for $G_t$ is a Gibbs $u$-state for $G_{-t}$. If one prefers it is a $G_t$-invariant probability measure on $\hM$ which has Lebesgue disintegration for $\W^s$.

Finally, a Gibbs $su$-state for $G_t$ is a probability measure which is both a Gibbs $u$-state and a Gibbs $s$-state for $G_t$.
\end{defi}

\begin{rem}
\label{logbounded}
The conditional measures of a Gibbs $u$-state $\mu$ in unstable plaques are of the form $\psi^u_w\,d\Leb^u_w$ , where $\psi^u_w$ are densities defined in $W^u_{loc}(w)$. Of course one has to be careful: these densities depend on the local chart where we disintegrated $\mu$ (see Remarks \ref{disintegrationsuccessive} and \ref{psi1} below).

However the following theorem shows that for dynamical reasons, these densities $\psi^u_w$ have to be continuous, uniformly (independently of $w$) bounded away from $0$ and $\infty$ (such a function is said to be \emph{log-bounded} since its logarithm has to be bounded) and to verify a prescribed cocycle formula.
\end{rem}

\begin{theorem}
\label{reconstruirelesetatsdeugibbs}
Let $(M,\F)$ be a closed foliated manifold endowed with a negatively curved leafwise metric. Then

\begin{enumerate}
\item for every $v\in\hM$ and every Borel set $D^u\dans W^u_{loc}(v)$ with positive Lebesgue measure, any accumulation point of the family of measures , $\mu_t$, $t>0$ is a Gibbs u-state where for every $A\dans\hM$ $\mu_t(A)$ is defined by
$$\mu_{t}(A)=\frac{1}{t}\int_0^{t}\frac{\Leb^u(D^u\cap G_{-s}(A))}{\Leb^u(D^u)}ds.$$
Moreover its densities along unstable plaques, denoted by $\psi^u_w$, are uniformly log-bounded and satisfy the following for $z_1,z_2\in W^u_{loc}(w)$
\begin{equation}
\label{relationdensiteslocales}
\frac{\psi^u_w(z_2)}{\psi^u_w(z_1)}=\lim_{t\to\infty}\frac{\Jac^u_{-t}(z_2)}{\Jac^u_{-t}(z_1)};
\end{equation}
\item all ergodic components of a Gibbs u-state are Gibbs u-states with local densities in the unstable plaques that are uniformly log-bounded and satisfy \eqref{relationdensiteslocales};
\item every Gibbs u-state for $G_t$ is a measure whose local densities in the unstable plaques are uniformly log-bounded and satisfy \eqref{relationdensiteslocales}.
\end{enumerate}
\end{theorem}

\begin{proof}
We give the proof of this theorem in the appendix: see \S \ref{proofugibbs}.
\end{proof}

\begin{figure}[hbtp]
\centering
\includegraphics[scale=0.4]{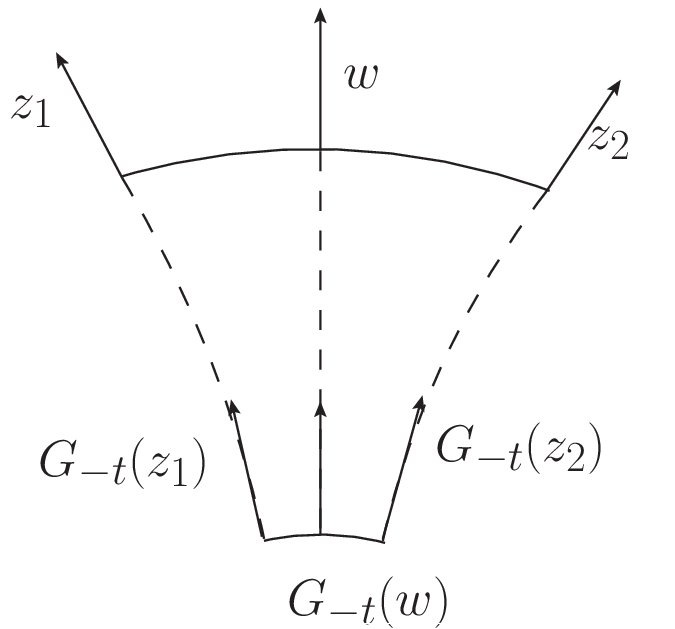}
\caption{Local densities of  Gibbs $u$-states}
\label{localdensitouille}
\end{figure}

\begin{rem}
\label{transversediscs}
The first item of Theorem \ref{reconstruirelesetatsdeugibbs} also holds if $D$ is a Borel subset of positive Lebesgue measure of a smooth disc included in a leaf of $\hcF$ which is transverse to $\W^{cs}$. We sketch the proof in the appendix: see Step 4 of \S \ref{proofugibbs}.
\end{rem}

\begin{rem} 
\label{disintegrationsuccessive}
In what follows, we will often consider charts of the form 
\begin{equation}
\label{specialcharts}
U=\bigcup_{v\in T}\RR_v,
\end{equation}
where $T$ is a transversal of $\hcF$ and $\RR_v$ are the rectangles defined in \S \ref{lpsliouv}. We refer to Remark \ref{transvcontrectangle} for a discussion on the transverse continuity of rectangles. Such charts trivialize the foliation $\hcF$ as well as the unstable foliation $\W^u$. Let $\mu$ be a Gibbs $u$-state and suppose $\mu(U)>0$. Denote by $(\mu_v)_{v\in T}$ the family of conditional measures of $\mu_{|U}$ in rectangles $\RR_v$ with respect to the projection of $\mu_{|U}$ on $T$ along plaques of $\hcF$ which we denote by $\nu_T$. By uniqueness of the disintegration (see Theorem \ref{Rokhlin}), for $\nu_T$-almost every $v$ the conditional measures of $\mu$ in unstable plaques of $\RR_v$ coincide with that of $\mu_v$.

\emph{To summarize}, in order to identify the densities of a Gibbs $u$-state in local unstable manifolds, we first disintegrate it in the plaques of $\hcF$, and then disintegrate the resulting conditional measures in the local unstable manifolds.
\end{rem}

\begin{rem}
\label{psi1}
Let $\mu$ be a Gibbs $u$-state and suppose $\mu(U)>0$ for some chart $U$ of the form \eqref{specialcharts}. Consider the family $(\mu_v)_{v\in T}$ of conditional measures in rectangles $\RR_v$ with respect to a transverse measure $\nu_T$. For $v$ which is $\nu_T$-typical (i.e. $v$ inside a Borel set full for $\nu_T$), $\mu_v$ is obtained by integrating against some measure $\eta^{cs}_v$ on $W^{cs}_{loc}(v)$ measures of the form $\psi^u_w\,\Leb^u_w$. Once again by uniqueness of the disintegration, the choice of $\psi^u_w$ is prescribed by the choice of the measure $\eta^{cs}_v$ against which one disintegrate $\mu_v$. By disintegrating against $d\nu^{cs}_v(w)=1/\psi^u_w(w)d\eta^{cs}_v(w)$ (recall that $\psi^u_w$ is uniformly log-bounded by Theorem \ref{reconstruirelesetatsdeugibbs}), we can always assume that $\psi_w^u(w)=1$ for $w\in W^{cs}_{loc}(v)$.

\emph{To summarize}, when we have a Gibbs $u$-state $\mu$, we first disintegrate its restriction to a foliated chart for $\hcF$ in the partition given by rectangles $\RR_v$ thus obtaining conditional measures $\mu_v$. This partition admits a subpartition by local unstable manifolds. We disintegrate the measures $\mu_v$ in local unstable manifolds so it has the form
\begin{equation}
\label{dis}
d\mu_v=\left(\psi^u_w \,d\Leb^u_w\right)\,d\nu^{cs}_v(w),
\end{equation}
with $\psi^u_w(w)=1$ or if one prefers, for $\zeta\in W_{loc}^u(w)$
\begin{equation}
\label{formulapsiu}
\psi^u_w(\zeta)=\lim_{t\to\infty}\frac{\Jac^u_{-t}(\zeta)}{\Jac^u_{-t}(w)}.
\end{equation}
\end{rem}

\begin{rem}
\label{custates}
A measure $\mu$ invariant by the flow induces the $dt$ element on flow lines. Hence, any Gibbs $u$-state for $G_t$ is a Gibbs $cu$-state, i.e. it has Lebesgue disintegration for $\W^{cu}$. Then there exist densities $\psi^{cu}_w$, which coincide with $\psi^{u}_w$ in the local unstable manifold of $w$ (in particular, $\psi^{cu}_w(w)=1$) such that $\mu_v$ disintegrates as follows in the local center unstable manifolds
$$d\mu_v=\left(\psi^{cu}_w \,d\Leb^{cu}_w\right)\,d\nu^{s}_v(w),$$
where $\nu^s_v$ is a finite measure on $W^s_{loc}(v)$. These local densities are determined by the following relation (see Figure \ref{localdensitouillecu}). If $\zeta\in W^{cu}_{loc}(w)$, and if $\tau$ is such that $G_{-\tau}(\zeta)\in W^u_{loc}(w)$, then
\begin{equation}
\label{formulapsicu}
\psi^{cu}_w(\zeta)=\lim_{t\to\infty}\frac{\Jac^u_{-t-\tau}(\zeta)}{\Jac^u_{-t}(w)}.
\end{equation}
\end{rem}

\begin{figure}[hbtp]
\centering
\includegraphics[scale=0.6]{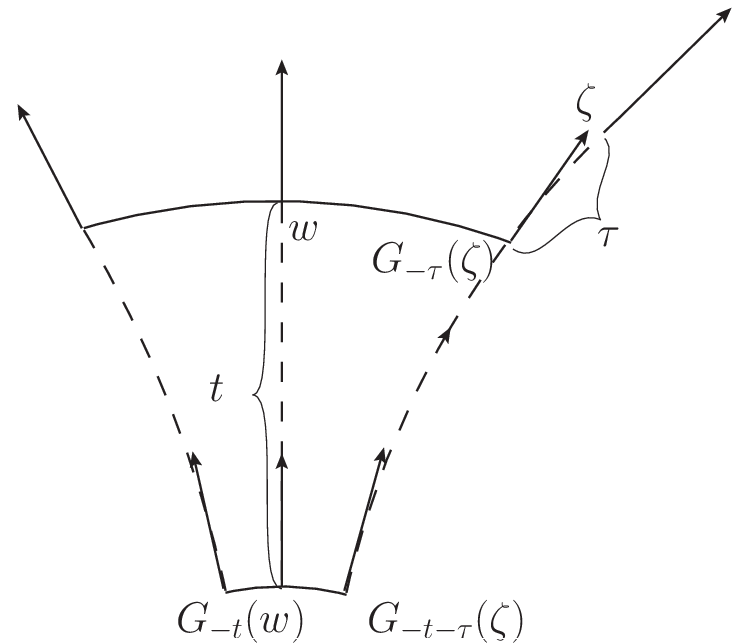}
\caption{Local densities of  Gibbs $cu$-states}
\label{localdensitouillecu}
\end{figure}

\subsection{A sufficient condition for existence of transverse invariant measures}

The goal of this Section is to prove Theorem \ref{mtheoremsugibbsgeodesique}. Let us briefly recall the hypothesis and what we want to prove. Here $(M,\F)$ is a closed foliated manifold endowed with a negatively curved leafwise metric. The foliated geodesic flow $G_t$ acts on $\hM$, the unit tangent bundle of $\F$.

We suppose the existence of a Gibbs $su$-state $\mu$ and we intend to prove that such a measure has to be totally invariant, i.e. has to be locally the product of the Liouville measure of the plaques times a transverse invariant measure for $\hcF$.

\paragraph{Strategy of the proof.} Following Remarks \ref{disintegrationsuccessive} and \ref{psi1}, we are going to disintegrate such a Gibbs $su$-state in rectangles $\RR_v$ and identify the densities in local unstable, and center-stable manifolds. Recall that rectangles are filled with local unstable manifolds, but also with local center-stable manifolds. By disintegrating $\mu_v$ both in unstable and center-stable local manifolds, and carefully analyzing the densities we will show that the conditional measures $\mu_v$ are \emph{multiples} of the Liouville measure. This will force the family of transverse measures induced by $\mu$ to be holonomy-invariant.

In what follows, we work with a good foliated atlas for $\hcF$ whose charts are of the form $U=\bigcup_{v\in T}\RR_v$ where $T\dans U$ is a transverse section of $\hcF$. We choose such a chart $U$ with $\mu(U)>0$ and denote by $(\mu_v)_{v\in T}$ the family of conditional measures of $\mu_{|U}$ in rectangles with respect to the transverse measure $\nu_T$.

\paragraph{Identifying the transverse measures.}
Let $v\in T$ be $\nu_T$-typical. Using Remarks \ref{psi1} and \ref{custates} there exist a measure $\nu_v^{cs}$ on $W^{cs}_{loc}(v)$, as well as a measure $\nu^u_v$ on $W^u_{loc}(v)$ such that $\mu_v$ disintegrates in $\RR_v$ as
\begin{eqnarray}
\label{nuu}
d\mu_v&=&(\psi_z^u\,d\Leb^u_z)\,d\nu_v^{cs}(z)\\
\label{nucs}
   &=&(\psi_w^{cs}\,d\Leb^{cs}_w)\,d\nu_v^u(w),
\end{eqnarray}
where $\psi^u_z$ is given by Formula \eqref{formulapsiu} and $\psi^{cs}_w$ is given by changing the role of $s$ and $u$ and reversing the time in Formula \eqref{formulapsicu}.

\begin{lemma}
\label{transleb}
The measures $\nu^{cs}_v$ and $\nu^u_v$ are respectively equivalent to $\Leb^{cs}_v$ and $\Leb^u_v$.
\end{lemma}

\begin{proof}
We only prove the fact that $\nu^u_v$ is equivalent to $\Leb^u_v$ in $W^u_{loc}(v)$. The other assertion follows from an analogous argument.

Note that by Remark \ref{psi1}, the projection of $\mu_v$ on $W^u_{loc}(v)$ along the center-stable plaques, which we denote by $\eta^u_v$, is equivalent to $\nu_v^u$. Hence it is enough to show that $\eta_v^u$ is equivalent to $\Leb^u_v$.

Using Formula \eqref{nuu} one sees that the measure $\eta^u_v$ defined above is equal to the $\nu_v^{cs}$-average of the projections on $W^u_{loc}(v)$  by center-stable holonomy maps of the measures $\psi^u_z\,d\Leb^u_z$ defined on $W^u_{loc}(z)$. 

By absolute continuity of center-stable holonomy maps, the latter projections on $W^u_{loc}(v)$ are equivalent to $\Leb^u_v$. Their $\nu_v^{cs}$-average must be as well, and the result follows.
\end{proof}

As a consequence there exist two measurable functions $f:W^u_{loc}(v)\to\R^{\ast}_+$ and $g:W^{cs}_{loc}(v)\to\R^{\ast}_+$ such that the disintegrations of $\mu_v$ in the local unstable and center-stable manifolds read as
\begin{eqnarray}
\label{dismuv1}
d\mu_v &=&\left(\psi^u_z\,d\Leb^u_z\right)\,g(z)d\Leb^{cs}_v(z)\\
\label{dismuv2}
     &=&\left(\psi^{cs}_w\,d\Leb^{cs}_w\right)\,f(w)d\Leb^u_v(w).
\end{eqnarray}

\paragraph{Conditional measures in rectangles are equivalent to Liouville.} We now use Lemma \ref{transleb} in order  to show that conditional measures of $\mu$ in rectangles are equivalent to the Liouville measure. The first step is to prove that they are equivalent to the natural measure $m_v$ defined in \S \ref{lpsliouv} and to identify the densities.

\begin{figure}[hbtp]
\centering
\includegraphics[scale=0.5]{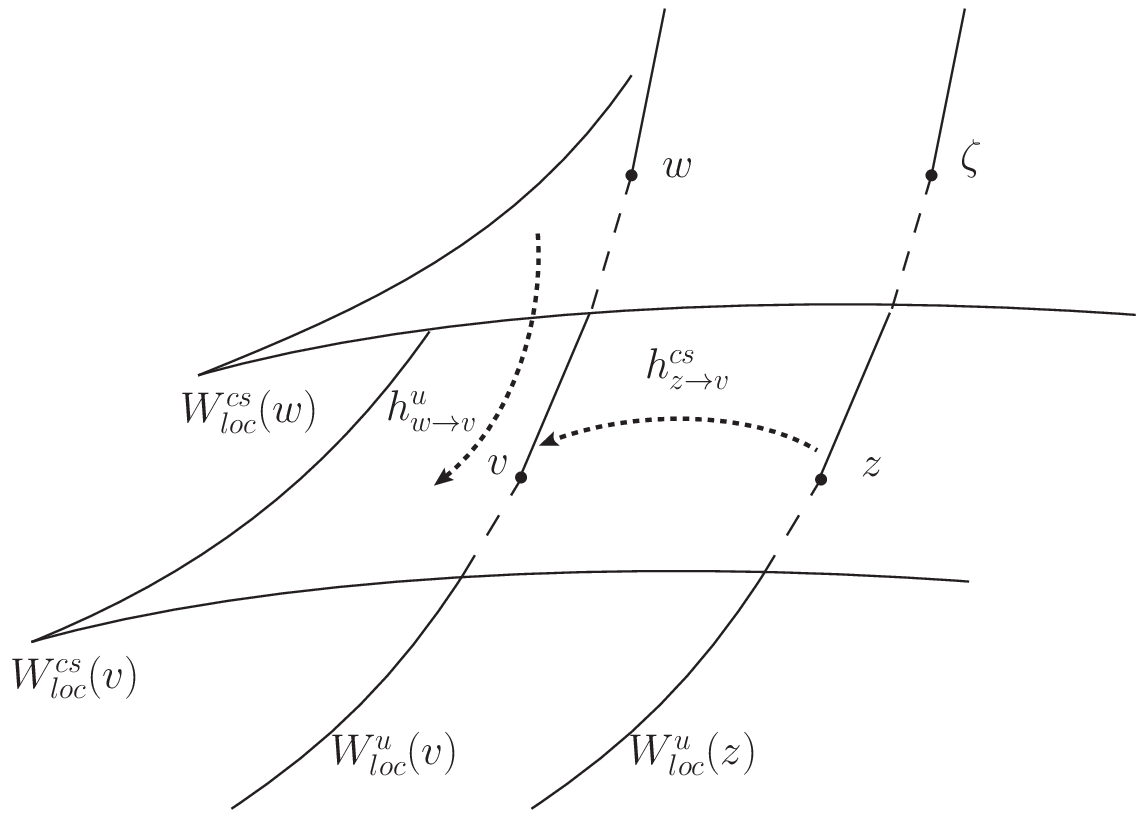}
\caption{Finding the densities}
\label{findthemyourself}
\end{figure}

\begin{lemma}
\label{equivliouville}
The measure $\mu_v$ is equivalent to $m_v$ in $\RR_v$. More precisely, if $\zeta\in\RR_v$, $w=[\zeta,v]\in W^u_{loc}(v)$, $z=[v,\zeta]\in W^{cs}_{loc}(v)$, and $F$ denotes the Radon-Nikodym derivative $d\mu_v/dm_v$, then
\begin{equation}
\label{dismux}
\frac{F(\zeta)}{F(v)}=\psi^u_v(w)\,\psi^{cs}_w(\zeta)\frac{\Jac\,\h uwv(w)}{\Jac\,\h uwv(\zeta)}.
\end{equation}
\end{lemma}

\begin{proof}
Using Disintegration Formulas \eqref{dismuv1} and \eqref{dismx1} we deduce that $\mu_v$ and $m_v$ have equivalent projections on $W^{cs}_{loc}(v)$ (they are both equivalent to Lebesgue) and that their conditional measures in the local unstable manifolds are also equivalent (to Lebesgue). As a consequence these measures have to be equivalent in $\RR_v$.

Let us be more precise and identify the density $F$. Comparing on the one hand Formulas \eqref{dismuv1} and \eqref{dismx1} and on the other hand Formulas \eqref{dismuv2} and \eqref{dismx2} we find that for all $\zeta\in\RR_v$, if $w=[\zeta,v]$ and $z=[v,\zeta]$

\begin{eqnarray}
\label{relationz}
F(\zeta)=\frac{d\mu_v}{dm_v}(\zeta)&=&\frac{\psi^u_z(\zeta)}{\Jac\,\h {cs}zv(\zeta)}g(z)\\
\label{relationy}                             
                          &=&\frac{\psi^{cs}_w(\zeta)}{\Jac\,\h uwv(\zeta)} f(w).
\end{eqnarray}
In particular the Radon-Nikodym derivative $F(v)$ satisfies

\begin{equation}
\label{initialisation}
g(v)=f(v)=\frac{d\mu_v}{dm_v}(v)=F(v).
\end{equation}
Note that we have used here (and that we will use below) the facts that $\psi^u_v(v)=1$ and $\psi^{cs}_v(v)=1$ (see Remarks \ref{psi1} and \ref{custates}). We find
$$\frac{f(w)}{g(z)}=\frac{\psi^u_z(\zeta)\Jac\,\h uwv(\zeta)}{\psi^{cs}_w(\zeta)\Jac\,\h {cs}zv(\zeta)},$$
and this relation holds for all $\zeta\in\RR_v$. If we choose $\zeta\in W^u_{loc}(v)$, then we have $\zeta=w$ and $z=v$. The previous equality then becomes
\begin{equation}
\label{identifdensity1}
\frac{f(w)}{g(v)}=\psi^u_v(w)\Jac\,\h uwv(w).
\end{equation}
We obtain Relation \eqref{dismux} for the normalized density by injecting Equalities \eqref{initialisation} and \eqref{identifdensity1} in \eqref{relationy}.
\end{proof}

\paragraph{Conditional measures in rectangles proportional to Liouville.} We now turn to the last step of the proof of Theorem \ref{mtheoremsugibbsgeodesique}. We will use Proposition \ref{leblps}, which gives the densities of the measures $m_v$  defined in \S \ref{lpsliouv}, as well as the crucial fact that the Liouville measures of the leaves are preserved by the foliated geodesic flow.

\begin{proposition}
\label{equalliouville}
Let $\mu$ be a Gibbs $su$-state for $G_t$ and let $U=\bigcup_{v\in T}\RR_v$ be a foliated chart for $\hcF$ with $\mu(U)>0$. Then the conditional measures $\mu_v$ in rectangles $\RR_v$ coincide with a multiple of the Liouville measure.
\end{proposition}

\begin{proof}
It follows from Lemma \ref{equivliouville} that the conditional measures $\mu_v$ of a Gibbs $su$-state in rectangles $\RR_v$ are equivalent to $m_v$ and hence to $\Leb_{\hL_v}$. Moreover the following cocycle relation holds $\Leb_{\hL_v}$-almost everywhere: $d\mu_v/d\Leb_{\hL_v}=d\mu_v/dm_v\,\,.\,\,dm_v/d\Leb_{\hL_v}$. In other words, if $H$ denotes the density $d\mu_v/d\Leb_{\hL_v}$ it coincides almost everywhere with the product of two factors. The first one is the Radon Nikodym derivative $F$ obtained in Lemma \ref{equivliouville}. The second one is the density obtained in Proposition \ref{leblps}.

In particular for every $\zeta\in\RR_v$ we obtain

$$
\frac{H(\zeta)}{H(v)}=\frac{\theta_{E^u,E^{cs}}(v)}{\theta_{E^u,E^{cs}}(\zeta)}\psi^{u}_v(w)\,\psi^{cs}_w(\zeta)\Jac\,\h {cs}zv(\zeta)\,\Jac\,\h uwv(w)\\
$$
where $\theta_{E^u,E^{cs}}$ is the angle function defined in \S \ref{lpsliouv}. Now we are going to use the precise expressions of local densities of Gibbs $u$-states and of the Jacobians of the holonomy maps given respectively in Theorem \ref{reconstruirelesetatsdeugibbs} (or more precisely in Remark \ref{psi1}) and in Theorem \ref{absolutecontinuity}. We have

$$\psi^u_v(w)=\lim_{t\to\infty}\frac{\Jac^u_{-t}(w)}{\Jac^u_{-t}(v)},$$
and
$$\Jac\,\h uwv(w)=\lim_{t\to\infty} \frac{\Jac^{cs}_{-t}(w)}{\Jac^{cs}_{-t}(v)}.$$

Since the foliated geodesic flow preserves the Liouville measure of the leaves we have for every $w\in\hM$, $\Jac_t(w)=\det(D_wG_t))_{|T_w\hL_w}=1$. Using the definition of the angle function $\theta_{E^u,E^{cs}}$ and the preservation of $E^u$ and $E^{cs}$ by $DG_{-t}$, we find

$$1=\Jac_{-t}(w)=\frac{\theta_{E^u,E^{cs}}(G_{-t}(w))}{\theta_{E^u,E^{cs}}(w)}\Jac^u_{-t}(w)\Jac^{cs}_{-t}(w).$$

But, $\log\theta_{E^u,E^{cs}}$ is uniformly H\"older continuous in the leaves of $\hcF$ because $\theta_{E^u,E^{cs}}$ is uniformly log-bounded and uniformly H\"older in the leaves. Hence when $v$ and $w$ lie in the same unstable manifold we have that $\lim_{t\to\infty}\theta_{E^u,E^{cs}}(G_{-t}(w))/\theta_{E^u,E^{cs}}(G_{-t}(v))=1$. We deduce from what precedes that for $v$ and $w$ lying on the same local unstable manifold

$$1=\lim_{t\to\infty}\frac{\Jac_{-t}(w)}{\Jac_{-t}(v)}=\frac{\theta_{E^u,E^{cs}}(v)}{\theta_{E^u,E^{cs}}(w)}\psi^u_v(w)\Jac\,\h uwv(w).$$

Similarly we prove that if $w$ lies on the same local center stable manifold as $\zeta$ then the following holds true
$$1=\frac{\theta_{E^u,E^{cs}}(w)}{\theta_{E^u,E^{cs}}(\zeta)}\psi^{cs}_w(\zeta)\Jac\,\h {cs}zv(\zeta).$$

Multiplying these two equalities, it comes that $H(\zeta)=H(v)$ for every $\zeta\in\RR_v$.
\end{proof}

\paragraph{End of the proof of Theorem \ref{mtheoremsugibbsgeodesique}.} We are now going to prove that the Gibbs $su$-state $\mu$ is totally invariant.

Consider a good foliated atlas for $\hcF$ whose charts are union of rectangles (see \eqref{specialcharts}). Let $U_i$ be such a chart  with $\mu(U_i)>0$ and call the corresponding transversal $T_i$. By Proposition \ref{equalliouville} there is a finite measure $\nu_i$ on $T_i$ such that the conditional measures of $\mu_{|U_i}$ with respect to $\nu_i$ in rectangles $\RR_v$ are given by $H_i(v)\Leb_{\hL_v}$ for some positive measurable function $H_i$. By uniqueness of the disintegration, if one considers the measure $\nu_i'=H_i(v)\nu_i(v)$, one sees that the conditional measures of $\mu$ with respect to $\nu_i'$ are precisely given by $\Leb_{\hL_v}$. In other words,  the disintegration of $\mu_{|U_i}$ in the plaques of $U_i$ with respect to $\nu_i'$ reads as \eqref{disintegration} with $h_i=1$.

We claim that the family of measures $(\nu_i')$ is holonomy-invariant: this implies that $\mu$ is totally invariant. Indeed this comes from the fact that if $\mu(U_i\cap U_j)>0$, one obtains Identity \eqref{quasiinvariantemesureabsolumtcont} where the densities are equal to $1$. This implies that $\tau_{ij\ast}^{-1}\nu_j'=\nu'_i$ as claimed. The proof of Theorem \ref{mtheoremsugibbsgeodesique} is now over.\quad \hfill $\square$

\section{Gibbs $u$-states for the foliated geodesic flow and $\phi^u$-harmonic measures}
\label{phiuharmonicsection}

\subsection{Sphere at infinity and Busemann cocycle}
\label{sphereinfty}

In this section, $(M,\F)$ still stands for a closed foliated manifold endowed with a negatively curved leafwise metric: recall that it implies that the sectional curvatures of all leaves are pinched between two negative constants $-b^2\leq-a^2<0$. If $L$ is a leaf of $\F$, then $\tL$ denotes its universal cover.

\paragraph{Sphere at infinity.} The space $\tL(\infty)$ represents the \emph{sphere at infinity} of $\tL$, that is to say, the set of equivalence classes of geodesic rays for the relation ``stay at bounded distance''. Say that a geodesic ray \emph{points to} $\xi\in\tL(\infty)$ if $\xi$ is its equivalence class. Since $\pi_1(L)$ acts on $\tL$ by isometry, it acts by sending equivalent geodesic rays on equivalent geodesic rays: we have a natural action of $\pi_1(L)$ on $\tL(\infty)$.

Given $z\in\tL$, we denote by $\widehat{\pi}_z:T^1_z\tL\to \tL(\infty)$ the natural projection which associates to $v$ the class of the geodesic ray it determines. Since the curvature of $\tL$ is pinched between two negative constants, all transition maps 
\begin{equation}
\label{transitionhat}
\widehat{\sigma}_{y,z}=\widehat{\pi}_{z}^{-1}\circ\widehat{\pi}_{y}:T^1_{y}\tL\to T^1_{z}\tL
\end{equation}
 are H\"older continuous (see \cite[Proposition 2.1]{AS}): there is a well defined H\"older class on $\tL(\infty)$. Note that these maps are holonomy maps along the center-stable foliation between unit tangent fibers.

\begin{rem}
\label{remvisibility}
In this paper, we are more interested in the unstable foliation than in the stable one. Define the \emph{flip function} as the involution $\iota:\hM\to\hM$ which associates to $(x,v)\in\hM$ the element $(x,-v)$.  This function conjugates the flows $G_t$ and $G_{-t}$, preserves the unit tangent fibers and exchanges stable and unstable manifolds.

Consider the projection maps $\pi_z=\widehat{\pi}_z\circ\iota:T^1_z\tL\to\tL(\infty)$ as well as the transition maps $\sigma_{y,z}:T^1_y\tL\to T^1_z\tL$. Properties of $\iota$ show that $\sigma_{y,z}=\pi^{-1}_z\circ\pi_y$ are holonomy maps along the center-unstable foliation between unit tangent fibers.
\end{rem}

\paragraph{Busemann cocycle and horospheres.}We define the \emph{Busemann cocycle} by the following
$$\beta_{\xi}(y,z)=\lim_{t\to\infty}\dist(c(t),z)-\dist(c(t),y),$$
where $\xi\in \tL(\infty)$, $y,z\in\tL$ and $c$ is \emph{any geodesic ray} parametrized by arc length and pointing to $\xi$ (recall that two geodesic rays pointing to the same limit become exponentially close at infinity). The Busemann cocycle is a smooth function of $(y,z)$ and a H\"older continuous function of $\xi$.

The \emph{horospheres} are the level sets of this cocycle: two points $y,z$ are said to be on the same horosphere centered at $\xi$ if $\beta_{\xi}(y,z)=0$. Horospheres are smooth manifolds and when endowed with the normal vector field pointing outwards  (resp. inwards) they provide the unstable (resp. stable) manifolds of the geodesic flow.

\subsection{Gibbs kernel}
\label{gibbskernelsection}

\paragraph{Potential.} Following the classical notation (see \cite{BR}) we define the \emph{potential} $\phi^u$ as the infinitesimal volume change rate in the unstable direction. It is defined by the following formula
\begin{equation}
\label{Eq:potential}
\phi^u(v)=-\left.\frac{d}{dt}\right|_{t=0}\log\Jac^u_t(v),
\end{equation}
which is well defined because $\Jac^u G_t(v)$ varies smoothly with $t$ (due to the chain rule). The potential $\phi^u$ is continuous in $\hM$ and varies H\"older continuously in the leaves of $\hcF$. This is due to the fact that $G_t$ is smooth in the leaves and varies continuously transversally in the smooth topology, as well as from the fact that $E^u$ is continuous in $\hM$ and H\"older continuous in the leaves of $\hcF$ (see \cite[Section 4]{BR} for precise justifications).

\paragraph{Gibbs kernel.} The restriction to $\hL$ of the potential lifts as a bounded and H\"older continuous function in $T^1\tL$ denoted by $\widetilde{\phi^u}$. This allows one to define the \emph{Gibbs kernel} as the following function of $(y,z,\xi)\in\tL\times\tL\times\tL(\infty)$

\begin{equation}
\label{gibbskernel}
k^u(y,z;\xi)=\exp\left[\int_{\xi}^{z}\widetilde{\phi}^u-\int_{\xi}^{y}\widetilde{\phi}^u\right].
\end{equation}
Here the difference of the integrals has the following meaning. 
$$\int_{\xi}^{z}\widetilde{\phi}^u-\int_{\xi}^{y}\widetilde{\phi}^u=\lim_{t\to\infty}\left(\int_{c(t)}^{z}\widetilde{\phi}^u-\int_{c(t)}^{y}\widetilde{\phi}^u\right).$$
where $c$ is a geodesic ray asymptotic to $\xi$ and $\int_{c(t)}^z\widetilde{\phi}^u$ denotes the integral of the potential on the \emph{directed geodesic} segment starting at $c(t)$ and ending at $z$. The limit exists by the usual distortion argument because $\widetilde{\phi}^u$ is H\"older in $T^1\widetilde{L}$ (see the proof of \cite[Lemma 3.2, Chapter III]{M}). Moreover, as for the limit used in the definition of the Busemann cocycle, this limit does not depend on the geodesic ray ending at $\xi$. A simple computation based on the chain rule $\Jac^u_{s+t}=\Jac^u_{t}\circ G_{s}\,.\,\Jac^u_{s}$ shows the following (see for example \cite[Section 4]{BR}).

\begin{lemma}
\label{kerneljacob1}
For every triple $(y,z,\xi)\in \widetilde{L}\times\widetilde{L}\times\widetilde{L}(\infty)$ we have the following equality (see also Figure  \ref{bordel})
$$k^u(y,z;\xi)=\lim_{t\to\infty}\frac{\Jac^u_{-t-\beta_{\xi}(y,z)}(v_{\xi,z})}{\Jac^u_{-t}(v_{\xi,y})},$$
where $v_{\xi,z}$ denotes the unit vector based at $z$ such that $\lim G_{-t}(v_{\xi,z})=\xi$ as $t\to\infty$ (the same notation is used for $v_{\xi,y}$).
\end{lemma}

\begin{rem}
\label{remcocyclerelation}
Remark that the Gibbs kernel satisfies the following cocycle relation for all $x,y,z\in\tL$ and $\xi\in\tL(\infty)$
\begin{equation}
\label{cocyclerelation}
k^u(x,y;\xi)k^u(y,z;\xi)=k^u(x,z;\xi).
\end{equation}
It also satisfies the following equivariance relation for every $y,z\in\tL$, $\xi\in\tL(\infty)$ and $\gamma\in\pi_1(L)$
\begin{equation}
\label{equivariancerelation}
k^u(\gamma y,\gamma z;\gamma\xi)=k^u(y,z;\xi)
\end{equation}

\end{rem}

\begin{rem}
\label{kerneljacob2}
What follows will be useful in the sequel. Let $v,w\in\hM$ be two vectors lying in the same center-unstable manifold. Call $L$ the leaf containing their basepoints. Consider lifts $\tilde{v},\tilde{w}\in T^1\tL$ lying inside the same fundamental domain for the action of $\pi_1(L)$. Denote by $y,z$ the respective basepoints of $\tilde{v}$ and $\tilde{w}$.

Let $c_{\tilde{v}}$, and $c_{\tilde{w}}$ be the geodesic directed respectively by $\tilde{v}$ and $\tilde{w}$. Let $\xi=c_{\tilde{v}}(-\infty)=c_{\tilde{w}}(-\infty)$ where the latter notation is used for the common past extremity of geodesics $c_{\tilde{v}}$ and $c_{\tilde{w}}$ (recall that $\tilde{v}$ and $\tilde{w}$ belong to the same center-unstable manifold). Then by Lemma \ref{kerneljacob1} and Formula \eqref{formulapsicu}
$$\psi^{cu}_v(w)=k^u(y,z;\xi).$$
Note that by Equivariance Relation \eqref{equivariancerelation} $k^u(y,z;\xi)$ does not depend of the lifts of $v$ and $w$. We will sometimes use the abusive but convenient notation $k^u(v,w;c_v(-\infty))$.
\end{rem}

\begin{rem}
\label{kernelhyperbolic}
When the sectional curvature of $L$ is constant equal to $-1$ we have for every $v\in\hL$, $\phi^u(v)=-(d-1)$ in such a way that  $k^u$ coincide with the \emph{Poisson kernel}
$$k(y,z;\xi)=\exp\left[-(d-1)\beta_{\xi}(y,z)\right].$$
\end{rem}

\subsection{$\phi^u$-harmonic measures}
\label{phiumeasuresection}

\paragraph{$\phi^u$-harmonic functions.} We are going to define a class of positive functions of $\tL$ with an integral representation similar to the \emph{Poisson representation} of harmonic functions in negative curvature (see \cite{AS}).

A positive function $h:\tL\to(0,\infty)$ is said to be \emph{$\phi^u$-harmonic} if there exists a finite Borel measure $\eta_o$ on $\tL(\infty)$ such that for every $z\in\tL$
\begin{equation}
\label{phiuharmonic}
h(z)=\int_{\tL(\infty)}k^u(o,z;\xi) d\eta_o(\xi),
\end{equation}
$o\in\tL$ being a base point. A positive function of $L$ is said to be $\phi^u$-harmonic if it lifts as a $\phi^u$-harmonic function of $\tL$. 

\begin{rem}
The notion of $\phi^u$-harmonic function is independent of the choice of the base point $o$. If $o'$ is another point of $\tL$, if $\eta_o$ is a finite Borel measure on $\tL(\infty)$ and if $h$ is the corresponding $\phi^u$-harmonic function, then we can write for every $z\in\widetilde{L}$
$$h(z)=\int_{\widetilde{L}(\infty)}k^u(o',z;\xi)d\eta_{o'}(\xi),$$
where $d\eta_{o'}(\xi)=k^u(o,o';\xi)d\eta_o(\xi)$.
\end{rem}

The next proposition follows directly from Remark \ref{kernelhyperbolic}.
\begin{proposition}
\label{phiunormal}
Assume that $L$ is a hyperbolic manifold. Then $\phi^u$-harmonic functions coincide with harmonic function (for the Laplace operator).
\end{proposition}

\paragraph{A natural $\phi^u$-harmonic function on the leaves.} We want to show the existence of a natural positive function on $M$ which is continuous and whose restriction to any leaf $L$, is $\phi^u$-harmonic. This will allow us to define the notion of \emph{totally invariant $\phi^u$-harmonic measure}, and to show that any transverse holonomy invariant measure gives rise to such a totally invariant measure.

Before we state the result we introduce a family $(E_t)_{t\geq 0}$ of subbundles of $T\hcF$. Let $L$ be a leaf of $\F$ and let $v\in\hL$ be based at $x\in L$. First set $E(v)=E_0(v)=T_vT^1_xL$. By Lemma \ref{transvouille} (it was stated with $E^{cs}$ instead of $E^{cu}$, but a symmetric argument applies) we have
$$T\hcF=E\oplus E^{cu}.$$

For $t\geq 0$ define the bundle $E_t$ by requiring that $E_t(G_t(v))=D_{x}G_t\,E(v)$. Note in particular that for every $t\geq 0$ we have
$$T\hcF=E_t\oplus E^{cu}.$$

\begin{proposition}
\label{trouverunefctharmonique}
Let $(M,\F)$ be a closed foliated manifold endowed with a negatively curved leafwise metric. Then the map $h_0:M\to\R$ which associates to $x\in M$ the number
$$h_0(x)=\int_{T_x^1\F}\theta_{E,E^{cu}}(v)\,d\Leb_{T^1_x\F}(v),$$
is continuous on $M$ and, when restricted to the leaves of $\F$, is $\phi^u$-harmonic.
\end{proposition}

The first step in proving this proposition is the following lemma.

\begin{lemma}
\label{holonomiecentreinstablesphereasphere}
Let $T_1$, $T_2$ be two open transverse sections to $\W^{cu}$ which are included in some unit tangent fibers. Assume that there exists a holonomy map along $\W^{cu}$, $\hh:T_1\to T_2$, which is a homeomorphism. Then for all $w\in T_2$ and $v=\hh^{-1}(w)\in T_1$:
$$\frac{d\left[\hh_{\ast}\Leb_{T_1}\right]}{d\Leb_{T_2}}(w)=\frac{\theta_{E,E^{cu}}(w)}{\theta_{E,E^{cu}}(v)}k^u(w,v;\xi),$$
where $\xi=c_{w}(-\infty)$ (we use here the notations of Remark \ref{kerneljacob2}).
\end{lemma}

\begin{proof}
Let $T_1$, $T_2$ be two transverse sections to $\W^{cu}$, as stated in the lemma: they are open subsets of two unit tangent fibers, and there is a holonomy map along center-unstable leaves $\hh:T_1\to T_2$, which is a homeomorphism.

By Theorem \ref{absolutecontinuity} we know that this map is absolutely continuous and we know its Jacobian. We have, for all $w\in T_2$, and $v=\hh^{-1}(w)$

\begin{equation}
\label{jacobouille}
\frac{d\left[\hh_{\ast}\Leb_{T_1}\right]}{d\Leb_{T_2}}(w)=\lim_{t\to\infty}\frac{\Jac^{T_2}_{-t-\beta_{\xi}(v,w)}(w)}{\Jac^{T_1}_{-t}(v)},
\end{equation}
where $\Jac^{T_i}_t$ stands for the restriction of $\det(DG_t)$ to the tangent space to $T_i$ and $\xi=c_{w}(-\infty)$ (we have used the same abusive notation for the Busemann function as explained  in Remark \ref{kerneljacob2}).

\begin{rem}
Note that the definition of the Jacobian, one has $d[\hh_{\ast}\Leb_{T_1}]/d\Leb_{T_2}=\Jac\hh^{-1}$, where $\hh^{-1}$ is a holonomy map between $T_2$ and $T_1$: Formula \eqref{jacobouille} follows (see also Remark \ref{othertransversals}).
\end{rem}

Let us emphasize the following fact that we shall use later. We used that if $v$ and $w$ lie in the same center-unstable manifold then $G_{-\beta_{\xi}(v,w)}(w)$ and $v$ lie in the same unstable manifold, and the limit \eqref{jacobouille} makes sense.

Using the invariance of the Liouville measure by the geodesic flow inside the leaves, we get for all $w\in T_2$ and $t\in\R$

$$1=\frac{\theta_{E_t,E^{cu}}(G_t(w))}{\theta_{E,E^{cu}}(w)}\Jac^{T_2}_t(w)\Jac^{cu}_t(w),$$

By definition the unstable distribution $E^u$ is the space of unstable Jacobi fields which are everywhere orthogonal to the geodesics: we refer to \cite[p.90]{Ba} for more details. This implies that the vector field generating the geodesic flow of $\hL$ is orthogonal to the distribution $E^u$. Moreover the geodesic flow acts by isometries on the orbits. Hence we have for all $t\in\R$, $\Jac^{cu}_t(w)=\Jac^u_t(w)$.

Finally, using the $\lambda$-lemma, i..e. Theorem \ref{lambdalemma} (reversing the time and exchanging the roles of $s$ and $u$), it comes that $E_{-t}(G_{-t}(w))$ approaches $E^s(G_{-t}(w))$ as $t$ tends to $\infty$. In particular we have

$$\lim_{t\to\infty}\frac{\theta_{E_{-t},E^{cu}}(G_{-t}(w))}{\theta_{E^s,E^{cu}}(G_{-t}(w))}=1.$$

Now note that $\log\theta_{E^s,E^{cu}}$, is H\"older continuous in leaves of $\hcF$ (see the proof of Proposition \ref{equalliouville}). If $v=\hh^{-1}(w)\in W^{cu}(w)$, then $v$ and $G_{-\beta_{\xi}(w,v)}(w)$ are in the same unstable manifold. Hence we deduce that
$$\lim_{t\to\infty}\frac{\theta_{E^s,E^{cu}}(G_{-t}(v))}{\theta_{E^s,E^{cu}}(G_{-t-\beta_{\xi}(v,w)}(w))}=1.$$
It follows that

$$\lim_{t\to\infty}\frac{\theta_{E_{-t},E^{cu}}(G_{-t}(v))}{\theta_{E_{-t-\beta_{\xi}(v,w)},E^{cu}}(G_{-t-\beta_{\xi}(w,v)}(w))}=1.$$

Putting all this together we find
\begin{eqnarray*}
\lim_{t\to\infty}\frac{\Jac^{T_2}_{-t-\beta_{\xi}(v,w)}(w)}{\Jac^{T_1}_{-t}(v)}&=&\frac{\theta_{E,E^{cu}}(w)}{\theta_{E,E^{cu}}(v)}\lim_{t\to\infty}\frac{\Jac^u_{-t}(v)}{\Jac^u_{-t-\beta_{\xi}(v,w)}(w)}\\
                                                                                     &=&\frac{\theta_{E,E^{cu}}(w)}{\theta_{E,E^{cu}}(v)}k^u(w,v;\xi),
\end{eqnarray*}

which proves the lemma.
\end{proof}

\begin{proof}[Proof of Proposition \ref{trouverunefctharmonique}]
Define on the unit tangent fiber $T^1_x\F$, $x\in M$, the measure $\omega_x$ with a density $\theta_{E,E^{cu}}(v)$ with respect to the Lebesgue measure. Note that $h_0(x)$ is precisely the total mass of $\omega_x$.

Consider a pair of transverse sections of the center-unstable foliation $T_1,T_2$ which are open sets of some unit tangent fibers, and such that there is a holonomy map $\hh:T_1\to T_2$ which is a homeomorphism. Let $w\in T_2$ be a vector based at $z\in M$ and let $v=\hh^{-1}(w)$ be based at $y\in M$: the two basepoints belong to a same leaf $L$. Then by Lemma \ref{holonomiecentreinstablesphereasphere} we get
\begin{equation}
\label{relholonomy}
\frac{d\left[\hh_{\ast}\omega_{y}\right]}{d\omega_{z}}(w)=k^u(z,y;\xi),
\end{equation}
where $\xi=c_{w}(-\infty)$ and we made the following abusive notation explained in Remark \ref{kerneljacob2}.

Now we can naturally lift the family $(\omega_x)_{x\in L}$ of measures on $T_x^1L$ to the universal cover so as to obtain a family of measures on $T_z^1\widetilde{L}$ still denoted by $(\omega_z)_{z\in\widetilde{L}}$. Consider the maps
$\pi_z$ and $\sigma_{y,z}$ defined in Remark \ref{remvisibility}. Recall that $\pi_z$ maps a vector $v\in T_z^1\widetilde{L}$ on the \emph{past} extremity of the geodesc it directs, and that $\sigma_{y,z}$ is a holonomy map between $T^1_y\widetilde{L}$ and $T^1_z\widetilde{L}$ along center-unstable manifolds. Given $z\in\tL$ define $\eta_z=\pi_{z\ast}\omega_z$, which is a measure on $\tL(\infty)$. Using Relation \eqref{relholonomy} we see that it is a family of equivalent measures which satisfy

\begin{equation}
\label{mesuresinfiniteugibbs}
\frac{d\eta_{z}}{d\eta_{y}}(\xi)=\frac{d\omega_{z}}{d[\sigma_{y,z\ast}\omega_{y}]}(\pi_{z}^{-1}(\xi))=k^u(y,z;\xi).
\end{equation}
In other words the Gibbs kernel is realized as the Radon-Nikodym cocycle of the family $(\eta_z)_{z\in\tL}$. In particular integrating Relation \eqref{mesuresinfiniteugibbs} against $\eta_{y}$ we obtain that the function $z\mapsto\mass(\eta_z)$ is $\phi^u$-harmonic. By projecting this function down to $L$ one obtains precisely the function $(h_0)_{|L}$, showing that this function is $\phi^u$-harmonic.

In order to conclude the proof, it remains to prove that $h_0$ is continuous in $M$. Note that $\theta_{E,E^{cu}}$ is continuous in $\hM$ (as are unit tangent fibers and center-unstable manifolds). Note moreover that the metric, and therefore the volume element $x\mapsto \Leb_{T^1_x\F}$, varies continuously. More precisely, let $U$ be a local foliated chart for $\F$ which trivializes the unit tangent bundle. In smooth coordinates $pr^{-1}(U)\simeq U\times S$ where $S$ is a sphere endowed with a volume form $\vol$. In these coordinates $\Leb_{T^1_x\F}$ has a smooth density with respect to $\vol$. This density varies continuously with $x$ in the smooth topology.

Finally $h_0(x)$ being the integral of $\theta_{E,E^{cu}}$ against $\Leb_{T^1_x\F}$, it varies continuously with $x$.
\end{proof}

\paragraph{$\phi^u$-harmonic measures.} Now we can define the notion of $\phi^u$-harmonic measure for foliations with negatively curved leaves just as we defined harmonic measures (see Definition \ref{harmonicmeasure}).

\begin{defi}[$\phi^u$-harmonic measures]
Let $(M,\F)$ be a closed foliated manifold endowed with a negatively curved leafwise metric. A probability measure $m$ on $M$ is said to be $\phi^u$-harmonic if it has Lebesgue disintegration for $\F$, and if the local densities are $\phi^u$-harmonic functions.
\end{defi}
The question of existence of these measures will be treated in the next paragraph. We first show how transverse invariant measures give rise to canonical $\phi^u$-harmonic measures.

\paragraph{Totally invariant $\phi^u$-harmonic measures.} We said that when $\F$ possesses a transverse measure invariant by holonomy, we can form a harmonic measure by combining it with the volume inside the leaves.

Similarly, we can form a $\phi^u$-harmonic measure by combining it with the measure whose density with respect the volume of the leaves is given by the $\phi^u$-harmonic function $h_0$.

\begin{defi}[Totally invariant measures]
\label{totallyinvphiu}
Let $(M,\F)$ be a closed foliated manifold endowed with a negatively cured leafwise metric. Let $\A=(U_i,\phi_i)$ be a good foliated atlas and $(T_i)_{i\in I}$ be an associated complete system of transversals. A $\phi^u$-harmonic measure $m$ on $M$ is said to be totally invariant if there exists a holonomy-invariant family of transverse measures $(\nu_i)_{i\in I}$ on $T_i$ such that when $m(U_i)>0$ we have

$$dm_{|U_i}=(h_0(z)\,d\Leb_{P_i(x)}(z))\,d\nu_i(x),$$
where $h_0$ is defined in Proposition \ref{trouverunefctharmonique} and $P_i(x)$ denotes the plaque of $x$.
\end{defi}

\subsection{Bijective correspondence between Gibbs u-states and $\phi^u$-harmonic measures}
\label{bijcorrespondencesection}

The goal here is to introduce a natural bijective correspondence $p:\G ibbs^u\to\HH ar^{\phi^u}(\F)$ where

\begin{itemize}
\item $\G ibbs^u$ denotes the set of Gibbs u-states for the foliated geodesic flow;
\item $\HH ar^{\phi^u}(\F)$ denotes the set of $\phi^u$-harmonic measures for $\F$.
\end{itemize}

\subsubsection{A toy model}
\label{toy}

The proof follows closely the main line of reasoning of \cite{Al1} where we show how to lift canonically harmonic measures in the sense of Garnett. We propose to give a glimpse of the argument in a toy example. 

In what follows $L$ is a leaf  of $\F$ and  $\tL$ denotes its universal cover. We will show how to lift to $T^1\tL$ a $\phi^u$-harmonic measure of
$\tL$. Lifts of center unstable manifolds to $\tL$ shall be denoted by $\widetilde{W}^{cu}(v)$, and the foliation they define, by $\widetilde{\W}^{cu}$. Manifolds $\widetilde{W}^u(v),\widetilde{W}^{cs}(v),\widetilde{W}^s(v)$ and foliations $\widetilde{\W}^{u},\widetilde{\W}^{cs},\widetilde{\W}^{s}$ are defined analogously.

\paragraph{Trivialization of the center unstable foliation.} There is a identification $\rho:T^1\tL\to\tL\times\tL(\infty)$ sending a vector $v$ 
on the couple $(z,\xi)=(c_v(0),c_v(-\infty))$ where $c_v$ is the geodesic directed by $v$.

The center unstable foliation $\widetilde{\W}^{cu}$ is sent onto the trivial foliation $(\tL\times\{\xi\})_{\xi\in\tL(\infty)}$.
A slice $\tL\times\{\xi\}$ has to be thought as filled by unstable horospheres centered at $\xi$ and by geodesics starting from $\xi$.
 In particular the geodesic flow acts on such a slice.

Moreover, even if a priori $\rho$ is only a homeomorphism, its restriction to any center unstable leaf is a diffeomorphism on its image.

\paragraph{Volume elements.} The unit tangent bundle $T^1\tL$ is endowed with its Sasaki metric. Each center unstable leaf $\widetilde{W}^{cu}(v)$ 
is endowed with the induced Riemannian structure, and with a volume form. Since $\rho$ is a smooth diffeomorphism in restriction to $\widetilde{W}^{cu}(v)$
the corresponding slice $\tL\times\{\xi\}$ may be endowed with the image volume form denoted by $\Leb^{cu}_{\tL\times\{\xi\}}$.

Moreover, any $\tL\times\{\xi\}$ carries a metric that makes it an isometric copy of $\tL$. The corresponding volume form shall be denoted
by $\Leb_{\tL\times\{\xi\}}$.

Note that a priori these two volume forms are different, although equivalent. However, because of the algebraic nature of the hyperbolic space, they coincide when the curvature is constant.

\paragraph{Unrolling argument.} Consider a $\phi^u$-harmonic measure of $\tL$ given by $dm(z)=h(z)d\Leb(z)$, where $h$ is a $\phi^u$-harmonic function.
Write $h(z)=\int_{\tL(\infty)} k^u(o,z;\xi)d\eta_o(\xi)$ where $o\in\tL$ is a base point and 
$\eta_o$ is a finite Borel measure on $\tL(\infty)$.

On the slice $\tL\times\{\xi\}$ consider a measure $dm_{\xi}(z)=k^u(o,z;\xi)d\Leb_{\tL\times\{\xi\}}(z)$.

Now ``unroll'' the $\phi^u$-harmonic measure $m$ i.e. consider the measure $\overline{m}$ on $\tL\times\tL(\infty)$ having the following disintegration in the spaces $\tL\times\{\xi\}$
$$d\overline{m}=(dm_{\xi})\,d\eta_o(\xi).$$

This defines a Borel measure in $\tL\times\tL(\infty)$ which projects down to $m$.

The correspondence $m\mapsto\overline{m}$ is therefore injective, and $\overline{m}$ is called the \emph{canonical lift} of $m$.

\begin{figure}[hbtp]
\centering
\includegraphics[scale=0.6]{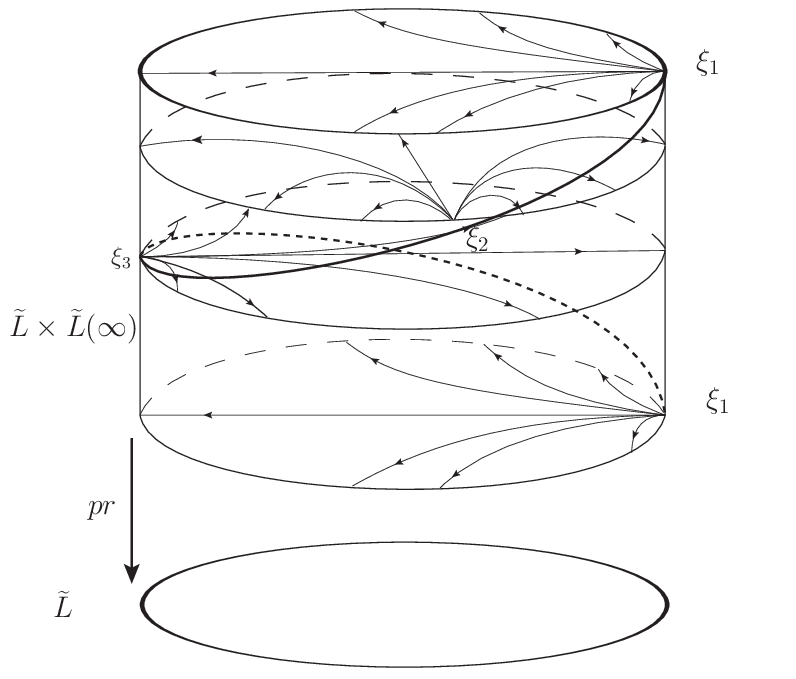}
\caption{The unrolling argument}
\label{unrouille}
\end{figure}

\paragraph{Reparametrization.} A priori the measures $m_{\xi}$ are not invariant under the geodesic flow. Hence let us ``reparametrize'' $\overline{m}$
in the center unstable manifolds. That is consider the family $d\mu_{\xi}=k^u(o,z;\xi)d\Leb^{cu}_{\tL\times\{\xi\}}$.

Using the definition of $k^u$ and the Cocycle Relation given in Remark \ref{remcocyclerelation}
 it is easy to see that $\mu_{\xi}$ is invariant by the geodesic flow. Moreover, since the unstable foliation is smooth in a center-unstable leaf, 
the conditional measures of $\mu_{\xi}$ in horospheres centered at $\xi$ are equivalent to Lebesgue.

Now let $\mu$ be the measure on $\tL\times\tL(\infty)$ having the following disintegration in the spaces $\tL\times\{\xi\}$
$$d\mu=(d\mu_{\xi})\,d\eta_o(\xi).$$

As are all $\mu_{\xi}$, the measure $\mu$ is invariant by the geodesic flow and has Lebesgue disintegration in unstable manifolds. It is a Gibbs $u$-state.
This reparametrization gives an injective correspondence $\overline{m}\mapsto\mu$.

Hence we have an injective map from $\phi^u$-harmonic measures on $\tL$ to Gibbs $u$-states on $T^1\tL$.

\subsubsection{Lifting $\phi^u$-harmonic measures}

The proof of the general case consists in applying the same line of reasoning for all leaves of the foliation simultaneously. To give a detailed proof of this fact was the purpose of \cite{Al1}. Let us review the main steps of the proof.

\paragraph{Lifted atlases.} Let us give some notations  that we are going to use in the following. Some of them have already been given in \S \ref{definitionunittgtbdle}. We work with a good foliated atlas $\A=(U_i,\phi_i)_{i\in I}$ for $\F$ and a complete system of transversals $(T_i)_{i\in I}$. We denote by $\TT$ the associated complete transversal.

We lift $\A$ to $\hM$ via the basepoint projection so as to obtain an atlas $\widehat{\A}$. Refining $\A$ if needed we assume that charts of $\widehat{\A}$ trivialize the center-unstable foliation. A chart of $\widehat{\A}$ then has the form
$$\widehat{U}_i=pr^{-1}(U_i)=\bigcup_{v\in S_i} W^{cu}_{loc}(v),$$
where $U_i$ is a chart of $\A$ and $S_i$ is a transversal of $\W^{cu}$ which is trivially foliated by unit tangent fibers: it reads as $\bigcup_{w\in T_i}T^1_w\F$. Note that the family $(T_i)_{i\in I}$ is also a complete system of transversals for $\hcF$.

\paragraph{Induced measures on a complete transversal.}
 A chart $\widehat{U}_i$ comes naturally with a projection along leaves of $\hcF$ denoted by $\widehat{p}_i:\widehat{U}_i\to T_i$. 
 Let $\mu$ be a probability measure in $\hM$. \emph{The measure on $\TT$ induced by $\mu$} will be denoted by $\widehat{\mu}$. By definition, it is a measure on $\TT$ which in restriction to $T_i$, it is given by the projection of $\mu_{|\widehat{U}_i}$ by $\widehat{p}_i$ that we denote by $\widehat{\mu}_i$.

Similarly, a measure $m$ on $M$ induces a measure $\widehat{m}$ on $\TT$ or, if one prefers, a family of measures $(\widehat{m}_i)_{i\in I}$ on transversals $T_i$.  Recall that $pr:\hM\to M$ denotes the basepoint projection.

\begin{proposition}[Lifts of $\phi^u$-harmonic measure]
\label{liftouille} There exists an map $\sigma: \HH ar^{\phi^u}(\F)\to\Pp(\hM)$, where $\Pp(\hM)$ denotes the space of Borel probability measures on $\hM$, such that the following holds for every $m\in\HH ar^{\phi^u}(\F)$:
\begin{enumerate}
\item $pr_
{\ast}\sigma(m)=m$, in particular $\sigma$ is injective;
\item $m$ and $\sigma(m)$ induce the same measures on $\TT$.
\end{enumerate}
\end{proposition}

\begin{proof}
The idea is to use the unrolling argument for all leaves simultaneously. This technical point has been treated in \cite[Proposition 3.6]{Al1}. It should be noted that by definition of the basepoint projection, the second property implies the first one.
\end{proof}

\subsubsection{Reparametrization and projection of Gibbs $u$-states}
\label{reparouille}
\paragraph{Measures induced by Gibbs $u$-states.} We will need the following proposition that we will state without proof. It relies on the absolute continuity of invariant foliations of $G_t$ and an argument \emph{\`a la Hopf} which can be copied verbatim from the proof of \cite[Proposition 4.6]{Al1} (which is one of the main points of that paper).

\begin{proposition}
\label{mesuresinduitesugibbs}
Let $(M,\F)$ be a closed foliated manifold endowed with a negatively curved leafwise metric. Let $\A$ be a good foliated atlas for $\F$, $\TT$ an associated complete transversal, and $\widehat{\A}$ be the lifted atlas via the basepoint projection. Then

\begin{enumerate}
\item any Gibbs u-state for $G_t$ induces a measure on $\TT$ which is quasi-invariant by the holonomy pseudogroup of $\hcF$;
\item two mutually singular Gibbs $u$-states for $G_t$ induce mutually singular measures on $\TT$. In particular, using the ergodic decomposition of Gibbs $u$-states, two different Gibbs $u$-states for $G_t$ induce different measures on $\TT$.
\end{enumerate}
\end{proposition}

\paragraph{Reparametrization of a Gibbs $u$-state.} Set $\HH=\sigma(\HH ar^{\phi^u}(\F))$ where $\sigma$ is obtained in Proposition \ref{liftouille}.

\begin{proposition}[Reparametrization]
\label{repouille} There exists a bijective correspondence $Rep:\G ibbs^u\to\HH$ such that for every $\mu\in\G ibbs^u$, $\mu$ and $Rep(\mu)$ induce the same measure on the complete transversal $\TT$.
\end{proposition}

\begin{proof}
By definition of Gibbs $u$-states and of canonical lifts of $\phi^u$-harmonic measures, both measures have Lebesgue disintegration along $\W^{cu}$. As in \S \ref{toy} it is possible to identify explicitely the form of the conditional measures in the local center unstable leaves. They have the same density (given by the Gibbs kernel) with respect to two different volume forms (see the discussion of \S \ref{toy}).

Now it is straightforward to generalize to our case the argument given in \S \ref{toy} in order to give the ``reparametrization correspondence" with the desired property (we refer to \cite[Proposition 4.8]{Al1} for the technical details).
\end{proof}

\begin{rem}
\label{Gibbsueq}
A Gibbs $u$-state $\mu$ and its reparametrization $Rep(\mu)$ are always equivalent.
\end{rem}

\begin{proof}
As we explained in the previous paragraph, the reparametrization consists in changing the local form of $\mu$: we replace the Lebesgue measure in center unstable plaques by another volume form, which is equivalent, without changing the transverse component. This does not change the measure class of $\mu$.
\end{proof}

\paragraph{End of the proof of Theorem \ref{mthbijcorresp}.}

Every $\mu\in\G ibbs^u$, induces the same measure as $Rep(\mu)$ in a complete transversal $\TT$ (Proposition \ref{repouille}). The measure $pr_{\ast}Rep(\mu)$ is a $\phi^u$-harmonic measure $m$ inducing the same measure on $\TT$ as $\mu$ (Proposition \ref{liftouille}). Moreover, as a combination of Propositions \ref{liftouille} and \ref{repouille}, it comes easily that $pr_{\ast}Rep:\G ibbs^u\to\HH ar^{\phi^u}(\F)$ is surjective.

To prove that it is injective, it is enough to note that two different Gibbs $u$-states induce different measures on $\TT$. Hence the same property holds for their images by $pr_{\ast}Rep$. This implies in particular that their images by $pr_{\ast} Rep$ are different. The injectivity follows.

\subsection{Ergodic decomposition of $\phi^u$-harmonic measures}
\label{ergdec} 

 In what follows, we shall denote by $p:\G ibbs^u\to\HH ar^{\phi^u}(\F)$ the bijective correspondence constructed above. Its inverse shall be denoted by $s:\HH ar^{\phi^u}(\F)\to\G ibbs^u$.

\subsubsection{Ergodicity}

\paragraph{Ergodic measures.} We first define a natural notion of ergodic $\phi^u$-harmonc measure for $\F$.
\begin{defi}[Ergodic $\phi^u$-harmonic measures]
Let $(M,\F)$ be a closed manifold endowed with a negatively curved leafwise metric. A $\phi^u$-harmonic measure $m$ on $M$ is said to be ergodic if any Borel set saturated by $\F$ is full or null for $m$.
\end{defi}

\begin{rem}
\label{ermergodicity}
Let $\TT$ be a complete transversal for $\F$. A $\phi^u$-harmonic measure has Lebesgue disintegration for $\F$ and hence induces a finite measure $\widehat{m}$, on $\TT$ which is quasi-invariant by the holonomy pseudogroup (see \S \ref{disintegrationfoliations}). Saying that $m$ is ergodic amounts to saying that any Borel subset of $\TT$ that is saturated by the action of the holonomy pseudogroup is full or null for $\widehat{m}$.
\end{rem}

\begin{lemma}
\label{supportsature}
Let $m_1$ and $m_2$ be two mutually singular $\phi^u$-harmonic measures. Then, there exists a Borel set $\X\dans M$ which is saturated by $\F$ and such that $m_1(\X)=1$ and $m_2(\X)=0$.

\end{lemma}

\begin{proof}
We endow the manifold $M$ with a good foliated atlas for $\F$, $\A=(U_i,\phi_i)_{i\in I}$. An associated complete system of transversal is denoted by $(T_i)_{i\in I}$ and the union of these transversals is denoted by $\TT$.

Let $m_1$ and $m_2$ be two mutually singular $\phi^u$-harmonic measures. There exists a Borel set $\X_0\dans M$ such that $m_1(\X_0)=1$ and $m_2(\X_0)=0$. Since the measures $m_1$ and $m_2$ have Lebesgue disintegration for $\F$, they induce quasi-invariant families of measures on $\TT$ denoted by $\widehat{m}_1$ and $\widehat{m}_2$ (see \S  \ref{disintegrationfoliations}).

For $i\in I$ denote by $X_{0,i}$ the projection on $T_i$ of $\X_0\cap U_i$ along the plaques of $U_i$. Define $\widehat{\X}_0=\bigcup_i X_{0,i}$. This is a set which is full for $\widehat{m}_1$ and null for $\widehat{m}_2$. Using Lemma \ref{fullsaturated} we may assume that $\widehat{\X}_0$ is saturated by the holonomy pseudogroup $\Pp$.

Let $\X \dans M$ be the saturation by $\F$ of $\widehat{\X}_0$. Denote by $X_i$ the projection of $\X\cap U_i$ on $T_i$ along the plaques of $U_i$, and by $\widehat{\X} \dans\bigcup_i T_i$ the union of these sets. Since $\widehat{\X}_0$ has been chosen invariant by holonomy we find that $\widehat{\X}=\widehat{\X}_0$. In particular this set is full for $\widehat{m}_1$ and null for $\widehat{m}_2$.

One deduces that $m_1(\X)=1$ and $m_2(\X)=0$. The lemma follows since $\X$ is saturated.
\end{proof}

\begin{proposition}
\label{mesuresergodiquesphiharm}
A $\phi^u$-harmonic measure on $M$ is ergodic if and only if $\mu=s(m)$ is an ergodic Gibbs $u$-state for $G_t$.
\end{proposition}

\begin{proof}
We continue to use the same notations we used since the beginning of this section: $\A=(U_i,\phi_i)_{i\in I}$ is a good foliated atlas for $\F$, that we lift to $\hM$ via the basepoint projection so as to obtain an atlas $\widehat{\A}$. Let $\TT$ be an associated complete transversal. 

Let $m$ be a $\phi^u$-harmonic measure for $\F$ and $\widehat{m}$ be the measure it induces on $\TT$. Let $\mu=s(m)$: this is a Gibbs $u$-state for $G_t$, and the measure $\widehat{\mu}$ it induces on $\TT$ equals $\widehat{m}$ (see Theorem \ref{mthbijcorresp}). It is clear by construction that both $p$ and $s$ are linear.

Let $\widehat{\Y}\dans\TT$ be a Borel set which is saturated by the holonomy pseudogroup of $\F$ and $\Y$ be the Borel set of $\hM$ defined as the union of the leaves of $\hcF$ which meet $\widehat{\Y}$. The set $\Y$ is saturated by $\hcF$, and consequently is invariant by the foliated geodesic flow $G_t$.

As a consequence if we assume that $\mu$ is $G_t$-ergodic, $\Y$ has to be full or null for $\mu$, and hence $\widehat{\Y}$ is full or null for $\widehat{\mu}=\widehat{m}$. This proves the ergodicity of $m$ (see Remark \ref{ermergodicity}).

Now assume that a Gibbs $u$-state $\mu$ is not ergodic. Then there are two singular Gibbs $u$-states, denoted by $\mu_1$ and $\mu_2$, as well as a number $0<\alpha<1$ such that $\mu=\alpha\mu_1+(1-\alpha)\mu_2$ (recall that ergodic components of Gibbs $u$-states are still Gibbs $u$-states).

Since $p:\G ibbs^u\to\HH ar^{\phi^u}(\F)$ is linear, we obtain $m=\alpha m_1+(1-\alpha)m_2$. Since $\mu_1$ and $\mu_2$ are singular, by Proposition \ref{mesuresinduitesugibbs} they induce singular measures on the transversal $\TT$. Since $m_1$ and $m_2$ induce respectively the same measures on $\TT$, they are also singular.

By Lemma \ref{supportsature}, there exists a saturated Borel set $\X\dans M$ such that $m_1(\X)=1$, and $m_2(\X)=0$. Then we find $m(\X)=\alpha\in(0,1)$: $m$ is not ergodic.

We can conclude the proof of the proposition.
\end{proof}

\subsubsection{Ergodic decomposition}
\label{ergounette}

\paragraph{Accessibility property.} The proof of Theorem \ref{decompositionergodiquephiu} lies on an argument \emph{\`a la Hopf}. We need first to define some notions.

\begin{defi}
\label{uregular}
 An element $v\in\hM$ is said to be $u$-regular if the limits
$$\mu^+_v=\lim_{T\to\infty}\frac{1}{T}\int_0^T\delta_{G_t(v)}dt\,\,\,\,\,\,\,\,\,\,\,\,\,\,\textrm{and}\,\,\,\,\,\,\,\,\,\,\,\,\,\,\mu^-_v=\lim_{T\to\infty}\frac{1}{T}\int_0^T\delta_{G_{-t}(v)}dt$$
exist, coincide and if the common value $\mu_v$ is a Gibbs $u$-state. We denote by $\Y^u$ the set of $u$-regular vectors.
\end{defi}

\begin{rem}
\label{muplusmumoins}
Note that if $v$ and $w$ lie on the same unstable (resp. center-stable) manifold then $\mu^-_v=\mu^-_w$ (resp. $\mu^+_v=\mu^+_{w}$) when these measures exist.
\end{rem}

\begin{rem}
\label{ergdecuGibbsstates}
By Birkhoff's ergodic theorem the measures $\mu_v$ defined above are ergodic and by Theorem \ref{reconstruirelesetatsdeugibbs} these are precisely the ergodic components of Gibbs-$u$-states.

In particular the set $\Y^u$ is full for every Gibbs $u$-state. Since Gibbs $u$-states have Lebesgue disintegration for $\W^u$ we deduce that for every $v\in\Y^u$ the intersection $\Y^u\cap W^u(v)$ is $\Leb^u_v$-full in $W^u(v)$.
\end{rem}

\begin{rem}
\label{regholinv}
Let $\B=(V_i,\psi_i)_{i\in I}$ be a good foliated atlas for $\hcF$, and $\TT=\bigcup_{i\in I} T_i$ be an associated complete transversal. The set $\Y^u$ \emph{induces a set} $\widehat{\Y}_0$ on $\TT$. By definition this set is defined as the union of the projections of $\Y^u\cap V_i$ on $T_i$. It is full for the measure induced on $\TT$ by any Gibbs $u$-state (see \S \ref{bijcorrespondencesection} for the definition of induced measure).

Using the construction of Lemma \ref{fullsaturated} we find a set $\widehat{\Y}\dans \widehat{\Y}_0$ which is holonomy invariant. Recall that $\Y^u$ is full for every Gibbs $u$-state, and that every Gibbs $u$-state induces a quasi-invariant measure on $\TT$ (see Proposition \ref{mesuresinduitesugibbs}). It follows from Lemma \ref{fullsaturated} that $\widehat{\Y}$ is still full for the measure induced by every Gibbs $u$-state.

Finally define the set $\Y$ as the intersection of $\Y^u$ with the saturation of $\widehat{\Y}$ for $\hcF$. It is full for every Gibbs $u$-state. and induces $\widehat{\Y}$ on $\TT$.

\emph{Let us  recapitulate}. We can, and will, assume that $\Y^u$ induces on $\TT$ a set which is holonomy invariant.
\end{rem}

\begin{defi}
\label{ucspath}
A $(u,cs)$-path is a piecewise smooth arc obtained by concatenation of smooth arcs included in $\W^u$ and $\W^{cs}$.

A $(u,cs)$-path is said to be regular if it is the concatenation of smooth arcs included in $\W^u$ and $\W^{cs}$ whose extremities are $u$-regular.
\end{defi}

The following property of accessibility is the key point in the proof of Theorem \ref{decompositionergodiquephiu}.

\begin{proposition}[ Accessibility property]
\label{accessibiity}
Let $v,w$ be two $u$-regular vectors belonging to the same leaf $\hL$. Then there exists a regular $(u,cs)$-path joining $v$ and $w$.
\end{proposition}

\begin{proof}
We will consider a good foliated atlas $\B=(V_i,\psi_i)_{i\in I}$ for $\hcF$  which is fine enough so that the union of any pair of intersecting plaques is included in some rectangle $\RR_v$ as defined in \S \ref{foliatedhyperbolicity} (this is possible thanks to the local product structure of the leaves of $\hcF$: see Proposition \ref{lpshypfeuilletee}). Denote by $\TT$ an associated complete transversal.

\emph{Step 1.} We start by proving the result when $v$ and $w$ are close enough.

Consider two $u$-regular vectors $v$ and $w$ that belong to the same rectangle $\RR_{v_0}$. In that case $W^u_{loc}(v)$ intersects $W^{cs}_{loc}(w)$ (see Proposition \ref{lpshypfeuilletee}), but the intersection is not necessarily $u$-regular. Using that Lebesgue-almost every vectors of $W^u_{loc}(v)$ and of $W^u_{loc}(w)$ are $u$-regular (see Remark \ref{ergdecuGibbsstates}) and the absolute continuity of center-stable holonomies, we deduce that there exist $v'\in W^{u}_{loc}(v)$ and $w'\in W^{u}_{loc}(w)$ lying on the same local center-stable manifold and which are $u$-regular. Hence there exists a regular $(u,cs)$-path between $v$ and $w$.

\emph{Step 2.} We show now that two $u$-regular vectors lying in the same leaf can be joined by a piecewise smooth arc obtained by concatenation of smooth arcs $\gamma_k$ tangent to $\hcF$, whose extremities lie inside $\Y^u$, and such that $\gamma_k\dans\RR_{v_k}$ for some $v_k\in\hM$.

Let $v\in\Y^u\cap V_i$ and let $j\in I$ such that $V_j\cap P\neq\vide$, where $P\dans V_i$ is the plaque of $v$. Since $\B$ is a good atlas the plaque $P$ intersects a unique plaque $P'$ of $V_j$. We assumed that $\Y^u$ induces on $\TT$ a holonomy invariant set (see Remark \ref{regholinv}). This implies that there exists $v'\in P'\cap\Y^u$.

Assume now that $v$ and $w$ are $u$-regular and lie in the same leaf $\hL$. Consider a path $\gamma$ tangent to $\hL$ joining them as well as a chain of charts $(V_{i_1},...,V_{i_k})$ covering $\gamma$. Call $(P_{i_1},...,P_{i_k})$ the corresponding chain of plaques. The discussion above shows that each plaque $P_{i_j}$ contains a point $v_{i_j}\in\Y^u$. The concatenation of the geodesic segments $[v_{i_j},v_{i_{j+1}}]$, $j=0,...,k$ (where $v_{i_0}=v$ and $v_{i_{k+1}}=w$) provides the desired path. Indeed since by hypothesis for every $j=1,...,k$ the union $P_{i_j}\cup P_{i_{j+1}}$ is included in some rectangle, the geodesic segment $[v_{i_j},v_{i_{j+1}}]$ must be as well.

\emph{Step 3.} Now we can conclude the proof of the proposition by combining Steps 1 and 2. Let $v,w\in\Y^u$ which belong to the same leaf. Consider the path $\gamma$ that joins $v$ and $w$ obtained by the concatenation of paths $\gamma_k\dans\RR_{v_k}$ constructed in Step 2. By construction the extremities of $\gamma_k$ are $u$-regular and belong to the same rectangle. Step 1 then provides a regular $(u,cs)$-path between them. Concatenating these paths provides the desired regular $(u,cs)$-path.
\end{proof}

\begin{coro}
\label{indepleaf}
Let $v,w\in\Y^u$  {two elements of} the same leaf of $\hcF$. Then $\mu_v=\mu_w$  (recall that $\mu_v$ denotes the common value of $\mu_v^{\pm}$: see Definition \ref{uregular}).
\end{coro}

\begin{proof}
The proof is a combination of Proposition \ref{accessibiity} and Remark \ref{muplusmumoins}.
\end{proof}

\paragraph{Ergodic decomposition.} 
Now, we are ready to prove that all $\phi^u$-harmonic measures may be obtained as convex combinations of ergodic $\phi^u$-harmonic measures.

\begin{proof}[Proof of Theorem \ref{decompositionergodiquephiu}]
The fact that extremal points of the convex set $\HH ar^{\phi^u}$ are the ergodic measures follows from the linearity of the projection $pr_{\ast}$, from Proposition \ref{mesuresergodiquesphiharm}, as well as from the fact that it is true for the convex $\G ibbs^u$ of Gibbs $u$-states for $G_t$.

Consider the set $\Y^u\dans\hM$ of $u$-regular vectors (see Definition \ref{uregular} and Remark \ref{regholinv}) as well as its projection $\X=pr(\Y^u)$. Since $\Y^u$ is full for every Gibbs $u$-state, it must be full for every element of $\HH=Rep(\G ibbs^u)$ (see Remark  \ref{Gibbsueq}). Since $pr_{\ast}:\HH\to\HH ar^{\phi^u}(\F)$ is a bijection, we deduce that $\X$ is full for every $\phi^u$-harmonic measure.

Let $x\in\X$. By Corollary \ref{indepleaf} the measure $\mu_v$ is independent of $v\in T^1_x\F\cap\Y^u$. The measure $m_x=p(\mu_v)$ is well defined independently of $v\in T^1_x\F\cap\Y^u$. By Remark \ref{ergdecuGibbsstates} and Proposition \ref{mesuresergodiquesphiharm} the $\phi^u$-harmonic measures $m_x$ are ergodic. Finally, using one more time Corollary \ref{indepleaf}, we see that if $x,y\in\X$ belong to the same leaf we have $m_x=m_y$.

In order to conclude, the ergodic decomposition remains to be proven. But to obtain it, we only have to apply $p$ to the ergodic decomposition of Gibbs $u$-states and to note that the measures $\mu_v$ are precisely the ergodic components of Gibbs $u$-states.
\end{proof}

\section{Another sufficient condition for the existence of transverse invariant measures}
\label{matsumotosection}

\subsection{Visibility class}
We make use of the notation introduced in Section \ref{sphereinfty}. Consider a leaf $L$ of $\F$. Recall that $\widehat{\pi}_z$ maps a vector $v\in T^1_z\widehat{L}$ on the \emph{future} extremity of the geodesic it directs and that $\widehat{\sigma}_{y,z}=\widehat{\pi}_z^{-1}\circ\widehat{\pi}_y$.

Fix a point $z\in\tL$ and define the measure $\widehat{\nu}_z$ as the projection of the Lebesgue measure of $T^1_z\tL$ by $\widehat{\pi}_z$. Even if the measure a priori depends on $z$, its \emph{class} does not.

\begin{figure}[hbtp]
\centering
\includegraphics[scale=0.8]{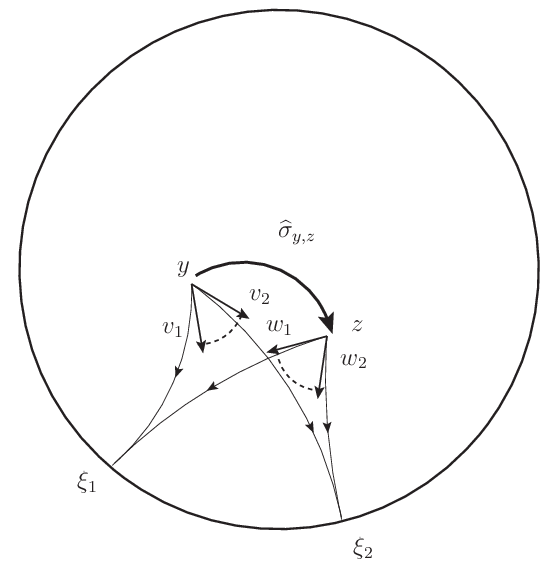}
\caption{The visibility class}
\label{visibilitounette}
\end{figure}

\begin{lemma}
\label{visibility}
The measures $\widehat{\nu}_z$ are pairwise equivalent.
\end{lemma}

\begin{proof}
Let $y,z\in\tL$. Note that by definition $\widehat{\sigma}_{y,z}$ induces holonomy maps along the center-stable foliation between open sets of unit tangent fibers at $y$ and $z$. By Theorem \ref{absolutecontinuity} these maps have to be absolutely continuous: the Lebesgue class is preserved when we apply $\widehat{\sigma}_{y,z}$. This is another way to formulate that the measure classes of $\widehat{\nu}_{y}$ and $\widehat{\nu}_{z}$ are the same.
\end{proof}

\begin{defi}[Visibility class]
\label{visibilityclass}
Let $L$ be a leaf of a closed foliated manifold endowed with a negatively curved leafwise metric. The common class of measures $\widehat{\nu}_y$, $y\in\tL$ is called the visibility class of $\tL(\infty)$.
\end{defi}

\begin{rem}
\label{inversetime}
Consider the flip function $\iota:\hM\to\hM$ defined in Remark \ref{remvisibility}. This involution preserves the Liouville measure in the leaves of $\hcF$ (see \cite[Lemma 1.34]{P}). In particular it preserves the Lebesgue class of unit tangent fibers. Hence if we consider the measure $\nu_z$ obtained by pushing $\Leb_{T^1_z\tL}$ by $\pi_z$ we obtain a measure $\nu_z$ which is still in the visibility class.
\end{rem}

\subsection{Characteristic measure classes and transverse invariant measures}
\label{generthmatsumoto}

\paragraph{Characteristic functions and classes.}
 Let $m$ be a $\phi^u$-harmonic measure for $\F$. We say that a property holds for a $m$-\emph{typical leaf} if it holds in a saturated Borel set full for $m$. We will recall and adopt Matsumoto's terminology (see \cite{Mat}).

By definition $m$ has Lebesgue disintegration with local densities in the plaques which are $\phi^u$-harmonic. Using Ghys' Lemma (more precisely its direct consequence Lemma \ref{extenddensity}) it is possible to extend the density of the conditional measure of $m$ in a plaque $P$, thus obtaining a positive $\phi^u$-harmonic function on the whole leaf $L$ containing $P$.

We refer to this function as the \emph{characteristic function} of the leaf $L$. Note that it is not canonical and depends of the choice of the plaque $P$. Note moreover that the boundedness of this function does not depend of this choice. In fact the characteristic function is defined up to a multiplicative constant (this is the content of \cite[Proposition 2.8]{Mat}).

By definition the lift of the characteristic function of a leaf $L$ to $\tL$ is the integral of the Gibbs kernel against a finite Borel measure $\eta_L$ defined on $\widetilde{L}(\infty)$. The measure class $[\eta_L]$ on $\widetilde{L}(\infty)$ is independent of the initial choice of plaque.

We refer to this measure class as the \emph{characteristic  measure class} on the sphere at infinity the leaf.

\paragraph{Matsumoto's result.} In \cite{Mat} Matsumoto considers closed foliated manifolds $(M,\F)$ with \emph{hyperbolic leaves}. Recall that in that case Gibbs and Poisson kernels coincide and that $\phi^u$-harmonic functions are in fact harmonic (see Remark \ref{kernelhyperbolic} and Proposition \ref{phiunormal}). Thus $\phi^u$-harmonic measures for  $\F$ are harmonic in the sense of Garnett (see Definition \ref{harmonicmeasure}).

In that case the sphere at infinity of the universal cover of a leaf possesses a natural smooth structure and the visibility class defined in Definition \ref{visibilityclass} coincides with the Lebesgue class (all projection maps $\pi_z$ and transition maps $\sigma_{y,z}$ are smooth).

Matsumoto proves (see Theorem \ref{matsumototheorem}) that \emph{when $m$ is a not totally invariant harmonic measure of a closed manifold foliated by hyperbolic leaves}:

\begin{itemize}
\item the characteristic harmonic function of a typical leaf is unbounded;
\item the characteristic measure class on the sphere at infinity of a typical leaf is singular with respect to Lebesgue.
\end{itemize}

We intend here to prove an analogous theorem in the case of $\phi^u$-harmonic measures. It will provide in particular a new and dynamical proof of Matsumoto's result where the properties of Brownian motion are not needed: it relies entirely on the absolute continuity of invariant foliations. We recall that we want to prove that \emph{if $(M,\F)$ is a closed foliated manifold endowed with a negatively curved leafwise metric and} $m$ \emph{is a non totally invariant} $\phi^u$-\emph{harmonic measure, then the characteristic measure class of} $m$-\emph{every leaf is singular with respect to the visibility class}.

\begin{proof}[Proof of Matsumoto's theorem from Theorem \ref{matsumotocbrvar}]
Assume here that all leaves are hyperbolic. The fact that the characteristic  measure class of a typical leaf for a non totally invariant harmonic measure is singular with respect to the Lebesgue measure is an immediate consequence of Theorem \ref{matsumotocbrvar}.

In order to prove that the characteristic function of a typical leaf $L$ is unbounded we will use \emph{Fatou's theorem} of nontangential convergence of harmonic functions: see \cite[Chapter 11]{Rudin}. Let $h(z)=\int_{\partial\Hyp^n}k(o,z;\xi)d\nu(\xi)$ be a positive harmonic function of the hyperbolic space $\Hyp^n$, where $k$ is the Poisson kernel and $\nu$ is a finite Radon measure on the boundary $\partial\Hyp^n$. Let $d\nu=fd\Leb_{\partial\Hyp^n}+d\eta$ be the Lebesgue decomposition of $\nu$, i.e. $f$ is Lebesgue-integrable and $\eta$ is singular with respect to Lebesgue. Fatou's theorem has two parts. Firstly, for Lebesgue almost-every $\xi\in\partial\Hyp^n$ and every geodesic ray $c$ pointing to $\xi$, $\lim_{t\to\infty}h(c(t))=f(\xi)$. Secondly for $\eta$-almost-every $\xi\in\partial\Hyp^n$ and every geodesic ray $c$ pointing to $\xi$, $\lim_{t\to\infty}h(c(t))=\infty$.

Now that we know that the characteristic measure class $[\eta_L]$ is singular with respect to Lebesgue, it follows from the second part of Fatou's theorem that the characteristic function $h$ of a typical leaf is unbounded.
\end{proof}

\subsection{Proof of Theorem \ref{matsumotocbrvar}}

The first step in the proof of Theorem \ref{matsumotocbrvar} is to reduce the theorem to the case of ergodic $\phi^u$-harmonic measures. To do so, we use our ergodic decomposition theorem.

Now, let $m$ be an ergodic $\phi^u$-harmonic measure such that there exists a saturated Borel set $\X$ with positive measure such that the characteristic measure class of every leaf of $\X$ is not singular with respect to the visibility class. By ergodicity, this Borel set is full for $m$. Theorem \ref {matsumotocbrvar} is then a consequence of the following proposition as well as of Theorem \ref{mtheoremsugibbsgeodesique}:

\begin{proposition}
\label{relevesugibbs}
Assume that there exists an ergodic $\phi^u$-harmonic measure $m$ such that the characteristic measure class of $m$-almost every leaf of $\F$ is not singular with respect to the visibility class. Then $\mu=s(m)$ is a Gibbs $s$-state.
\end{proposition}

\subsection{Proof of Proposition \ref{relevesugibbs}}

\paragraph{Disintegration in stable manifolds.}
We will first show that proving Proposition \ref{relevesugibbs} reduces to proving the following one which is a property of Gibbs $u$-states. We use below he notation introduced before Theorem \ref{lambdalemma}: $\UF$ is the foliation of $\hM$ by unit tangent fibers.

\begin{proposition}
\label{propositionprincipalesugibbs}
Let $\mu$ be an ergodic Gibbs $u$-state on $\hM$ for the foliated geodesic flow $G_t$. Then the following alternative holds:
\begin{itemize}
\item either $\mu$ is a Gibbs su-state;
\item or $\mu$ has Lebesgue-singular disintegration for $\UF$.
\end{itemize}
\end{proposition}

\paragraph{Proposition \ref{propositionprincipalesugibbs} implies Proposition \ref{relevesugibbs}.} Let $m$ be a $\phi^u$-harmonic measure and $\mu=s(m)$ be the corresponding Gibbs $u$-state. Let $(\omega_x)_{x\in M}$ be the family of disintegration of $\mu$ in the invariant fibers $T^1_x\F$ and $L$ be a typical leaf. Lift the family $(\omega_x)_{x\in L}$ to $\widetilde{L}$ and denote the lifted family by $(\omega_z)_{z\in\widetilde{L}}$. This is a family of measures in the $T_z^1\widetilde{L}.$ 

\begin{lemma}
\label{remonterlaclassecaracteristique}
Let $L$ be a typical leaf and $\pi_z: T^1_z\widetilde{L}\to\widetilde{L}(\infty)$ be the projection defined in Remark \ref{remvisibility}. Then $\pi_z\,_{\ast}\omega_z$ is in the characteristic measure class.
\end{lemma}

\begin{demo}
Let $U$ be a foliated chart with $m(U)>0$. Let $L$ be a typical leaf crossing $U$ and $P\dans U$ being a connected component of $U\cap L$.

Let $m_P$ be the conditional measure of $m$ in the plaque $P$. By definition it has a density with respect to $\Leb_P$, denoted by $h_P$ which is $\phi^u$-harmonic. Using Lemma \ref{extenddensity} we see that we can extend $h_P$ to a positive function of the whole leaf $L$, whose lift to $\widetilde{L}$ reads
$$\tilde{h}(z)=\int_{\widetilde{L}(\infty)}k(o,z;\xi)d\eta_L(\xi),$$
where $o\in\widetilde{L}$. By definition, $[\eta_L]$ is the characteristic measure class of $\widetilde{L}$.

Let $\overline{m}$ be the canonical lift of $m$: by construction it has the same disintegration in unit tangent fibers as $\mu=s(m)$. Consider its conditional measure $\overline{m}_P$ in $T^1P$ and consider $\overline{m}_{\widetilde{P}}$, the lift to $T^1\widetilde{P}$ where $\widetilde{P}\dans \widetilde{L}$ is some lift of $P$. Consider the identification $\rho:T^1\widetilde{L}\to \widetilde{L}\times \widetilde{L}(\infty)$ sending a vector $v$ on the pair $(c_v(0),c_v(-\infty))$ where $c_v$ is the geodesic directed by $v$.  By the construction of the canonical lift made in \cite{Al1} it comes that $\rho_{\ast}\overline{m}_{\tilde{P}}$ is a multiple of the measure 
$(k^u(o,z,\xi)\,d\eta_L(\xi))d\Leb(z)$ (see also \S \ref{toy}) In particular, the conditional measures of this measure in the $\{z\}\times\widetilde{L}$ are in the visibility class.

By uniqueness of the disintegration, the conditional measures of $\rho_{\ast}\overline{m}_{\tilde{P}}$ in the $\{z\}\times L(\infty)$ are precisely the $\pi_z\,_{\ast}\omega_z$, where $\omega_z$ are the conditional measures of $\overline{m}_{\tilde{P}}$. This proves that all the $\pi_z\,_{\ast}\omega_z$  lie inside the visibility class, concluding the proof of the lemma.

\end{demo}

We can now conclude the proof of Proposition \ref{relevesugibbs}. Let $m$ be an ergodic $\phi^u$-harmonic measure. For $z\in\widetilde{L}$, the projection by $\pi_z^{-1}$ of the visibility class on the unit tangent fiber $T^1_z\widetilde{L}$ is the Lebesgue class. Hence Lemma \ref{remonterlaclassecaracteristique} implies that if the characteristic measure class of a $m$-typical leaf is not singular with respect to the visibility class, then the disintegration of the associated Gibbs $u$-state $\mu=s(m)$ in the leaves of $\UF$ is not singular with respect to the Lebesgue measure. Then by Proposition \ref{propositionprincipalesugibbs} $\mu$ has to be a Gibbs $su$-sate, proving Proposition \ref{relevesugibbs}.\quad \hfill $\square$

\paragraph{Proof of proposition \ref{propositionprincipalesugibbs}.} It is enough to prove that an ergodic Gibbs $u$-state whose disintegration in the leaves of $\UF$ is not singular with respect to Lebesgue is a Gibbs $su$-state.

Let $\mu$ be such an ergodic Gibbs $u$-state. There exists a $G_t$-invariant Borel set $\Y$, full for $\mu$ and such that for every $v\in\Y$ and every continuous function $f:\hM\to\R$ we have:
$$\lim_{T\to\infty}\frac{1}{T}\int_0^Tf\circ G_{-t}(v) dt=\int_{\hM}f\,d\mu.$$

By hypothesis, if we disintegrate $\mu$ in unit tangent fibers, we can find a vector $v_0\in\Y$ based at $x_0\in M$, as well as a Borel set $D\dans T^1_{x_0}\F$ which contains $v_0$ and such that $\Leb_{T^1_{x_0}\F}(D)>0$ and $D\dans\Y$. In order to avoid a heavy notation, in the sequel we will denote simply by $\Leb$ the Lebesgue measure $\Leb_{T^1_{x_0}\F}$.

Since $D\dans\Y$, we obtain that for every $v\in D$ the Birkhoff averages $1/T\int_0^T\delta_{G_{-t}(v)}\,dt$ converge to $\mu$. So by use of dominated convergence, the following family converges to $\mu$:
$$\mu_T=\frac{1}{T}\int_0^T \frac{G_{-t\ast}(\Leb_{|D})}{\Leb(D)}dt.$$

But by Remark \ref{transversediscs}, we know that every accumulation point of $\mu_T$ is a Gibbs $s$-state. As a conclusion, $\mu$ is both a Gibbs $u$-state and a Gibbs $s$-state: this is a Gibbs $su$-state.

We have now proven the desired dichotomy: the proof of Proposition \ref{propositionprincipalesugibbs}, and thus that of Theorem \ref{matsumotocbrvar}, is complete.\quad \hfill $\square$

\section{Appendix.}

Below $(M,\F)$ is a closed foliated manifold with a negatively curved leafwise metric. We recall below the main steps of the proofs of classical results concerning uniformly and partially hyperbolic dynamical systems. As we will see all these arguments apply to our context.

\subsection{Absolute continuity}
\label{absolouille}

Let us treat now the absolute continuity of invariant foliations of the foliated geodesic flow and prove Theorem \ref{absolutecontinuity}.

\begin{figure}[hbtp]
\centering
\includegraphics[scale=0.5]{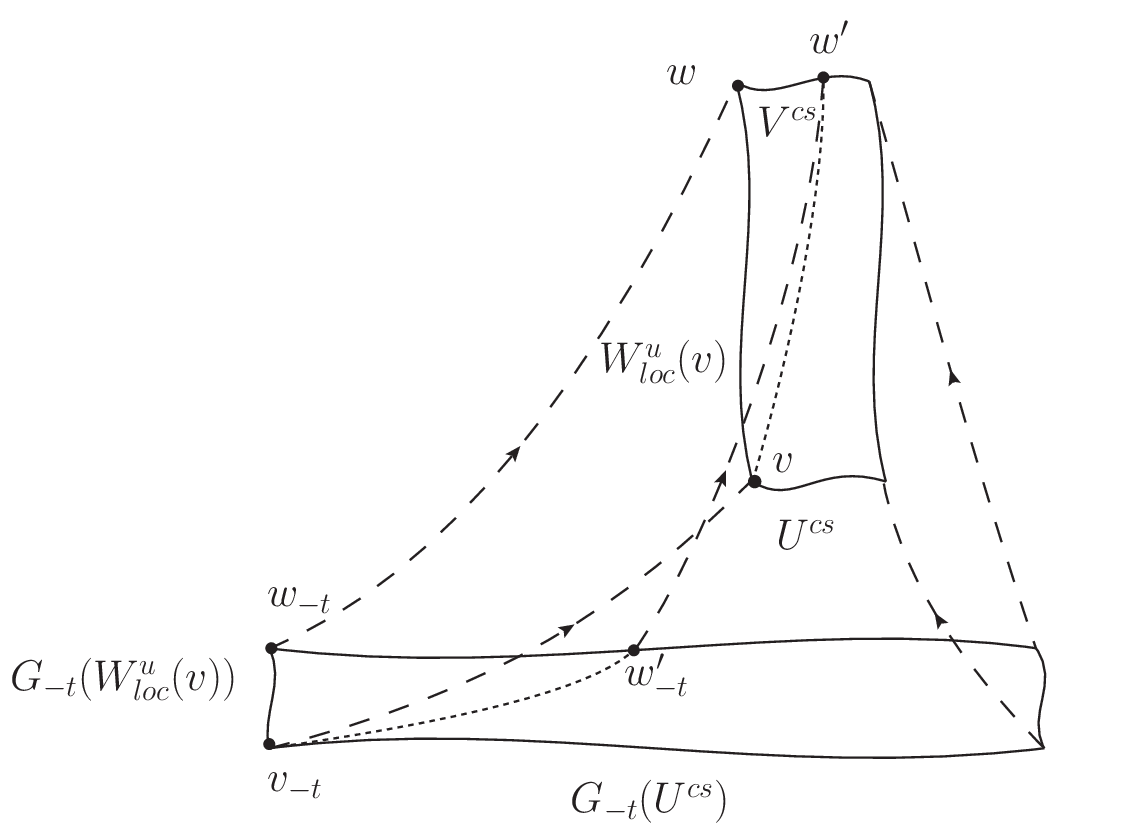}
\caption{Absolute continuity}
\label{figureabscont}
\end{figure}

\paragraph{Step 1.} Let $U^{cs}$ and $V^{cs}$ be two center-stable discs contained in the same leaf $\hL$ such that there exists an unstable holonomy map $\h uvw:U^{cs}\to V^{cs}$. Consider the function of $U^{cs}$ defined by
\begin{equation}
\label{limitabscont}
\Psi_t(v)=\frac{\Jac^{cs}_{-t}(v)}{\Jac^{cs}_{-t}(\h uvw(v))}=\frac{\Jac^{cs}_{-t}(v)}{\Jac^{cs}_{-t}(w)},
\end{equation}
where $v\in U^{cs}$ and $w=\h uvw(v)\in V^{cs}$.

The limit $\Psi_{\infty}$ as $t\to\infty$ is well defined by the distortion control \eqref{distortion2}. We will prove that $\h uvw$ is absolutely continuous and that its jacobian is given by $\Psi_{\infty}$.

\paragraph{Step 2.} We build smooth approximations of unstable holonomies.

Let $\cN^{cs}$ be the subbundle of $T\hcF$ defined as the normal bundle of $\W^{cs}$. We denote by $\cN^{cs}_{\eta}$ the $\eta$-neighbourhood of the zero section, and by $\cN^{cs}_{\eta}(v)$, the fiber of $v\in\hM$. When $K\dans \hM$ denote by $\cN^{cs}_{\eta}(K)$ the union of the fibers $\cN_{\eta}^{cs}(v)$ for $v\in K$. When $\eta>0$ is chosen small enough, the exponential map induces a diffeomorphism from $\cN^{cs}_{\eta}(W^{cs}_{loc}(v))$ onto a neighbourhood of $v$ inside $\hL_v$ for every $v\in\hM$. By compactness and continuity of the rectangles, we can find $\gamma>0$ such that for every $v\in\hM$ the rectangle $\RR_v=[W^u_{2\gamma}(v),W^{cs}_{2\gamma}(v)]$ is well defined and contained inside this neighbourhood.

For every $v\in\hM$ consider the smooth foliation of $\RR_v$ defined as the image by the exponential map of the foliation of $\cN^{cs}_{\eta}(W^{cs}_{loc}(v))$ defined by the fibers $\cN^{cs}_{\eta}(z)$. This foliation is transverse to $\W^{cs}$. Its holonomy provides $\pi_{v\to w}^0:W^{cs}_{\gamma}(v)\to W^{cs}(w)$ which is a diffeomorphism on its image, well defined when $w\in W^u_{\gamma}(v)$, and such that there exists a uniform $C_1>0$ such that when $w\in W^u_{\gamma}(v)$  
\begin{equation}
\label{approx1}
\dist_{\hcF}(\pi^0_{v\to w}(v),w)\leq C_1\dist_{\hcF}(v,w),
\end{equation}

\begin{equation}
\label{approx2}
||D_v\pi^0_{v\to w}-Id||\leq C_1\dist_{\hcF}(v,w).
\end{equation}

Consider $v\in U^{cs}$ and $w=\h uvw(v)$. We will asume that the diameter of $U^{cs}$ is small enough and that $v,w$ are close enough so that $\pi^0_{v\to w}$ is well defined inside $U^{cs}$.

Define $\pi^t_{v\to w}=G_{t}\circ\pi_{G_{-t}(v)\to G_{-t}(w)}^0\circ G_{-t}$. Denote by $w'=\pi_{v\to w}^t(v)\in W^{cs}_{loc}(w)$. In the sequel we shall denote $v_{-t}=G_{-t}(v)$. The elements $w_{-t}$ and $w'_{-t}$ are defined analogously. Note that by definition $w_{-t}'=\pi_{v_{-t},w_{-t}}^0(v_{-t})\in W^{cs}_{loc}(w_{-t})$. By definition of center stable and unstable manifolds and by \eqref{approx1} it follows that there exist $C_2,\lambda>0$ such that for every $s\geq 0$
\begin{equation}
\label{centrouillestablouille}
\dist(G_s(w_{-t}),G_s(w'_{-t}))\leq\dist(w_{-t},w_{-t})\leq C_1\dist(v_{-t},w_{-t})\leq  C_2 e^{-t\lambda}\dist(v,w).
\end{equation}
Using these inequalities we can prove that $\pi^t_{v\to w}$ converges uniformly to $\h uvw$ on $U^{cs}$. Now denote by $\Phi_t(v)$ the Jacobian at $v$ of $\pi_{v\to w}^t$ . We have
\begin{equation}
\label{convjacob}
\frac{\Phi_t(v)}{\Psi_t(v)}=\Jac\,\pi^0_{v_{-t},w_{-t}}(v_{-t})\,\frac{\Jac^{cs}_{-t}(w)}{\Jac^{cs}_{-t}(w')}.
\end{equation}

\paragraph{Step 3.} Finally we must prove the uniform convergence of $\Phi_t$ to $\Phi_{\infty}$. The first factor of the right hand side of \eqref{convjacob} tends to $1$ uniformly: this follows from \eqref{approx1} and \eqref{approx2}. In order to prove that the second one also tends to $1$ we need a second distortion control. Let $n\in\N$. Using the H{\"o}lder continuity of $\log\,\Jac^{cs}_{-1}$ and \eqref{centrouillestablouille} we can prove

\begin{eqnarray*}
\left|\log\frac{\Jac^{cs}_{-n}(w)}{\Jac^{cs}_{-n}(w')}\right|&\leq&\sum_{i=0}^{n-1}|\log\,\Jac^{cs}_{-1}(w_{-i})-\log\,\Jac^{cs}_{-1}(w'_{-i})|\\
&\leq& C_3\sum_{i=0}^{n-1}\dist(w_{-i},w_{-i}')^{\alpha_1}\\
&\leq& C_4ne^{-\alpha_1\lambda n}\To_{n\to\infty} 0.
\end{eqnarray*}
Using that $\Psi_t$ converge uniformly on $U^{cs}$ to $\Psi_{\infty}$ we deduce that $\Phi_t$ converge uniformly to the same limit. 

Now conclude using an abstract result of absolute continuity: take two continuous functions $h$ and $\Psi$ such that there exists a family of smooth functions $\pi_t$ that converge uniformly to $h$, whose Jacobians $\Phi_t$ converge uniformly to $\Psi$. Then $h$ is absolutely continuous and his Jacobian is given by $\Psi$ (see \cite[Theorem 3.3 Chap. III]{M}).

\subsection{Existence of Gibbs $u$-states}
\label{proofugibbs}

We give here the main steps of the proof of the existence of Gibbs $u$-states (i.e. Theorem \ref{reconstruirelesetatsdeugibbs}). This is an adaptation of \cite[Section 11.2.2]{BDV}. See also \cite[Theorem 4]{PS}.

We first treat the case where $D^u$ is a small unstable disc. We want first to prove that accumulation points of
$$\mu_t=\frac{1}{t}\int_0^t G_{s\ast}\left(\frac{\Leb^u_{D^u}}{\Leb^u_{D^u}(D^u)}\right)\,ds$$
are Gibbs-$u$-states.

\paragraph{Distortion control.} We first need an application of the distortion control \eqref{distortion1}: we leave the proof of the next lemma to the reader.

\begin{lemma}
\label{distctrl}
Let $D^u$ be a small unstable disc. Then given a positive number $\Delta>0$ there exists $K_0>1$ such that for every $t\geq 0$, every open set $O\dans G_t(D^u)$ of diameter smaller than $\Delta$ and every Borel set $Z\dans O$ we have
$$\frac{1}{K_0}\frac{\Leb^u_{G_t(D^u)}(Z)}{\Leb^u_{G_t(D^u)}(O)}\leq \frac{G_t\,_{\ast}\Leb^u_{D^u}(Z)}{G_t\,_{\ast}\Leb^u_{D^u}(O)}\leq K_0\frac{\Leb^u_{G_t(D^u)}(Z)}{\Leb^u_{G_t(D^u)}(O)}$$
\end{lemma}

\paragraph{Markovian components.} 
Let us consider $U=P\times T$ a foliated chart for $\W^u$: for every $y\in T$, $P(y)=P\times\{y\}$ is an unstable domain. A connected component of $G_t(D^u)\cap U$ is \emph{Markovian} if it is equal to a plaque $P(y)$, i.e. if it crosses entirely $U$. Otherwise we say that it is \emph{non-Markovian}. Call $D_{t,M}$ the union of Markovian components of $G_t(D^u)\cap U$ and $D_{t,NM}$ that of non-Markovian components.

We assume that $\mu_{t_k}$ converges to $\mu$, that $\mu(U)>0$ and that $\mu(\partial U)=0$ in such a way that $(\mu_{t_k})_{|U}$ converges to $\mu_{|U}$. Write
$$(\mu_{t_k})_{|U}=\mu_{t_k,M}+\mu_{t_k,NM}$$
where $\mu_{t_k,M}$ and $\mu_{t_k,NM}$ denote respectively the restrictions of $\mu_{t_k}$ to $D_{t_k,M}$ and $D_{t_k,NM}$.

\begin{lemma}
\label{Markov}
We have $\mu_{t_k,M}\to\mu_{|U}$ and $\mu_{t_k,NM}\to 0$ as $k\to\infty$.
\end{lemma}

\begin{proof}
Using that $G_{-t}$ contracts uniformly the unstable foliation it comes that $G_{-t_k}(D_{t_k,NM})$ are neighbourhoods of $\partial D^u$ whose diameters tend to zero. The lemma follows.
\end{proof}

\paragraph{Step 1.} We are ready to prove the first item of Theorem \ref{reconstruirelesetatsdeugibbs} when $D^u$ is an unstable disc.

\begin{lemma}
\label{distlemma}
The measures $\mu_{t_k,M}$ have Lebesgue disintegration along $\W^u$ with densities which are uniformly log-bounded.
\end{lemma}

\begin{proof}
The conditional measure of $G_{t}\,_{\ast}(\Leb^u_{D^u}/\Leb^u_{D^u}(D^u))$ in a markovian component $P(y)$ of $U\cap G_t(D^u)$ is the probability measure associating to a Borel set $A\dans P(y)$ the number 
$$\frac{\Leb^u_{D^u}(G_{-t}(A))}{\Leb^u_{D^u}(G_{-t}(P(y))).}$$
By Lemma \ref{distctrl} this measure has a uniformly log-bounded density with respect to Lebesgue.
\end{proof}

Using Lemmas \ref{Markov} and \ref{distlemma} we deduce that $\mu$ as well has Lebesgue disintegration with densities which are uniformly log-bounded (we can prove the compactness of the set of probability measures with Lebesgue density along the leaves of $\W^u$ with local densities uniformly bounded from above and below by two  given constants). We can say more: using Lemma \ref{distlemma} it is possible to prove that the local densities satisfy \eqref{relationdensiteslocales}. By construction it is $G_t$-invariant so it is a Gibbs $u$-state.

\paragraph{Step 2.} We now treat the case where $D^u$ is a Borel set of positive Lebesgue measure. This is a Lebesgue density argument. If $\delta>0$ is small it is possible to find disjoint unstable discs $D_1,...,D_n$ such that
\begin{itemize}
\item $\Leb_{D^u}^u(D_i\moins D^u)\leq\delta\Leb^u_{D^u}(D_i);$
\item $\Leb_{D^u}^u(D^u\moins\bigcup D_i)\leq\delta\Leb^u_{D^u}(D^u)$.
\end{itemize}
Write

$$\frac{G_t\,_{\ast}\Leb_{D^u}^u}{\Leb_{D^u}^u(D^u)}=\sum_{i=1}^n\frac{\Leb_{D^u}^u(D_i)}{\Leb_{D^u}^u(D^u)}\,\frac{G_t\,_{\ast}\Leb_{D_i}^u}{\Leb_{D^u}^u(D_i)}
 +\frac{G_t\,_{\ast}\Leb_{D^u\moins\bigcup D_i}^u}{\Leb_{D^u}^u(D^u)}-\sum_{i=1}^n  \frac{G_t\,_{\ast}\Leb_{D_i\moins D^u}^u}{\Leb_{D^u}^u(D^u)}.$$ Notice that the two last terms of the right hand side have total mass less than $\delta$, for every $t$. Using Step 1 we see that every accumulation point of Cesaro averages of the first term of the right hand side tends to a Gibbs $u$-state.
 
 Finally we deduce that every accumulation point of the left hand side differs by a measure of total mass less than $\delta$ from a Gibbs $u$-state with uniformly log-bounded local densities. Since the latter property is closed, and letting $\delta$ tend to zero, we can conclude Step 2.
 
 \paragraph{Step 3.} We now have to prove that ergodic components of Gibbs $u$-states are still Gibbs $u$-states. Consider the set $\RR$ of regular points that is to say the set of elements $v\in\hM$ whose past and future Birkhoff averages exist and coincide (i.e. $\mu^+_v=\mu_v^-=\mu_v$ with notations of Definition \ref{uregular}). All measures $\mu_v$ are ergodic.
 
A Gibbs $u$-state $\mu$ gives total measure to $\RR$ and  has Lebesgue disintegration in the unstable manifolds. Hence $\mu$-almost every $v\in\RR$ lies inside an unstable plaque which intersects $\RR$ on a set positive Lebesgue measure. Using Step 2 and the fact that $\mu_v^-$ does not depend on the choice of $v$ in a given unstable manifold, one gets that for $\mu$-almost every $v\in\RR$, $\mu_v$ is an ergodic Gibbs $u$-state.

\paragraph{Step 4.} The case of smooth discs $D$ tangent to $\hcF$ an transverse to $\W^{cs}$ (see Remark \ref{transversediscs}) follows from the fact that, by the inclination lemma, their iterates approach leaves of $\W^{cu}$ in the smooth topology.

The case of Borel subsets of such discs of positive Lebesgue measure follows from the same Lebesgue density argument as sketched in Step 2.

\vspace{10pt}
\paragraph{Acknowledgments.} This paper corresponds to the fifth chapter of my PhD thesis \cite{Al3}. I am deeply grateful to my supervisor Christian Bonatti for his advices, insights and great patience. It is also a pleasure to thank Lucia Duarte who forced me to take some time to draw the pictures, and who patiently showed me how to do. Finally I am thankful to the anonymous referee for his/her advices which improved the readability of the article.

\begin{flushleft}
{\scshape S\'ebastien Alvarez}\\
Instituto Nacional de Matem\'atica Pura e Aplicada (IMPA)\\
Estrada Dona Castorina 110, Rio de Janeiro, 22460-320, Brasil\\
email: salvarez@impa.br
\end{flushleft}


\begin{thebibliography}{10}

\bibitem{Al3}
S.~Alvarez.
\newblock Mesures de {G}ibbs et mesures harmoniques pour les feuilletages aux
  feuilles courb{\'e}es n{\'e}gativement.
\newblock {\em Th{\`e}se de l'{U}niversit{\'e} de {B}ourgogne}, 2013.

\bibitem{Al2}
S.~Alvarez.
\newblock Gibbs measures for foliated bundles with negatively curved leaves.
\newblock {\em Accepted in Ergodic Theory Dynam. Systems}, 2016.

\bibitem{Al1}
S.~Alvarez.
\newblock Harmonic measures and the foliated geodesic flow for foliations with
  negatively curved leaves.
\newblock {\em Ergodic Theory Dynam. Systems}, 36(2):355--374, 2016.

\bibitem{AlLe}
S.~Alvarez and P.~Lessa.
\newblock The {T}eichm{\"u}ller space of the {H}irsch foliation.
\newblock {\em Preprint, arXiv:1507.01531}, 2015.

\bibitem{AlY}
S.~Alvarez and J.~Yang.
\newblock Physical measures for the geodesic flow tangent to a transversally
  conformal foliation.
\newblock {\em Preprint, arXiv:1512.01842}, 2015.

\bibitem{AS}
M.~Anderson and R.~Schoen.
\newblock Positive harmonic functions on complete manifolds of negative
  curvature.
\newblock {\em Ann. of Math. (2)}, 121(3):429--461, 1985.

\bibitem{BMar}
Y.~Bakhtin and M.~Mart{\'i}nez.
\newblock A characterization of harmonic measures on laminations by hyperbolic
  {R}iemann surfaces.
\newblock {\em Ann. Inst. Henri Poincar\'e Probab. Stat.}, 44(6):1078--1089,
  2008.

\bibitem{Ba}
W.~Ballmann.
\newblock {\em Lectures on spaces of nonpositive curvature}, volume~25 of {\em
  DMV Seminar}.
\newblock Birkh\"auser Verlag, Basel, 1995.
\newblock With an appendix by Misha Brin.

\bibitem{BDV}
C.~Bonatti, L.~D{\'i}az, and M.~Viana.
\newblock {\em Dynamics beyond uniform hyperbolicity}, volume 102 of {\em
  Encyclopaedia of Mathematical Sciences}.
\newblock Springer-Verlag, Berlin, 2005.
\newblock A global geometric and probabilistic perspective, Mathematical
  Physics, III.

\bibitem{BG}
C.~Bonatti and X.~G{\'o}mez-Mont.
\newblock Sur le comportement statistique des feuilles de certains feuilletages
  holomorphes.
\newblock In {\em Essays on geometry and related topics, {V}ol. 1, 2},
  volume~38 of {\em Monogr. Enseign. Math.}, pages 15--41. Enseignement Math.,
  Geneva, 2001.

\bibitem{BGM}
C.~Bonatti, X.~G{\'o}mez-Mont, and M.~Mart{\'i}nez.
\newblock Foliated hyperbolicity and foliations with hyperbolic leaves.
\newblock {\em Preprint, arXiv:1510.05026}, 2015.

\bibitem{BGV}
C.~Bonatti, X.~G{\'o}mez-Mont, and M.~Viana.
\newblock G\'en\'ericit\'e d'exposants de {L}yapunov non-nuls pour des produits
  d\'eterministes de matrices.
\newblock {\em Ann. Inst. H. Poincar\'e Anal. Non Lin\'eaire}, 20(4):579--624,
  2003.

\bibitem{BV}
C.~Bonatti and M.~Viana.
\newblock S{RB} measures for partially hyperbolic systems whose central
  direction is mostly contracting.
\newblock {\em Israel J. Math.}, 115:157--193, 2000.

\bibitem{BM}
R.~Bowen and B.~Marcus.
\newblock Unique ergodicity for horocycle foliations.
\newblock {\em Israel J. Math.}, 26(1):43--67, 1977.

\bibitem{BR}
R.~Bowen and D.~Ruelle.
\newblock The ergodic theory of {A}xiom {A} flows.
\newblock {\em Invent. Math.}, 29(3):181--202, 1975.

\bibitem{C}
A.~Candel.
\newblock Uniformization of surface laminations.
\newblock {\em Ann. Sci. \'Ecole Norm. Sup. (4)}, 26(4):489--516, 1993.

\bibitem{DK}
B.~Deroin and V.~Kleptsyn.
\newblock Random conformal dynamical systems.
\newblock {\em Geom. Funct. Anal.}, 17(4):1043--1105, 2007.

\bibitem{dC}
M.~do~Carmo.
\newblock {\em Riemannian geometry}.
\newblock Mathematics: Theory \& Applications. Birkh\"auser Boston, Inc.,
  Boston, MA, 1992.
\newblock Translated from the second Portuguese edition by Francis Flaherty.

\bibitem{Gar}
L.~Garnett.
\newblock Foliations, the ergodic theorem and {B}rownian motion.
\newblock {\em J. Funct. Anal.}, 51(3):285--311, 1983.

\bibitem{Gh}
{\'E}.~Ghys.
\newblock Topologie des feuilles g\'en\'eriques.
\newblock {\em Ann. of Math. (2)}, 141(2):387--422, 1995.

\bibitem{GLW}
{\'E}.~Ghys, R.~Langevin, and P.~Walczak.
\newblock Entropie g\'eom\'etrique des feuilletages.
\newblock {\em Acta Math.}, 160(1-2):105--142, 1988.

\bibitem{KL}
V.~Kaimanovich and M.~Lyubich.
\newblock Conformal and harmonic measures on laminations associated with
  rational maps.
\newblock {\em Mem. Amer. Math. Soc.}, 173(820):vi+119, 2005.

\bibitem{M}
R.~Ma{\~n}{\'e}.
\newblock {\em Ergodic theory and differentiable dynamics}, volume~8 of {\em
  Ergebnisse der Mathematik und ihrer Grenzgebiete (3) [Results in Mathematics
  and Related Areas (3)]}.
\newblock Springer-Verlag, Berlin, 1987.
\newblock Translated from the Portuguese by Silvio Levy.

\bibitem{Ma}
M.~Mart{\'i}nez.
\newblock Measures on hyperbolic surface laminations.
\newblock {\em Ergodic Theory Dynam. Systems}, 26(3):847--867, 2006.

\bibitem{Mat}
S.~Matsumoto.
\newblock The dichotomy of harmonic measures of compact hyperbolic laminations.
\newblock {\em Tohoku Math. J. (2)}, 64(4):569--592, 2012.

\bibitem{Matti}
P.~Mattila.
\newblock {\em Geometry of sets and measures in {E}uclidean spaces}, volume~44
  of {\em Cambridge Studies in Advanced Mathematics}.
\newblock Cambridge University Press, Cambridge, 1995.
\newblock Fractals and rectifiability.

\bibitem{P}
G.~Paternain.
\newblock {\em Geodesic flows}, volume 180 of {\em Progress in Mathematics}.
\newblock Birkh\"auser Boston, Inc., Boston, MA, 1999.

\bibitem{PS}
Ya.~B. Pesin and Ya.~G. Sina{\u\i}.
\newblock Gibbs measures for partially hyperbolic attractors.
\newblock {\em Ergodic Theory Dynam. Systems}, 2(3-4):417--438 (1983), 1982.

\bibitem{Pl}
J.~F. Plante.
\newblock Foliations with measure preserving holonomy.
\newblock {\em Ann. of Math. (2)}, 102(2):327--361, 1975.

\bibitem{PuSu}
C.~Pugh and M.~Shub.
\newblock The {$\Omega $}-stability theorem for flows.
\newblock {\em Invent. Math.}, 11:150--158, 1970.

\bibitem{PSW}
C.~Pugh, M.~Shub, and A.~Wilkinson.
\newblock H\"older foliations.
\newblock {\em Duke Math. J.}, 86(3):517--546, 1997.

\bibitem{Rok}
V.~A. Rokhlin.
\newblock On the fundamental ideas of measure theory.
\newblock {\em Amer. Math. Soc. Translation}, 1952(71):55, 1952.

\bibitem{Rudin}
W.~Rudin.
\newblock {\em Real and complex analysis}.
\newblock McGraw-Hill Book Co., New York, third edition, 1987.

\bibitem{Sh}
M.~Shub.
\newblock {\em Global stability of dynamical systems}.
\newblock Springer-Verlag, New York, 1987.
\newblock With the collaboration of Albert Fathi and R{\'e}mi Langevin,
  Translated from the French by Joseph Christy.

\bibitem{W}
P.~Walczak.
\newblock Dynamics of the geodesic flow of a foliation.
\newblock {\em Ergodic Theory Dynam. Systems}, 8(4):637--650, 1988.



\end{thebibliography}
\end{document}